\documentclass[12pt]{article}%
\usepackage[applemac]{inputenc}
\usepackage{amsmath, amsfonts, amssymb, amsthm}
\usepackage{graphicx}
\usepackage{algorithm}
\usepackage{subcaption}
\usepackage{geometry}
\usepackage{amsthm}
\usepackage{algpseudocode}
\usepackage{hyperref}
\usepackage{amsfonts}
\usepackage[applemac]{inputenc}
\usepackage{color}
\usepackage{color, colortbl}
\usepackage{amsmath}
\usepackage{amssymb}
\usepackage{enumitem}
\usepackage{bbm}
\usepackage{graphicx}
\usepackage{tikz}
\usepackage{float}
\usepackage{geometry}
\usetikzlibrary{arrows}
\usepackage{amsfonts}
\usepackage[applemac]{inputenc}
\usepackage{color}
\usepackage{amsmath}
\usepackage{amssymb}
\usepackage{graphicx}
\usepackage{tikz}
\usepackage{bbm}
\usepackage{mathtools}
\usepackage{amssymb}
\usepackage{comment}
\usepackage{graphicx}
\usepackage{gensymb}
\usepackage{amsthm}
\usepackage{adjustbox}
\usepackage{float}
\usepackage{algorithm}
\usepackage{algorithmicx}
\usepackage{algpseudocode}

\newtheorem{theorem}{Theorem}[section]
\newtheorem{lemma}[theorem]{Lemma}

\newtheorem{example}[theorem]{Example}
\newtheorem{assumption}[theorem]{Assumption}

\newtheorem{remark}[theorem]{Remark}
 \usepackage{color}
\usepackage{color, colortbl}

\newcommand{\R}{\mathbb{R}}
\newcommand{\E}{\mathbb{E}}

\newcommand{\bR}{\mathbb{R}}

\begin{document}

\title{Deep BSVIEs Parametrization and Learning-Based Applications}\author{Nacira Agram$^{1}$ \and Giulia Pucci$^{1}$}
\date{\today}
\maketitle

\footnotetext[1]{Department of Mathematics, KTH Royal Institute of Technology, 100 44 Stockholm, Sweden. Email: nacira@kth.se, pucci@kth.se. Supported by the Swedish Research Council (VR 2020-04697).}

\begin{abstract}
We study the numerical approximation of backward stochastic Volterra integral equations (BSVIEs) and their reflected extensions, which naturally arise in problems with time inconsistency, path dependent preferences, and recursive utilities with memory. These equations generalize classical BSDEs by involving two dimensional time structures and more intricate dependencies.

We begin by developing a well posedness and measurability framework for BSVIEs in product probability spaces. Our approach relies on a representation of the solution as a parametrized family of backward stochastic equations indexed by the initial time, and draws on results of Stricker and Yor to ensure that the two parameter solution is well defined in a joint measurable sense.

We then introduce a discrete time learning scheme based on a recursive backward representation of the BSVIE, combining the discretization of Hamaguchi and Taguchi with deep neural networks. A detailed convergence analysis is provided, generalizing the framework of deep BSDE solvers to the two dimensional BSVIE setting. Finally, we extend the solver to reflected BSVIEs, motivated by applications in delayed recursive utility with lower constraints.
\end{abstract}


\section{Introduction}

Backward stochastic Volterra integral equations (BSVIEs) extend the classical theory of backward stochastic differential equations (BSDEs) by introducing two time variables. This extension enables the modeling of systems with memory, path-dependence, and time inconsistent preferences. Originally introduced by Yong \cite{yong2006backward}, BSVIEs have become central tools in the study of recursive utilities, non-Markovian risk measures, and dynamic systems with delay effects; see also \cite{yong2007continuous, agram2019dynamic}.

A typical Type-I BSVIE is given by
\[
Y(t) = g(t, X(t), X(T)) + \int_t^T f(t, s, X(t), X(s), Y(s), Z(t,s))\, ds - \int_t^T Z(t,s)\, dB(s),
\]
where \( X \) solves a forward SDE and \( g, f \) are the terminal condition and generator, respectively. In contrast to BSDEs, the control process \( Z(t,s) \) depends on both \( t \) and \( s \), and the adaptedness of \( Y(t) \) may fail. This temporal asymmetry breaks the backward flow and prevents the application of classical PDE methods.

A key analytical difficulty in solving BSVIEs lies in the regularity properties of \( Z(t,s) \). The It\^o-Ventzel formula shows that unless \( Z(t,s) \) is sufficiently smooth in the "frozen" time variable \( t \), the stochastic integral
\[
Y(t) = \int_0^t Z(t,s) \, dB(s)
\]
is itself non differentiable in \( t \). This creates serious challenges for both analysis and numerical approximation. A particular case in which the integral \(Y(t)\) becomes a semimartingale despite the two time structure has been studied by Agram~\cite{agram2019dynamic}, using an auxiliary SDE to regularize the dependence on \(t\).  To illustrate this, we provide a few canonical examples. 

First, consider a piecewise smooth kernel: 
\[
Z(t,s) = 
\begin{cases} 
t^2 \sin(s), & t < 1, \\
t \log(s+1), & t \geq 1.
\end{cases}
\]
Here, \( Z \) is continuous but not differentiable at \( t = 1 \), and this irregularity is inherited by \( Y(t) \).

Next, consider the singular kernel \( Z(t,s) = |t-s|^\alpha \), with \( 0 < \alpha < \frac{1}{2} \). The derivative \( \partial_t Z(t,s) \) does not exist due to the singularity at \( t = s \), and the integral \( Y(t) = \int_0^t |t-s|^\alpha\, dB(s) \) exhibits the same lack of regularity.

Finally, take the oscillatory kernel \( Z(t,s) = \sin\left( \frac{1}{t - s} \right) \mathbf{1}_{\{s < t\}} \). The violent oscillations near the diagonal make \( Y(t) \) highly irregular.

\begin{figure}[H]
\centering
\includegraphics[width=0.85\textwidth]{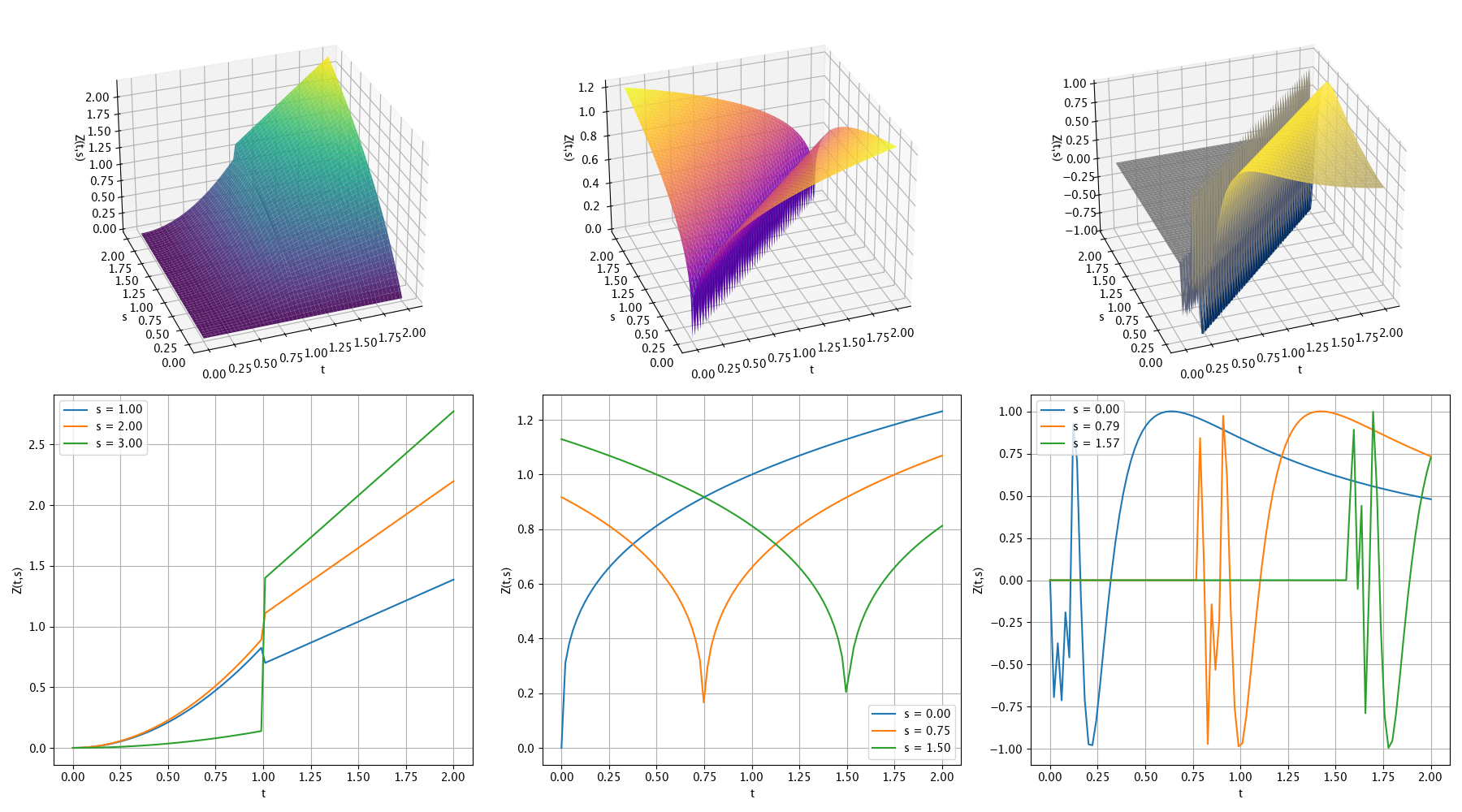}
\caption{Surface plots and slices of the kernel functions \( Z(t,s) \). 
Left: Piecewise smooth kernel; 
Middle: Singular power law kernel \( |t-s|^\alpha \); 
Right: Oscillatory kernel \( \sin(1/(t-s)) \).}
\label{fig:kernel_examples}
\end{figure}

These examples demonstrate that classical differentiability based approaches are often inadequate for BSVIEs. As a remedy, one prevailing strategy in the literature is to interpret BSVIEs as families of BSDEs parametrized by the initial time \( t \). This idea is central in the work of Yong  \cite{yong2006bsde, yong2008well}, where existence and uniqueness results are established via Picard iteration on the family of BSDEs indexed by \( t \). This representation is not only useful analytically, it also enables algorithmic implementation by leveraging BSDE methods at each discretized time step.

A major theoretical development was provided by Wang and Yong \cite{wang2019backward}, who derived a nonlinear Feynma-Kac representation for Type-I BSVIEs. They showed that the solution can be linked to a path dependent PDE defined on the time simplex \(\Delta_T := \{(t,s) \in [0,T]^2 : t \le s\}\), satisfying:
\[
Y(s) = u(s,s,X(s),X(s)), \qquad Z(t,s) = \nabla_x u(t,s,X(t),X(s)) \cdot \sigma(s, X(s)),
\]
where \( u(t,s,x,\xi) \) solves a nonlocal PDE with terminal condition at \( s = T \). While this insight is conceptually powerful, the resulting PDE is generally intractable, particularly in high dimensions or irregular settings.

On the numerical side, existing methods for BSVIEs are still limited. Hamaguchi and Taguchi \cite{hamaguchi2023approximations}  developed a dynamic programming scheme tailored specifically for Type-II BSVIEs,exploiting their two-time structure via a recursive time grid.

Deep learning methods have shown great potential for solving BSDEs. The DeepBSDE framework of Han, Jentzen, and E \cite{han2018solving}, later extended by Hur\'e, Pham, and Warin \cite{hure2020deep}, uses neural networks to approximate the solution pair \( (Y,Z) \), trained via stochastic optimization to minimize a loss function derived from the BSDE dynamics. These methods are flexible, mesh free, and in principle dimension independent, making them particularly well suited to high dimensional stochastic control and finance applications. More recently, Gnoatto, Trillos, and Andersson \cite{andersson2025deep} proposed a forward in time neural solver for BSVIEs, where a single stage network is trained to approximate the solution processes.

In this paper, we develop a new deep learning algorithm tailored to Type-I BSVIEs. Our method builds on the BSDE reformulation by discretizing the two time integral equation and applying a backward recursive learning scheme. The architecture is designed to exploit the parametric dependence on the first time variable \( t \), and jointly learns the fields \( Y(t) \) and \( Z(t,s) \) by fitting conditional expectations at each stage. This structure respects the temporal causality and facilitates convergence analysis through discrete martingale projections.

Our contributions are threefold. First, we construct a discrete time neural scheme specifically adapted to the BSVIE structure, including its reflected variant, by employing backward dynamic programming over a two time grid. Second, we establish a convergence theorem that couples the time discretization error with the neural approximation error, thereby extending the theoretical analysis of \cite{hure2020deep} to the BSVIE setting. Third, we rigorously ensure the measurability of the stochastic integral \( \int_0^t Z(t,s)\, dB(s) \) in the extended product space \([0,T]^2 \times \Omega\), using a result of Stricker and Yor \cite{stricker1978calcul}. This guarantees the well posedness of the learned solution and supports the validity of the loss function in our learning algorithm.\vspace{1em}

The paper is organized as follows. Section~\ref{sec:prelim} introduces the model setup and defines the relevant function spaces. Section~\ref{sec:measurability} presents the measurability results in the product probability space. Section~\ref{sec:scheme} details the discretized solver and neural architecture. Section~\ref{sec:convergence} provides a complete convergence analysis. Section~\ref{sec:numerics} illustrates the effectiveness of our method through numerical experiments. Finally, Section~\ref{sec:reflected} extends the algorithm to handle reflected BSVIEs (RBSVIEs).

\section{Analytical Framework and Structural Properties}
\label{sec:prelim}
In this section, we introduce the structural and analytical ingredients underlying our study. We focus exclusively on Type-I BSVIEs, a class of equations where the generator depends on a frozen time parameter \( t \) and the integration variable \( s \), but the unknown process \( Y \) only appears under the integral in the form \( Y(s) \), and not as \( Y(t) \). This restriction simplifies the analysis and aligns with recent developments in numerical and learning-based methods. Throughout the remainder of the paper, the term BSVIE will always refer to Type-I equations unless explicitly stated otherwise.
\subsection{Spaces for BSVIE Solutions}

We define the following function spaces for $\mathbb{R}$-valued adapted processes:

\begin{itemize}
    \item $\mathcal{S}^2$ is the space of $\mathbb{F}$-adapted processes $(Y(t))_{t \in [0,T]}$ such that
    \[
    \|Y\|^2_{\mathcal{S}^2} \coloneqq \mathbb{E} \left[\sup_{t \in [0,T]} |Y(t)|^2 \right] < \infty.
    \]
    
    \item $\mathcal{H}^2$ is the space of $\mathbb{F}$-adapted processes $(v(t))_{t \in [0,T]}$ satisfying
    \[
    \|v\|^2_{\mathcal{H}^2} \coloneqq \mathbb{E} \left[\int_0^T |v(s)|^2 \, ds \right] < \infty.
    \]

    \item ${L}^2$ is the space of jointly measurable processes $(Z(t,s))_{(t,s) \in [0,T]^2}$ such that for almost every $t$, the map $s \mapsto Z(t,s)$ is $\mathbb{F}$-adapted and
    \[
    \|Z\|^2_{{L}^2} \coloneqq \mathbb{E} \left[\int_0^T \int_t^T |Z(t,s)|^2 \, ds \, dt \right] < \infty.
    \]

    \item $\mathcal{K}^2$ is the space of processes $K$ on $[0,T]^2 \times \Omega$ such that $u \mapsto K(t,u)$ is $\mathbb{F}$-adapted, continuous, and non-decreasing with $K(t,0) = 0$, and $t \mapsto K(t,T)$ lies in $\mathcal{H}^2$.
\end{itemize}

All spaces above are Hilbert spaces under their respective norms.

\subsection{Volterra BSVIE Setup}

Let $(\Omega, \mathcal{F}, \mathbb{F}, \mathbb{P})$ be a filtered probability space carrying a $d$-dimensional Brownian motion $B = (B(t))_{t \in [0,T]}$. The filtration $\mathbb{F} = (\mathcal{F}_t)_{t \in [0,T]}$ is the usual augmentation. We denote the time simplex:
\[
\Delta_T \coloneqq \{ (t,s) \in [0,T]^2 : 0 \le t \le s \le T \}.
\]

We study a decoupled forward-backward system. The forward component is the classical SDE:
\begin{equation}\label{eq:SDE}
X(t) = x + \int_0^t b(s, X(s))\, ds + \int_0^t \sigma(s, X(s))\, dB(s), \quad t \in [0,T],
\end{equation}
with initial condition $x \in \mathbb{R}^n$.

Given $X$, we consider the BSVIE:
\begin{equation}\label{eq:BSVIE}
Y(t) = g(t, X(t), X(T)) + \int_t^T f(t, s, X(t), X(s), Y(s), Z(t,s)) \, ds - \int_t^T Z(t,s) \, dB(s), \quad t \in [0,T].
\end{equation}

\begin{assumption}
\label{ass}
\begin{enumerate}[label=(\alph*)] 
\item The coefficients
\[
b: [0,T] \times \mathbb{R}^n \to \mathbb{R}^n, \quad \sigma: [0,T] \times \mathbb{R}^n \to \mathbb{R}^{n \times d}
\]
are continuous and satisfy:
\begin{align*}
|b(s,x) - b(s,x')| + \|\sigma(s,x) - \sigma(s,x')\| &\le L |x - x'|, \\
|b(s,x)| + \|\sigma(s,x)\| &\le L (1 + |x|).
\end{align*}

\item The maps
\[
g: [0,T] \times \mathbb{R}^n \times \mathbb{R}^n \to \mathbb{R}^m, \quad f: \Delta_T \times \mathbb{R}^n  \times \mathbb{R}^n  \times \mathbb{R}^m \times \mathbb{R}^{m \times d} \to \mathbb{R}^m
\]
are continuous and satisfy, for some $L > 0$:
\begin{align*}
|g(t,\xi,x) - g(t',\xi',x')| &\le L \left( |t - t'|^{1/2} + |\xi - \xi'| + |x - x'| \right), \\
|f(t,s,\xi,x,y,z) - f(t',s,\xi',x',y',z')| &\le L \left( |t - t'|^{1/2}+  |\xi - \xi'| \right.\\ &\left.+ |x - x'| + |y - y'| + \|z - z'\| \right),\\
|f(t,s,\xi,x,0,0) | &\le L \left( 1+|\xi| + |x|  \right).
\end{align*}
\end{enumerate}
\end{assumption}

\paragraph{Existence and Uniqueness.}
Under Assumption \ref{ass},  the SDE \eqref{eq:SDE} admits a unique strong solution, and the BSVIE \eqref{eq:BSVIE} has a unique adapted solution pair \((Y,Z) \in \mathcal{H}^2 \times {L}^2\) (cf. \cite{yong2008well}). Moreover, 
it holds \[
\sup_{t \in [0,T]} \mathbb{E}|Y(t)|^2 + \mathbb{E} \int_0^T \int_t^T |Z(t,s)|^2 ds dt < \infty.
\] 
\subsection{Feynman-Kac Representation and Neural Approximation}
\label{sec:Feynman}
Wang and Yong~\cite{wang2019backward} extended the classical Feynman--Kac formula to the BSVIEs setting and showed that the solution can be represented in terms of a path-dependent PDE defined on the time simplex $\Delta_T$.\\
More precisely, there exists a deterministic function $u : \Delta_T \times \mathbb{R}^n \times \mathbb{R}^n \to \mathbb{R}^m$ such that:
\begin{align}
Y(t) &= u(t, t, X(t), X(t)), \quad t \in [0,T], \notag \\
Z(t,s) &= \nabla_x u(t, s, X(t), X(s)) \sigma(s, X(s)), \quad (t,s) \in \Delta[0,T],
\label{eq:FK-representation}
\end{align}
where $u$ is the solution of the following PDE:
\[
\begin{cases}
\partial_s u(t,s,\xi,x) + \dfrac{1}{2} \sigma(s,x)^\top D^2_x u(t,s,\xi,x)\, \sigma(s,x) + D_x u(t,s,\xi,x)\, b(s,x) \\
\quad + f\left(t, s, \xi, x, u(s,s,x,x), D_x u(t,s,\xi,x)\, \sigma(s,x)\right) = 0, \\
\hfill (t,s,\xi,x) \in \Delta_T \times \mathbb{R}^n \times \mathbb{R}^n, \\
u(t,T,\xi,x) = g(t,\xi,x).
\end{cases}
\]
The Feynman-Kac representation \eqref{eq:FK-representation} provides an analytic link between the solution of BSVIEs and a class of PDEs: 
the solution  $Y(t)$ depends on the current time and state $(t, X(t))$, while $Z(t,s)$ depends on both the current and future state pairs $(t, s, X(t), X(s))$. This insight motivates our neural network approach: rather than discretizing the PDE, we can directly approximate these underlying deterministic functions using parameterized function approximators. Specifically, we propose a deep learning algorithm, \textbf{DeepBSVIE}, which approximates the solution pair \((Y(t), Z(t,s))\) using parameterized neural networks informed by \eqref{eq:FK-representation}. We model:
\[
Y(t) \approx Y_\xi(t, X(t)), \quad Z(t,s) \approx Z_\eta(t, s, X(t), X(s)),
\]
where $\theta = ( \xi, \eta)$ denotes trainable parameters. This architecture reflects the analytic structure revealed by the Feynman-Kac representation and preserves the essential two-time dependency and forward-backward coupling structure of BSVIEs. By learning these functions from simulated trajectories rather than solving the PDE directly, our method circumvents the curse of dimensionality while maintaining theoretical consistency with the underlying analytical framework. The detailed algorithm and training procedure are presented in Section \ref{sec:scheme}.

\section{Measurability in Product Spaces for BSVIEs}
\label{sec:measurability}
In the Volterra setting, one must ensure that the solution maps
\[
  (t, u, \omega) \;\longmapsto\; \bigl(Y_u(t, \omega),\,Z_u(t, \omega)\bigr)
\]
are measurable with respect to the product $\sigma$-algebra \(\mathcal{B}([0,T]^2) \otimes \mathcal{F}\). 
To guarantee this, we employ measurability results for stochastic integrals with jointly measurable integrands.

For simplicity of presentation and notation, we restrict ourselves to the one-dimensional case throughout this section.

Before presenting the main results of this section, we clarify their role in the overall framework.

In the context of BSVIEs, the two-parameter structure of the solution \((Y_u(t), Z_u(t))\) requires that we rigorously verify its measurability on the product space \([0,T]^2 \times \Omega\). This is essential not only for the mathematical well posedness of the problem but also for the practical implementation of learning algorithms, which rely on sampling and regression over this domain.

To this end, we establish measurability of conditional expectations and stochastic integrals with parameter dependence, using classical results adapted to our setting.

\begin{lemma}[Measurable Conditional Expectations]\label{lem-SY}
Let 
\[
  X \;:\; [0,T]\times\Omega \;\longrightarrow\; \R
\]
be \(\mathcal{B}([0,T])\otimes\mathcal{F}\)-measurable, non-negative, and satisfy \(\E[X_u(\cdot)]<\infty\) for each \(u\). Then the map
\[
  (u,\omega)\;\longmapsto\;
  Y_u(\omega)\;:=\;\E\bigl[X_u(\cdot)\mid\mathcal{G}\bigr](\omega)
\]
is also \(\mathcal{B}([0,T])\otimes\mathcal{G}\)-measurable, where \(\mathcal{G}\subseteq\mathcal{F}\).
\end{lemma}
\begin{theorem}[Stricker-Yor Measurability for Parameter Dependent Integrals]\label{thm-SY}
Let 
\[
  J:\;[0,T]_u\times[0,T]_t\times\Omega\;\to\;\R
\]
be measurable with respect to \(\mathcal{B}([0,T]_u)\otimes\mathcal{P}\) (where \(\mathcal{P}\) is the predictable $\sigma$-algebra in \(t\)) and assume that, for each fixed \(u\), \(J_u(\cdot)\) is integrable with respect to the (non-parameterized) Brownian motion \(B\). Then the map
\[
  (t,u,\omega)\;\longmapsto\;
  \int_0^t J_u(s,\omega)\,dB(s)
\]
is measurable with respect to 
\(\mathcal{B}([0,T]_t)\otimes\mathcal{B}([0,T]_u)\otimes\mathcal{F}\),
and for each fixed \(u\), the process \(t\mapsto\int_0^t J_u(s)\,dB(s)\) coincides almost surely (as usual) with the classical Ito integral.
\end{theorem}
%\section{Measurability of BSVIE Solutions in Product Spaces}

Our goal now is to establish the joint measurability of the solution \((Y_u,Z_u)\) of a BSVIE in the product space \([0,T]^2 \times \Omega\).

\begin{theorem}[Measurability of the BSVIE Solution]\label{meas-thm}
Assume that:
\begin{enumerate}
    \item \(\phi: [0,T] \times \Omega \to \mathbb{R}\) is \( \mathcal{B}([0,T]) \otimes \mathcal{P} \)-measurable;
    \item \(f: [0,T] \times [0,T] \times \Omega \times \mathbb{R} \times \mathbb{R} \to \mathbb{R}\) is \( \mathcal{B}([0,T]^2) \otimes \mathcal{P} \)-measurable in the first two arguments and satisfies Lipschitz and linear growth conditions in the last two arguments.
\end{enumerate}
Then the solution \((Y_u, Z_u)\) to the BSVIE
\[
Y_u(t) = \phi_u + \int_t^T f_u(s, Y_u(s), Z_u(s)) \, ds - \int_t^T Z_u(s) \, dB(s), \quad \forall\, t, u \in [0,T],
\]
is \( \mathcal{B}([0,T]^2) \otimes \mathcal{F} \)-measurable.
\end{theorem}

\begin{proof}
We construct a Picard iteration scheme.

\paragraph{Step 1 (Inductive Construction)} Define the initial guess:
\[
Y^0_u(t) := 0, \quad Z^0_u(t) := 0.
\]
Given \((Y^n_u, Z^n_u)\), define the next iterate by:
\[
Y^{n+1}_u(t) := \phi_u + \int_t^T f_u(s, Y^n_u(s), Z^n_u(s))\, ds - \int_t^T Z^{n+1}_u(s)\, dB(s).
\]

We show by induction that \((Y^n_u, Z^n_u)\) is \( \mathcal{B}([0,T]^2) \otimes \mathcal{F} \)-measurable for each \(n\). The base case is trivial. For the inductive step, note that if \((Y^n_u, Z^n_u)\) is measurable, then \(f_u(s,Y^n_u,Z^n_u)\) is also measurable due to the assumptions on \(f\).

Define:
\[
X_u := \phi_u + \int_0^T f_u(s,Y^n_u(s),Z^n_u(s))\, ds.
\]
Then \( X_u \) is \( \mathcal{B}([0,T]) \otimes \mathcal{F} \)-measurable and square integrable. Define the adapted martingale:
\[
M_u(t) := \mathbb{E}[X_u \mid \mathcal{F}_t].
\]
By Lemma \eqref{lem-SY}, \( M_u(t) \) is also \( \mathcal{B}([0,T]) \otimes \mathcal{F} \)-measurable.

\paragraph{Step 2 (Stochastic Representation and Measurability)}  
Fix \(u \in [0,T]\). Since \(M_u(\cdot)\) is a square integrable martingale, the martingale representation theorem yields:
\[
M_u(t) = M_u(0) + \int_0^t Z^{n+1}_u(s)\, dB(s),
\]
where \(Z^{n+1}_u(\cdot)\) is progressively measurable. Moreover, the mapping \((s,u,\omega) \mapsto Z^{n+1}_u(s,\omega)\) is jointly measurable.

Then by Theorem \eqref{thm-SY}), the map
\[
(t,u,\omega) \mapsto \int_0^t Z^{n+1}_u(s)\, dB(s)
\]
is \( \mathcal{B}([0,T]^2) \otimes \mathcal{F} \)-measurable. Hence, the updated \(Y^{n+1}_u(t)\) is also measurable.

\paragraph{Step 3 (Limit Argument)}  
The sequence \((Y^n_u, Z^n_u)\) converges in \(L^2([0,T]^2 \times \Omega)\) to the solution \((Y_u,Z_u)\) of the BSVIE. As the pointwise limit of measurable functions, \((Y_u,Z_u)\) is also \( \mathcal{B}([0,T]^2) \otimes \mathcal{F} \)-measurable.
\end{proof}

\begin{remark}
This result ensures that BSVIE solutions are measurable over the product space \([0,T]^2 \times \Omega\), which is critical for numerical schemes and stability analysis.
\end{remark}

\subsection{BSVIEs on the Product Space}

We now reformulate the BSVIE framework on an extended probability space:
\begin{itemize}
    \item \( \tilde{\Omega} := \Omega \times [0,T] \),
    \item \( \tilde{\mathcal{F}} := \mathcal{F} \otimes \mathcal{B}([0,T]) \),
    \item \( \tilde{\mathbb{P}} := \mathbb{P} \otimes \eta \), where \( \eta \) is typically the Lebesgue measure on \( [0,T] \).
\end{itemize}

Expectations over this product space are written as:
\[
\tilde{\mathbb{E}}[Y_u(t)] = \int_\Omega \int_0^T Y_u(t)(\omega)\, \eta(du)\, \mathbb{P}(d\omega).
\]

\paragraph{Function spaces.}
\begin{itemize}
    \item \( L^2_{\mathcal{F}}([0,T]^2) \): the space of jointly measurable processes \( Y_u : [0,T]^2 \times \Omega \to \mathbb{R} \), adapted in the \( t \)-variable to \( \mathbb{F} \), with
    \[
    \|Y\|^2 := \mathbb{E} \left[ \int_0^T \int_0^T |Y_u(t)|^2 \, du \, dt \right] < \infty.
    \]

    \item \( L^2_{\mathcal{F}_T}([0,T]) \): the space of random fields \( \phi_u : [0,T] \times \Omega \to \mathbb{R} \) that are \( \mathcal{F}_T \otimes \mathcal{B}([0,T]) \)-measurable, with
    \[
    \|\phi\|^2 := \mathbb{E} \left[ \int_0^T |\phi_u|^2 \, du \right] < \infty.
    \]

    \item \( L^2_{\mathcal{F}}([0,T]^2) \): similarly for \( Z_u(s) \), with measurability in \( (s,u) \) and \( \mathcal{F}_s \)-adaptedness in \( s \).
\end{itemize}

\paragraph{Existence and Uniqueness via Parametrized BSDEs.}

Assume that the generator \( f_u: \Omega \times [0,T]^2 \times \mathbb{R}^2 \to \mathbb{R} \) satisfies:
\begin{itemize}
    \item measurability in \( (s,u,\omega) \), and \( \mathcal{F}_s \)-adaptedness in \( s \) for each \( u \),
    \item Lipschitz continuity and linear growth in \( (y,z) \), uniformly in \( (s,u) \).
\end{itemize}

The following result ensures that the BSVIE under our assumptions has a unique solution. This foundational result guarantees that our learning-based numerical scheme targets a well defined and meaningful object.

\begin{theorem}[Existence and Uniqueness in Product Space]
Let \( \phi_u \in L^2_{\mathcal{F}_T}([0,T]) \). Then, there exists a unique solution \( (Y_u, Z_u) \in L^2_{\mathcal{F}}([0,T]^2) \times L^2_{\mathcal{F}}([0,T]^2) \) to the BSVIE:
\[
Y_u(t) = \phi_u + \int_t^T f_u(s,Y_u(s),Z_u(s))\, ds - \int_t^T Z_u(s)\, dB(s).
\]
\end{theorem}

\begin{proof}[Proof Sketch]
Fix \( u \in [0,T] \). The equation in \( t \mapsto Y_u(t) \) defines a classical BSDE. Standard results (Pardoux-Peng \cite{PP}) yield existence and uniqueness. Measurability of the map \( (t,u,\omega) \mapsto Y_u(t)(\omega) \) follows from Theorem \ref{meas-thm}.%, ensuring \( (Y_u,Z_u) \in L^2(\tilde{\Omega}, \tilde{\mathcal{F}}, \tilde{\mathbb{P}}) \).
\end{proof}

\begin{remark}
This reformulation enables numerical approximation via parallel resolution of BSDEs indexed by \( u \), which is especially suitable for machine learning-based schemes such as DeepBSVIE.
\end{remark}

\begin{remark}[Multidimensional Extension]
The measurability result established in Theorem~\ref{meas-thm} can be extended to the multidimensional case, where the solution processes take values in \( \mathbb{R}^d \) and the Brownian motion is \( \mathbb{R}^m \)-valued. In this case, the generator \( f \) and the terminal condition \( \phi \) are vector-valued, and the stochastic integral is defined componentwise. The same Picard iteration and measurability arguments apply by treating each coordinate separately and using the Stricker-Yor theorem componentwise. 
\end{remark}

\section{Numerical Solutions for BSVIEs}
\label{sec:scheme}
A central motivation for our numerical scheme stems from the PDE based representation of adapted solutions to BSVIEs, as presented in \eqref{eq:FK-representation}. In this section, we develop a neural network-based solver for BSVIEs. Our approach extends the deep BSDE methodology to the Volterra setting, capturing the bi-temporal and path dependent structure inherent in such equations.

\subsection{The DeepBSVIE Algorithm}
\label{sec:algo}
We consider a uniform time discretization of the interval $[0,T]$ with $N$ points:
\[
\Delta t := \frac{T}{N}, \quad t_i := i \Delta t, \quad i = 0, \ldots, N.
\]
For each discrete time index \( i \), we consider the set of future indices \( j \in \{i+1, \ldots, N\} \), over which the second time variable in $Z(t_i, t_j)$ evolves.

At each time step \( t_i \), we train two separate neural networks to approximate the solution components:
\begin{itemize}
    \item \( Y(t_i) \approx \mathcal{Y}^i(t_i, X_i) \),
    \item \( Z(t_i, t_j) \approx \mathcal{Z}^i(t_i, t_j, X_i, X_j) \), for all \( j \in \{i+1, \ldots, N\} \).
\end{itemize}

Each network is implemented as a fully connected feedforward neural network with fixed depth and width.

\vspace{0.5em}
\paragraph{Architecture of $\mathcal{Y}^i$.} The network $\mathcal{Y}^i: \mathbb{R}^{1+n} \to \mathbb{R}^m$ is defined as
\[
\mathcal{Y}^i(t_i, X_i) = W_L^{(Y)} \, \varphi \left( \cdots \varphi \left( W_1^{(Y)} 
\begin{bmatrix}
t_i \\
X_i
\end{bmatrix}
+ b_1^{(Y)} \right) \cdots \right) + b_L^{(Y)},
\]
where:
\begin{itemize}
    \item $W_\ell^{(Y)}$, $b_\ell^{(Y)}$ are trainable weights and biases,
    \item $\varphi$ is the activation function,
    \item the input vector is $(t_i, X_i) \in \bR^{1+n}$, \item the output vector approximates \( Y(t_i)\in \bR^m \).
\end{itemize}

\paragraph{Architecture of $\mathcal{Z}^i$.} The network $\mathcal{Z}^i: \mathbb{R}^{2+2n} \to \mathbb{R}^{m\times d}$ is defined as
\[
\mathcal{Z}^i(t_i, t_j, X_i, X_j) = W_L^{(Z)} \, \varphi \left( \cdots \varphi \left( W_1^{(Z)} 
\begin{bmatrix}
t_i \\
t_j \\
X_i \\
X_j
\end{bmatrix}
+ b_1^{(Z)} \right) \cdots \right) + b_L^{(Z)},
\]
where the input dimension is $2 + 2n$ corresponding to $(t_i,t_j,X_i,X_j)$ and output is a matrix in $\mathbb{R}^{m\times d}$ approximating \( Z(t_i, t_j) \).

\vspace{0.5em}
\paragraph{Forward Simulation.} To generate training data, we simulate $M$ independent sample paths of the forward process $\{X_i^k\}_{i=0}^N$ using the Euler-Maruyama scheme:
\[
X_{i+1}^k = X_i^k + b(t_i, X_i^k) \Delta t + \sigma( t_i, X_i^k) \Delta B_{i+1}^k, \quad X_0^k = x,
\]
where \(
\Delta B_{i+1}^k \in \mathbb{R}^d \)  is the Brownian increment over \([t_i, t_{i+1}],\)  sampled as \(\mathcal{N}(0, \Delta t \cdot I_d)\).

\vspace{0.5em}
\paragraph{Discrete BSVIE Approximation.} For each path $k = 1, \ldots, M$ and time $t_i$, we estimate:
\[
\widehat{Y}_i^k := \mathcal{Y}^i(t_i, X_i^k ; \xi), \quad 
\widehat{Z}_i^k(t_i, t_j) := \mathcal{Z}^i(t_i, t_k, X_i^k, X_j^k; \eta),
\]
where \(\xi\) denotes the collection of trainable parameters (weights and biases) of the \(\mathcal{Y}^i\) networks, and \(\eta\) denotes the collection of trainable parameters of the \(\mathcal{Z}^i\) networks.\\

We define the discrete BSVIE residual:
\[
\mathcal{G}_i^k := g(t_i, X_i^k, X_N^k) +
\sum_{j=i}^{N-1} f(t_i, t_j, X_i^k, X_j^k, \widehat{Y}_j^k, \widehat{Z}_i^k(t_i, t_j)) \Delta t -
\sum_{j=i}^{N-1} \widehat{Z}_i^k(t_i, t_j) \Delta B_j^k.
\]
%\vspace{0.5em}
\paragraph{Loss Function.} The sample wise loss is:
\begin{equation}
\ell_i^k := \| \widehat{Y}_i^k - \mathcal{G}_i^k \|_2,
\quad
\ell_i := \frac{1}{M} \sum_{k=1}^M \ell_i^k
\label{loss}
\end{equation}
%\vspace{0.5em}

%\vspace{0.5em}
\paragraph{Summary of the Algorithm.} The full backward training and simulation scheme is summarized below.

\begin{algorithm}[H]
\caption{DeepBSVIE: Deep Learning Solver for BSVIEs}
\label{algo:BSVIE}
\begin{algorithmic}[1]
\For{$i = N$ to $0$}
    \For{each training epoch}
        \For{$k = 1$ to $M$}
            \State Initialize $X_0^k = x$
            \For{$\ell = 0$ to $N$}
                \State Sample $\Delta B_{\ell+1}^k$
                \State $X_{\ell+1}^k = X_\ell^k + b( t_\ell, X_\ell^k) \Delta t + \sigma( t_\ell, X_\ell^k) \Delta B_{\ell+1}^k$
            \EndFor
            \State $\widehat{Y}_i^k = \mathcal{Y}^i(t_i, X_i^k; \xi_i)$
            \For{$j = i+1$ to $N$}
                \State $\widehat{Z}_i^k(t_i, t_j) = \mathcal{Z}^i(t_i, t_j, X_i^k, X_j^k; \eta_i)$
            \EndFor
            \State $\mathcal{G}_i^k = g(t_i, X_i^k, X_N^k) + \sum_j f \Delta t - \sum_j \widehat{Z}_i^k \Delta B_j^k$
            \State $\ell_i^k = |\widehat{Y}_i^k - \mathcal{G}_i^k|^2$
        \EndFor
        \State Compute $\ell_i = \frac{1}{M} \sum_k \ell_i^k$
        \State Update $\xi_i$, $\eta_i$ 
    \EndFor
\EndFor
\State \Return $\{ \widehat{Y}_i^k, \widehat{Z}_i^k(t_i, \cdot) \}_{i}$
\end{algorithmic}
\end{algorithm}

\vspace{0.5em}
\paragraph{Error Metrics.} 
When reference solutions are available, we compute the empirical $L^2$-errors:

\begin{equation}
\begin{aligned}
\mathcal{E}(Y) &:= \frac{1}{M} \sum_{k=1}^M \sum_{i=0}^{N} \left| Y_{t_i}^{k} - \widehat{Y}_i^{k} \right|^2\Delta t,   \\
\mathcal{E}(Z) &:= \frac{1}{M}   \sum_{k=1}^M \sum_{i=0}^{N} \sum_{j=i}^{N} \| Z_{t_i,t_j}^{k} - \widehat{Z}_{i,j}^{k} \|\, (\Delta t)^2.
\end{aligned}
\label{eq:err}
\end{equation}
The corresponding relative $L^2$-errors:

\begin{equation}
\begin{aligned}
\mathcal{E}_R(Y) &:= \frac{ \sum_{k=1}^M \sum_{i=0}^{N}  | Y_{t_i}^{k} - \widehat{Y}_i^{k} |^2}
{ \sum_{k=1}^M \sum_{i=0}^{N}  | Y_{t_i}^{k} |^2 }, \\
\mathcal{E}_R(Z) &:= \frac{  \sum_{k=1}^M \sum_{i=0}^{N} \sum_{j=i}^{N} \| Z_{t_i,t_j}^{k} - \widehat{Z}_{i,j}^{k} \|^2 }
{ \sum_{k=1}^M \sum_{i=0}^{N} \sum_{j=i}^{N} \| Z_{t_i,t_j}^{k} \|^2},
\end{aligned}
\label{eq:rel_err}
\end{equation}
where, for each sample path \(k\):

\begin{itemize}
    \item \(Y_{t_i}^k\) and \(Z_{t_i, t_j}^k\) denote the reference (ground truth) values of the solution components evaluated at discrete times \(t_i\) and \((t_i, t_j)\), respectively.
    \item \(\widehat{Y}_i^k\) and \(\widehat{Z}_{i,j}^k := \widehat{Z}^k_{i}(t_i,t_j)\) denote the corresponding neural network approximations.
\end{itemize}
\paragraph{Implementation.}The solver is implemented in Python using \texttt{PyTorch}. Across all the examples presented in this paper, we use 
$N = 50$ time steps with $T = 1$ and generate batch samples of size 
$M = 2^{12}$ at each training iteration. The algorithm proceeds backward 
from $i = N$ to $i = 0$, training one pair of neural networks 
$(Y_i, Z_i)$ per time step. Both networks have $L = 3$ hidden layers with 
$\tanh$ activation: the $Y$-network uses $h_Y =40$ neurons per layer, while the 
$Z$-network uses $h_Z = 80$ neurons per layer. \\
At each time step $i$, we initialize the networks using a warm-start from 
the converged parameters at $i + 1$ (except at the terminal step $i = N$), 
and optimize using the \texttt{AdamW} optimizer. The learning rate is initially set to $10^{-2}$. It decays by a factor of $0.995$ at each time step, and additionally, it is adjusted adaptively within each epoch based on the optimizer's schedule. The terminal condition at $i = N$ 
is trained for up to $1000$ iterations, while backward steps $i < N$ use up 
to $500$ iterations each. 
The vectorized computation of $\hat{Z}_i(t_i, t_j)$ for all 
$j \in \{i, i+1, \ldots, N-1\}$ is performed in a single forward pass, 
avoiding nested loops over and maintaining constant computational time per epoch, regardless of the number of future timesteps. 
Complete implementation details are available at
\url{https://github.com/giuliapucci98/DeepBSVIE.git}.

\paragraph{Comparison with existing methods.} 
To the best of our knowledge, \cite{andersson2025deep} is the only other work that develops a deep learning solver for BSVIEs. Their approach employs a forward in time scheme that solves all timesteps simultaneously using a single global neural network. In contrast, our method respects the inherent backward nature of BSVIEs by starting at the terminal condition $t=T$ and proceeding backward in time. This backward localized approach has several advantages, for example (i) it treats future values as ground truth at each step, training one neural network per timestep using already converged solutions from subsequent times, thereby avoiding the circular dependencies present in simultaneous optimization across all timesteps; provided that the approximation is performed correctly at each step, this allows the method to work on longer time intervals and reduces dependence on the length of the time horizon; (ii) it is more memory efficient, as only the current timestep's network needs to be stored during training, allowing our method to scale to arbitrarily large values of $N$ without memory constraints. Despite these architectural differences, both methods achieve high accuracy and similar computation time, demonstrating that the backward localized framework provides a viable alternative perspective for solving BSVIEs with deep learning.\\
Besides these deep learning approaches, it is worth mentioning that alternative solution approaches like classical finite difference and Monte Carlo methods for BSVIEs, while theoretically applicable, would suffer from significant computational limitations, especially in higher dimensions.

\section{Convergence Analysis of the Deep Neural Scheme for BSVIEs }
\label{sec:convergence}
In this section, we analyze the convergence of a neural network-based solver for BSVIEs, using a discrete time scheme inspired by the Euler-Maruyama method and extending the deep BSDE framework in \cite{hure2020deep} to the more intricate two time BSVIE structure. To approximate the continuous time solution of the BSVIE, we adopt the discrete time scheme proposed by \cite{hamaguchi2023approximations}. Our main theoretical result establishes that the outputs of our neural network algorithm converge to the solution of this discrete time approximation as the network capacity and training accuracy increase.

\subsection{Discretization Scheme}
In this subsection, we summarize the main assumptions and the discrete time scheme for BSVIEs as introduced in \cite{hamaguchi2023approximations}.
We consider the following assumptions on the model coefficients:

\begin{assumption}
\label{ass:scheme_bsig}
The maps 
\[
b : [0,T] \times \mathbb{R}^n \to \mathbb{R}^n \quad \text{and} \quad \sigma : [0,T] \times \mathbb{R}^n \to \mathbb{R}^{n \times d}
\]
are measurable. Moreover,
\begin{enumerate}[label=(\alph*)] 
\item there exists a constant \(L > 0\) such that for any \(s \in [0,T]\) and \(x, x' \in \mathbb{R}^n\),
\[
|b(s,0)| + \|\sigma(s,0)\| \leq L,
\]
and
\[
|b(s,x) - b(s,x')| + \|\sigma(s,x) - \sigma(s,x')\| \leq L |x - x'|,
\]
\item for any \(s, s' \in [0,T]\) and \(x \in \mathbb{R}^n\), it holds that
\[
|b(s, x) - b(s', x)| + \|\sigma(s, x) - \sigma(s', x)\| \leq L |s - s'|^{\frac{1}{2}} \big(1 + |x|\big).
\]
\end{enumerate}
\end{assumption}

\begin{assumption}
\label{ass:scheme_gf}
The maps 
\[
g : [0, T] \times \mathbb{R}^n \times \mathbb{R}^n \to \mathbb{R}^m
\quad \text{and} \quad
f: \Delta_T \times \mathbb{R}^n  \times \mathbb{R}^n  \times \mathbb{R}^m \times \mathbb{R}^{m \times d} \to \mathbb{R}^m,
\]

are measurable. Moreover, 
\begin{enumerate}[label=(\alph*)] 
\item there exists a constant \(L > 0\) such that
\[
\int_0^T |g(t,0,0)|^2 dt + \int_0^T \int_t^T |f(t,s,0,0,0,0)|^2 ds dt \leq L,
\]
\item for any \((t,s) \in \Delta_T\) and
\[
(x_1,x_2,y,z_1), (x_1', x_2', y', z_1') \in \mathbb{R}^n \times \mathbb{R}^n \times \mathbb{R}^m \times \mathbb{R}^{m \times d},
\]
the following Lipschitz condition holds:
\begin{equation*}
    \begin{aligned} 
|g(t,x_1,x_2) - g(t,x_1', x_2')| + |f(t,s,x_1,x_2,y,z_1) - f(t,s,x_1', x_2', y', z_1')| \\ \leq L \big( |x_1 - x_1'| + |x_2 - x_2'| + |y - y'| + \|z_1 - z_1'\| \big),
    \end{aligned}
\end{equation*}
\item for any \((t,s), (t', s') \in \Delta_T\) and \((x_1, x_2, y, z_1) \in \mathbb{R}^n \times \mathbb{R}^n \times \mathbb{R}^m \times \mathbb{R}^{m \times d}\), it holds that
\[
\begin{aligned}
& |g(t,x_1,x_2) - g(t', x_1, x_2)| + |f(t,s,x_1,x_2,y,z_1) - f(t', s', x_1, x_2, y, z_1)| \\
& \quad \leq L \left( |t - t'|^{\frac{1}{2}} + |s - s'|^{\frac{1}{2}} \right) \left( 1 + |x_1| + |x_2| + |y| + \|z_1\| \right).
\end{aligned}
\]
\end{enumerate}
\end{assumption}

\begin{remark}
Assumptions  \ref{ass:scheme_bsig} and \ref{ass:scheme_gf} ensure the well-posedness of the forward diffusion as well as the stability of the backward Volterra equation under time discretization. Compared to Assumption~\ref{ass}, only required for the existence of a continuous-time BSVIE solution, these assumptions strengthen the regularity of the coefficients to ensure the convergence of the discretization scheme proposed in \cite{hamaguchi2023approximations}. Specifically, besides the classical Lipschitz continuity and linear growth conditions, we impose sharper H\"older-type continuity in the time variables on the triangular domain $\Delta_T$.  We note that such assumptions are standard in the numerical analysis of BSVIEs and are commonly used to control discretization errors. 
\end{remark}

\paragraph{Forward SDE Approximation.} 
For the forward SDE, we use the classical Euler-Maruyama scheme,  which discretizes the time interval $[0,T]$ into a uniform grid and approximates the continuous time process by incremental updates driven by the drift and diffusion coefficients evaluated at discrete times as follows:
\[
X^N_{i+1} = X^N_i + b(t_i, X^N_i) \Delta t + \sigma(t_i, X^N_i) \Delta B_i, \quad X_0 = x,
\]
which satisfies the convergence estimate
\[
\max_{0 \le i \le N} \mathbb{E} \big| X(t_i) - X^N_i \big|^2 \leq C \Delta t.
\]

\paragraph{Discrete BSVIE Scheme.} We adopt the discrete scheme proposed by \cite{hamaguchi2023approximations}. The scheme defines the discrete approximations $Y^{k,N}_l$ and $Z^{k,N}_l$ backwards in time. 
On a uniform partition \(\{t_k\}_{k=0}^N\) of \([0,T]\). For \(k \le l\), define
\begin{align*}
Y_l^{k,N} &= \mathbb{E}_l\left[ Y_{l+1}^{k,N} \right] + \Delta t\, f\big(t_k, t_l, X_k^N, X_l^N, Y_l^{l,N}, Z_l^{k,N}\big), \\
Z_l^{k,N} &= \frac{1}{\Delta t} \mathbb{E}_l \left[ Y_{l+1}^{k,N} \Delta B_l^\top \right],
\end{align*}
where \(\mathbb{E}_l[\cdot] := \mathbb{E}[\cdot \mid \mathcal{F}_{t_l}]\).
By unrolling the backward recursion, we express the solution at the grid points $t_k$ in terms of the terminal condition and the discrete driver evaluated along the approximate forward paths.
We thus have
\[
Y_k^{k,N} = \mathbb{E}_k \left[ g(t_k, X^N_k, X^N_N) + \sum_{l=k}^{N-1} f\big(t_k, t_l, X^N_k, X^N_l, Y_l^{l,N}, Z_l^{k,N}\big) \Delta t \right].
\]

Before stating the discrete martingale representation result (Lemma \ref{lem:discrete-mrt-fixed-k}), we emphasize that the construction of the \({L}^2\)-projections \(Z^{k,N}_l\) appearing in the scheme relies crucially on the joint measurability of the solution fields \((t_k, t_l) \;\longmapsto\; ( Y^{k,N}_l, Z^{k,N}_l)\). These properties were rigorously established in Section \ref{sec:measurability} using results of Stricker and Yor. In particular, the measurability in the product space \([0,T]^2 \times \Omega\) ensures that the regression targets used in the discrete time approximation and the associated conditional expectations are well defined. This justifies the projection formula that follows.
\begin{lemma}[Discrete Martingale Representation]
\label{lem:discrete-mrt-fixed-k}
Let \( \mathcal{F}_{t_l} \subset \mathcal{F}_{t_{l+1}} \) be the Brownian filtration on the discretized grid. Then for any \( \xi \in L^2(\mathcal{F}_{t_{l+1}}) \), there exists a unique \( \mathcal{F}_{t_l} \)-measurable random vector \( Z_l \in \mathbb{R}^d \) such that:
\[
\xi = \mathbb{E}_l[\xi] + Z_l \cdot \Delta B_l,
\quad \text{with} \quad
Z_l = \frac{1}{\Delta t} \mathbb{E}_l[\xi \Delta B_l^\top].
\]
\end{lemma}

\begin{remark}
In the context of our discrete BSVIE scheme, we apply Lemma~\ref{lem:discrete-mrt-fixed-k} with \( \xi := Y_{l+1}^{k,N} \), which is \( \mathcal{F}_{t_{l+1}} \)-measurable. The corresponding projection
\[
Z_l^{k,N} := \frac{1}{\Delta t} \, \mathbb{E}_l \left[ Y_{l+1}^{k,N} \Delta B_l^\top \right]
\]
yields the unique \( \mathcal{F}_{t_l} \)-measurable vector such that
\[
Y_{l+1}^{k,N} = \mathbb{E}_l[Y_{l+1}^{k,N}] + Z_l^{k,N} \cdot \Delta B_l,
\]
ensuring that the discrete martingale representation is preserved in our approximation scheme.
\end{remark}

The following theorem (see Theorem $5.5$ in \cite{hamaguchi2023approximations}) states the convergence of the discrete scheme to the continuous BSVIE solution.

\begin{theorem}
Under Assumptions \ref{ass:scheme_bsig} and \ref{ass:scheme_gf}, there exists a constant \(\delta > 0\), depending only on the Lipschitz constant \(L\), such that for any partition of the interval $[0,T]$ with mesh size \(\Delta t \leq \delta\), the following holds:
\small
\[
\sum_{k=0}^{N-1} \mathbb{E} \left[ \int_{t_k}^{t_{k+1}} \left| Y(t) - Y_k^{k,N} \right|^2 dt \right] 
+ \sum_{k=0}^{N-1} \sum_{l=k}^{N-1} \mathbb{E} \left[ \int_{t_k}^{t_{k+1}} \int_{t_l}^{t_{l+1}} \left| Z(t,s) - Z_l^{k,N} \right|^2 ds dt \right]
\leq C (1 + |x|^2) \Delta t,
\]
where \((Y(t), Z(t,s))\) is the solution of the continuous BSVIE and \(C\) depends on \(L\) and \(T\).
\label{thm:discrete-convergence}
\end{theorem}

\subsection{Convergence result via Error Decomposition}

Having introduced the discretization scheme and the conditions under which it converges to the solution of the continuous-time BSVIE, we now turn to the second stage of our analysis. Our goal is to investigate the accuracy of a neural network-based approximation of the discrete BSVIE solution.

To this end, we decompose the total error between the true continuous solution and the neural network approximation into two components:
\begin{enumerate}
    \item the error between the continuous solution and its time discretized counterpart,
    \item the error between the discrete solution and the neural network outputs.
\end{enumerate}

The first component has already been controlled by Theorem \ref{thm:discrete-convergence}, with a bound of order \( \mathcal{O}(\Delta t) \). We now focus on analyzing the second component, the approximation error between the discrete scheme \( (Y_k^{k,N}, Z^{k,N}_l) \) and the neural network outputs \( (\widehat{Y}_k, \widehat{Z}_{k,l}) \), and derive conditions under which this discrepancy remains uniformly small across the time grid.\\

To this end, we define the intermediate regression target 
\begin{equation*}
\begin{aligned}
    V_l^k &= \mathbb{E}_l\left[ V_{l+1}^k \right] + \Delta t\, f(t_k, t_l, X_k^N, X_l^N, \widehat{Y}_l, \bar{Z}^k_l), \\ 
    \bar{Z}_l^k &= \frac{1}{\Delta t} \mathbb{E}_l \left[ V_{l+1}^k\Delta B_l^\top \right].
\end{aligned}
\end{equation*}
with \( V_N^k = g(t_k,X^N_k,X^N_N)\) and define \( V^k \coloneqq V^k_k\). By unrolling $V_l^k$ backward in time we get: 
 \begin{equation*}
\begin{aligned}
    V^k &= \mathbb{E}_l\left[ g(t_k, X_k, X_N) + \sum_{l=k}^{N-1} f\big(t_k, t_l, X_k^N, X_l^N, \widehat{Y}_{l}, \bar{Z}^k_l\big) \Delta t \right], 
\end{aligned}
\end{equation*} and by the discrete martingale representation theorem we can write it: \begin{equation*}
g(t_k, X_k, X_N) = V^k - \sum_{l=k}^{N-1} f\left(t_k, t_l, X_k^N, X_l^N, \widehat{Y}_l, \bar{Z}_l^k\right) \Delta t + \sum_{l=k}^{N-1}  \bar{Z}_l^k \, \Delta B_l.
\end{equation*}
By the Markov Property of the discretized forward process, for $k = 0, \ldots, N-1, \; l \ge k$, there exist some deterministic functions $y_{k,k}$ and $z_{k,l}$ such that $y_{k,k}(X_k^N) = V_k^k$ and $z_{k,l}(X_k^N, X_l^N) = \bar{Z}_l^k$.\\
We denote the neural approximation errors by:
\begin{align*}
\varepsilon_k^{\text{NN}, y} &:= \inf_\xi \mathbb{E} \left| y_{k,k}(X_k^N) - \widehat{Y}_k(X_k^N; \xi) \right|^2, \\
\varepsilon_{k,l}^{\text{NN}, z} &:= \inf_\eta \mathbb{E} \left| z_{k,l}(X_k^N, X_l^N) - \widehat{Z}_k(X_k^N, X_l^N; \eta) \right|^2,
\end{align*}
which measure how well the neural networks can learn the target functions $y_{k,k}$ and $z_{k,l}$ based on the data from the forward process.
\\

The following theorem provides an a priori estimate for the global error of the DeepBSVIE scheme. It quantifies how the neural network approximation errors arising at each time step and each pair of time indices accumulate throughout the backward recursion. This result is essential for guiding the selection of practical network parameters such as depth, width, and training epochs.
\begin{theorem}[Error estimate of the DeepBSVIE Scheme]\label{thm:convergenceNN}
Under assumptions \ref{ass:scheme_bsig} and \ref{ass:scheme_gf}, there exists a constant \(C > 0\) independent of \(N\) such that the total approximation error satisfies
\begin{equation}
    \sum_{k=0}^{N-1} \mathbb{E}\left| Y_k^{k,N} - \widehat{Y}_k \right|^2 + \Delta t \sum_{k=0}^{N-1} \sum_{l=k}^{N-1} \mathbb{E}\left| Z_l^{k,N} - \widehat{Z}_{k,l} \right|^2
\leq C N^2 \sum_{l=0}^{N-1} \varepsilon_l^{\text{NN}, y}
+ C N \sum_{k=0}^{N-1} \sum_{l=k}^{N-1} \varepsilon_{k,l}^{\text{NN}, z}.
\label{boundfinal}
\end{equation}
\end{theorem}

Combining Theorem \ref{thm:convergenceNN}, which controls the error between the neural network approximations and the discrete scheme, with Theorem \ref{thm:discrete-convergence}, which estimates the error between the discrete scheme and the continuous solution, we obtain full convergence of the DeepBSVIE scheme.
In other words, as the time discretization step \(\Delta t\) tends to zero and the neural network approximation errors vanish, the neural approximations \(\widehat{Y}_k\) and \(\widehat{Z}_{k,l}\) converge to the true continuous solution \((Y, Z)\) of the BSVIE.

The projection formula for \( Z_l^{k,N} \) used throughout this proof, as well as the conditional expectations defining \( V^k \), rely on the joint measurability of the solution processes \( (Y_u(t), Z_u(t)) \) established in Section~\ref{sec:measurability}, and the discrete martingale representation result in Lemma~\ref{lem:discrete-mrt-fixed-k}.

\begin{proof}[Proof]
\subsection*{Step 1: Estimate between \(Y_k^{k,N}\) and \(V^k\)}
 Throughout the proof, $C>0$ will denote a generic constant depending only on the Lipschitz constant $L$ and the time horizon $T$, which may change from line to line. \\

Fix \(k \in \{0, \dots, N-1\}\) and \(l \in \{k, \dots, N-1\}\). Recall the discrete BSVIE scheme:
\[
Y_k^{k,N} = \mathbb{E}_k \left[ g(t_k, X^N_k, X^N_N) + \sum_{l=k}^{N-1} f\big(t_k, t_l, X^N_k, X^N_l, Y_l^{l,N}, Z_l^{k,N}\big) \, \Delta t \right].
\]
and the intermediate regression target:
\begin{equation*}
\begin{aligned}
    V^k &= \mathbb{E}_k\left[ g(t_k, X^N_k, X^N_N) + \sum_{l=k}^{N-1} f\big(t_k, t_l, X_k^N, X_l^N, \widehat{Y}_{l}, \bar{Z}_l^{k}\big) \Delta t \right], 
\end{aligned}
\end{equation*}

We compute the difference:
\begin{align*}
Y_k^{k,N} - V^k &= \mathbb{E}_k \left[ \sum_{l=k}^{N-1} \left( f(t_k, t_l, X_k^N, X_l^N, Y_l^{l,N}, Z_l^{k,N}) - f(t_k, t_l, X_k^N, X_l^N, \widehat{Y}_l, \bar{Z}_l^k) \right) \Delta t \right].
\end{align*}
By Jensen and Cauchy--Schwarz inequality:
\small
\begin{equation*}
\begin{aligned}
|Y_k^{k,N} - V^k|^2 
&\leq (N - k) \Delta t^2 \sum_{l=k}^{N-1} \mathbb{E}_k \left[ \left| f(t_k, t_l, X_k^N, X_l^N, Y_l^{l,N}, Z_l^{k,N}) - f(t_k, t_l, X_k^N, X_l^N, \widehat{Y}_l, \bar{Z}_l^k) \right|^2 \right]
\end{aligned}
\end{equation*}
By the Lipschitz continuity of $f$:
\begin{align}
\mathbb{E} \left[ |Y_k^{k,N} - V^k|^2 \right] 
&\leq C \sum_{l=k}^{N-1} \Delta t \left( \mathbb{E} \left[ |Y_l^{l,N} - \widehat{Y}_l|^2 \right] + \mathbb{E} \left[ |Z_l^{k,N} - \bar{Z}_l^k|^2 \right] \right), \label{bound1}
\end{align}
For the \(Z\)-component, consider:
\[
Z_l^{k,N} = \frac{1}{\Delta t} \mathbb{E}_l \left[ Y_{l+1}^{k,N} \Delta B_l^\top \right], \quad 
\bar{Z}_l^k = \frac{1}{\Delta t} \mathbb{E}_l \left[ V_{l+1}^k \Delta B_l^\top \right],
\]
so that:
\[
Z_l^{k,N} - \bar{Z}_l^k = \frac{1}{\Delta t} \mathbb{E}_l \left[ (Y_{l+1}^{k,N} -  V_{l+1}^k) \Delta B_l^\top \right].
\]
Using It\^o isometry and Cauchy-Schwarz:
\begin{equation}
\Delta t \; \mathbb{E}|Z_l^{k,N} - \bar{Z}_l^k|^2 \le  \mathbb{E} |Y_{l+1}^{k,N} -  V_{l+1}^k|^2. \label{boundZ}
\end{equation}
Additionally, by using the recursive definitions of \( Y_l^{k,N} \) and \( V_l^k \), Jensen's inequality, and the Lipschitz continuity of \(f\) in \((y,z)\), we get
\[
|Y_l^{k,N} - V_l^k|^2 
\leq  ( 1 + C \Delta t) \mathbb{E}_l \left[ |Y_{l+1}^{k,N} - V_{l+1}^k|^2 \right]
+ C \Delta t \left( |Y_l^{l,N} - \widehat{Y}_l^l|^2 + |Z_l^{k,N} - \widehat{Z}_{k,l}|^2 \right),
\]
By induction and taking expectations,
\[
\mathbb{E}|Y_l^{k,N} - V_l^k|^2 \leq C \Delta t \sum_{j=l}^{N-1} \left( \mathbb{E} |Y_j^{j,N} - \widehat{Y}_j|^2 + \mathbb{E} |Z_j^{k,N} - \widehat{Z}_j^k|^2 \right).
\]
Thus, \eqref{boundZ} becomes
\[
\Delta t \;  \mathbb{E}|Z_l^{k,N} - \bar{Z}_l^k|^2 \leq C \Delta t  \sum_{j=l}^{N-1} \left( \mathbb{E} |Y_j^{j,N} - \widehat{Y}_j|^2 + \mathbb{E} |Z_j^{k,N} - \widehat{Z}_j^k|^2 \right).
\]
Applying the discrete Gronwall inequality yields
\begin{equation} \label{boundZ2}
\Delta t \; \mathbb{E}|Z_l^{k,N} - \bar{Z}_l^k|^2 \leq C \Delta t \sum_{j=l}^{N-1} \mathbb{E} |Y_j^{j,N} - \widehat{Y}_j|^2.
\end{equation}

\subsection*{Step 2: Estimate between \(Y_k^{k,N}\) and \(\widehat{Y}_k\) via \(V^k\)}

Using the decomposition:
\[
Y_k^{k,N} - \widehat{Y}_k = (Y_k^{k,N} - V^k) + (V^k - \widehat{Y}_k),
\]
we obtain:
\[
\mathbb{E}|Y_k^{k,N} - \widehat{Y}_k|^2 
\leq ( 1 + C\Delta t)  \mathbb{E}|Y_k^{k,N} - V^k|^2 + \left( 1 + \frac{1}{C \Delta t} \right)\mathbb{E}|V^k - \widehat{Y}_k|^2.
\]
Using bounds \eqref{bound1} and \eqref{boundZ2}: 
\small
\begin{equation*}
    \begin{aligned}
        \mathbb{E}|Y_k^{k,N} - \widehat{Y}_k|^2
\leq& C(1 + \Delta t) \sum_{l=k}^{N-1} \Delta t \, \mathbb{E}|Y_l^{l,N} - \widehat{Y}_l|^2
+ C(1 + \Delta t) \sum_{l=k}^{N-1} \Delta t \sum_{j=l}^{N-1} \mathbb{E}|Y_j^{j,N} - \widehat{Y}_j|^2 \\&
+ \left( 1 + \frac{1}{C \Delta t} \right) \mathbb{E}|V^k - \widehat{Y}_k|^2.
    \end{aligned}
\end{equation*}
re-arranging the sums: \[
\mathbb{E} \left| Y_k^{k,N} - \widehat{Y}_k \right|^2 
\leq C (1 + \Delta t)  \sum_{l=k}^{N-1} \, \mathbb{E} \left| Y_l^{l,N} - \widehat{Y}_l \right|^2 
+ \left( 1 + \frac{1}{C \Delta t} \right) \mathbb{E} \left| V^k - \widehat{Y}_k \right|^2.
\]
Using Gronwall inequality, as developed for the discrete scheme in Theorem~\ref{thm:discrete-convergence}, we deduce
\begin{equation}
\mathbb{E} \big| Y_k^{k,N} - \widehat{Y}_k \big|^2 
\leq C \left( 1 + \frac{1}{ \Delta t} \right) \, \mathbb{E} \big| V^k - \widehat{Y}_k \big|^2.
\label{bound3}
\end{equation}

\paragraph{Step 3. Loss Functional and Approximation Control}

Fix \( k \in \{0, \ldots, N-1\} \), \( l \in \{k, \ldots, N-1\} \).  
We analyze the error induced by the neural network training for the pair \( (k, l) \), using the expected squared loss between the intermediate regression target \( V^k \) and the neural outputs. The neural approximations are given by:
\[
\widehat{Y}_k = \mathcal{Y}^k(t_k, X_k^N ; \xi), \quad \widehat{Z}_{k,l} = \mathcal{Z}^k(t_k, t_l, X_k^N, X_l^N; \eta),
\] 
where \( \theta = (\xi, \eta) \) denotes the parameters of the neural networks.\\

From the definition of \(V^k\), we can substitute into the quadratic loss \eqref{loss}:
\[
\begin{aligned}
\ell_k(\theta) =&\ \mathbb{E} \bigg| \left(  V^k - \widehat{Y}_k \right) + \sum_{l=k}^{N-1} \bigg( f(t_k, t_l, X^N_k, X^N_l, V^l, \bar{Z}_l^k) \\
&\quad - f\big(t_k, t_l, X_k^N, X_l^N, \widehat{Y}_{l}, \widehat{Z}_{k,l}\big) \bigg) \Delta t \bigg|^2 + \mathbb{E} \left|
\sum_{l=k}^{N-1} \left( \bar{Z}_l^k - \widehat{Z}_{k,l} \right) \right|^2 \Delta t.
\end{aligned}
\]

To control \(\ell_k(\theta)\), we apply Young's inequality and the Lipschitz continuity of \(f\): 

\begin{equation*}
\begin{aligned}
\ell_k(\theta) 
&\leq (1 + C\Delta t) \mathbb{E} \left| V^k- \widehat{Y}_k  \right|^2 \\
&\quad + C \Delta t \sum_{l=k}^{N-1} \mathbb{E} \left| f(t_k, t_l, X^N_k, X^N_l, \widehat{Y}_l, \bar{Z}_l^k) - f(t_k, t_l, X_k^N, X_l^N, \widehat{Y}_{l}, \widehat{Z}_{k,l}) \right|^2 \\
&\quad + C \Delta t \sum_{l=k}^{N-1} \mathbb{E} \left| \bar{Z}_l^k - \widehat{Z}_{k,l} \right|^2.
\end{aligned}
\end{equation*}

and by using the Lipschitzianity of $f$: 
\begin{equation*}
\begin{aligned}
\ell_k(\theta) \leq (1 + C\Delta t)\, \mathbb{E} \left|  V^k - \widehat{Y}_k \right|^2 +
C \Delta t \sum_{l=k}^{N-1} \mathbb{E}  \left| \bar{Z}_l^k - \widehat{Z}_{k,l} \right|^2
\end{aligned}
\end{equation*}

Similarly, applying a suitable reverse Young's inequality of the form
\[
(a + b)^2 \geq (1 - \gamma \Delta t) a^2 - \frac{1}{\gamma \Delta t} b^2,
\]
we obtain
\[
\begin{aligned}
\ell_k(\theta) 
&\geq (1 - \gamma \Delta t) \, \mathbb{E} \left|   V^k - \widehat{Y}_k  \right|^2 \\
&\quad - \frac{C}{\gamma \Delta t}  \sum_{l=k}^{N-1} \left( \mathbb{E}|V^l - \widehat{Y}_l|^2 + \mathbb{E}|\bar{Z}_l^k - \widehat{Z}_{k,l}|^2 \right) \Delta t^2 \\
&\quad + \Delta t \sum_{l=k}^{N-1} \mathbb{E} \left| \bar{Z}_l^k - \widehat{Z}_{k,l} \right|^2.
\end{aligned}
\]
Then, for \(\Delta t\) sufficiently small and an appropriate choice of \(\gamma\), the negative term can be controlled, yielding the simpler lower bound
\[
\ell_k(\theta) \geq (1 - C \Delta t) \, \mathbb{E} \left|   V^k - \widehat{Y}_k  \right|^2 +  \Delta t \sum_{l=k}^{N-1} \mathbb{E} \left| \bar{Z}_l^k - \widehat{Z}_{k,l} \right|^2.
\]

\subsection*{Step 4. Neural Regression Error Control}

Fix \( k \in \{0, \ldots, N-1\} \), \( l \in \{k, \ldots, N-1\} \).  
Let \( \theta^* = (\xi^*, \eta^*) \in \arg\min_{\theta} \ell_k(\theta) \) be the optimal neural parameters minimizing the loss.  
We define the corresponding network approximations:
\[
\widehat{Y}_k := U_l^k(X_k^N; \xi^*), \qquad \widehat{Z}_{k,l} := Z_l^k(X_l^N, X_k^N; \eta^*).
\]
From the upper and lower bounds on \( \ell_k(\theta) \), we obtain:
\begin{align*}
&( 1 - C \Delta t) \, \mathbb{E} \left|   V^k - \widehat{Y}_k  \right|^2 + C \Delta t \sum_{l=k}^{N-1} \mathbb{E}  \left| \bar{Z}_l^k - \widehat{Z}_{k,l} \right|^2  \leq \ell_k(\theta^*) \\ \leq \ell_k(\theta)
& \leq (1 + C\Delta t)\, \mathbb{E} \left|   V^k - \widehat{Y}_k  \right|^2  + \Delta t \sum_{l=k}^{N-1} \mathbb{E}  \left| \bar{Z}_l^k - \widehat{Z}_{k,l} \right|^2  \label{eq:step4-bound}
\end{align*}
For $\Delta t$ small enough this implies that for every $k \in \{0, \ldots, N-1\}$: \begin{equation}
    \begin{aligned}
         \mathbb{E} \left| V^k - \widehat{Y}_k \right|^2 +   \Delta t\;\sum_{l=k}^{N-1} \mathbb{E} \left| \bar{Z}_l^k - \widehat{Z}_{k,l} \right|^2
\leq
 \left( \varepsilon_k^{\text{NN}, y} + \Delta t \;  \sum_{l=k}^{N-1} \varepsilon_{k,l}^{\text{NN}, z} \right).
\label{boundNN}
    \end{aligned}
\end{equation}
Equation \eqref{boundNN} shows that the discrepancy between the learned variables and the ideal regression targets is controlled by the neural approximation errors, provided that the network minimizes the empirical loss function $\ell_k$ effectively.

Plugging this last inequality into \eqref{bound3} leads to:  
\[
\sum_{k=0}^{N-1} \mathbb{E} \left| Y_k^{k,N} - \widehat{Y}_k \right|^2 
\leq 
C N \sum_{k=0}^{N-1} \varepsilon_k^{\text{NN}, y}
+ C \sum_{k=0}^{N-1} \sum_{l=k}^{N-1} \varepsilon_{k,l}^{\text{NN}, z}.
\]
which proves the first bound in \eqref{boundfinal}.

\paragraph{Step 5. Convergence of the \( Z \)-Component.}

Using the triangular inequality on $Z_l^{k,N} - \widehat{Z}_{k,l}$, we write:
\[
\mathbb{E} \left| Z_l^{k,N} - \widehat{Z}_{k,l} \right|^2 
\leq 2 \mathbb{E} \left| Z_l^{k,N} - \bar{Z}_l^k \right|^2 
+ 2 \mathbb{E} \left| \bar{Z}_l^k - \widehat{Z}_{k,l} \right|^2.
\]
By using the bounds \eqref{boundZ2}, \eqref{bound3} and \eqref{boundNN} we obtain \begin{equation*}
\Delta t \; \mathbb{E}|Z_l^{k,N} - \bar{Z}_l^k|^2 \leq C \sum_{j=l}^{N-1} \mathbb{E} |V^j - \widehat{Y}_j|^2 \le C \left( \sum_{j=l}^{N-1}  \varepsilon_j^{\text{NN}, y} + \Delta t \;   \sum_{j=l}^{N-1}  \sum_{m=j}^{N-1} \varepsilon_{j,m}^{\text{NN}, z} \right), 
\end{equation*}
which leads to 
\begin{equation*}
    \begin{aligned}
       \Delta t \;  \sum_{k=0}^{N-1} \sum_{l=k}^{N-1} \mathbb{E}|Z_l^{k,N} - \bar{Z}_l^k|^2 \leq C N^2 \sum_{l=0}^{N-1} \varepsilon_l^{NN,y} + C N  \sum_{k=0}^{N-1} \sum_{l=k}^{N-1} \varepsilon_{k,l}^{\text{NN}, z}.
    \end{aligned}
\end{equation*}
While from \eqref{boundNN}
\[
\Delta t \; \sum_{k=0}^{N-1} \sum_{l=k}^{N-1} \mathbb{E} \left| \bar{Z}_l^k - \widehat{Z}_{k,l} \right|^2
\leq \sum_{k=0}^{N-1} \varepsilon_k^{\text{NN}, y}
+ \Delta t \; \sum_{k=0}^{N-1} \sum_{l=k}^{N-1} \varepsilon_{k,l}^{\text{NN}, z}.
\]
By putting this together, we prove the second part of the bound in \eqref{boundfinal}.
\end{proof}

\section{Applications}
\label{sec:numerics}
In this section, we present three representative applications that demonstrate the flexibility and effectiveness of our DeepBSVIE algorithm. The first example considers a recursive utility problem under ambiguity and memory, where the agent's preferences evolve over time and depend on past consumption. The second example involves a recursive valuation model under exponential growth and discounting, motivated by long term investment decisions in macroeconomic settings. The third example illustrates a nonlinear, time dependent recursive valuation problem with cyclical preferences. These case studies illustrate how BSVIEs can model non-Markovian features and time inconsistency, and how our solver can be effectively applied in such contexts.

\subsection{Recursive utility with memory and ambiguity} \label{sec:example}
In the spirit of Duffie and Epstein's \cite{duffie1992stochastic} continuous-time recursive utility, and following the BSVIE-based recursive utility formulations introduced by Yong \cite{yong2008well}, we consider an agent whose intertemporal preferences are described not by a standard BSDE, but by a BSVIE to account for nonlocal memory effects and ambiguity over long horizons. Specifically, let the agent's continuation utility $Y(t)$, be given by the following linear BSVIE
\[
Y(t) = \phi(t) + \int_t^T \Phi(t, s)\, Y(s)\, ds + \int_t^T  Z(t, s)\,\xi(s)\, ds - \int_t^T Z(t, s)\, dB(s),
\]
where $B(s) \in \mathbb{R}^d$ is a $d$-dimensional Brownian motion, $\phi(t) $ captures a time-dependent anticipatory reward based on the mean of terminal uncertainty, and $\Phi(t,s)$ incorporates fading memory. The integral term $\int_t^T Z(t,s) \xi(s)\, ds$ introduces a \emph{non-uniform ambiguity penalty} on volatility, thereby distorting the aggregation of future uncertainty in a forward-increasing manner. This formulation allows the agent to exhibit \emph{robust preferences} with growing concern about model misspecification at longer horizons, as in multiplier and multiple-priors utility models. The two-parameter nature of the process $Z(t, s) \in \mathbb{R}^{1 \times d}$ further encodes intertemporal sensitivity to risk, thus capturing a richer structure of hedging motives than classical BSDE models. Such a framework can be applied in robust intertemporal consumption models, long-horizon asset pricing, or climate finance, where agents exhibit memory, ambiguity aversion, and horizon-dependent preferences in the valuation of future outcomes. We fix
\[
\phi(t) := \sin(\pi t) \frac{1}{d}\sum_{i=1}^d B^i_T,
\quad
\Phi(t, s) = e^{-(s - t)} \mathbf{1}_{\{s \ge t\}}, \quad \xi(s) = (e^s, \dots, e^s)^\top \in \mathbb{R}^d,
\]
and follow \cite{hu2019linear} to  solve the BSVIE explicitly. 

\begin{remark}
    Notice that the driver of this BSVIE is linear in $Y$ and $Z$, the terminal term $\phi(t)$ is square-integrable, and all coefficients are Lipschitz in the spatial variables and H\"older-continuous in time. Therefore, this example satisfies all the assumptions required for the existence, uniqueness, and convergence of the numerical scheme described in Assumptions \eqref{ass:scheme_bsig} and  \eqref{ass:scheme_gf}.
\end{remark}

Define \(x = s - t\ge 0\), and write \(\Phi(t, s) = \rho(x)\) where:
\[
\rho(x) := e^{-x}, \quad x \ge 0.
\]
We define the resolvent kernel \(\Psi(t, s)\) as:
\[
\Psi(t, s) = \sum_{n=1}^{\infty} \Phi^{(n)}(t, s),
\]
where \(\Phi^{(n)}(t, s)\) is the \(n\)-fold convolution of \(\Phi\) with itself:
\[
\Phi^{(n)}(t, s) = \rho^{*n}(x) = \underbrace{(\rho * \cdots * \rho)}_{n \text{ times}}(x).
\]
It is known from convolution theory that:
\[
\rho^{*n}(x) = \frac{x^{n-1}}{(n-1)!} e^{-x}, \quad x \ge 0.
\]
Hence, we compute the series:
\[
\Psi(x) = \sum_{n=1}^\infty \rho^{*n}(x) = e^{-x} \sum_{n=1}^\infty \frac{x^{n-1}}{(n-1)!}
= e^{-x} \sum_{k=0}^\infty \frac{x^k}{k!} = e^{-x} \cdot e^x = 1.
\]
Therefore, we conclude:
\[
\Psi(t, s) = 1 \quad \text{for all } r \ge t.
\]
Thus, the solution is:
\[
Y(t) = \mathbb{E}^{\mathbb{Q}} \left[ \phi(t) + \int_t^T \phi(s)\, ds \ \middle| \ \mathcal{F}_t \right],
\]
where \(\mathbb{Q}\) is the Girsanov measure associated with \(\xi(s)\):
\[
\frac{d\mathbb{Q}}{d\mathbb{P}} = \exp\Bigg( - \int_0^T \xi(s) dB(s) - \frac12 \int_0^T \|\xi(s)\|^2 \, ds \Bigg), 
\quad dB(s)^{\mathbb{Q}} = dB(s) + \xi(s)\, ds.
\]
Substituting \(\phi(t)\), we have:

\[
Y(t) = \mathbb{E}^{\mathbb{Q}} \Bigg[ \frac{1}{d} \sum_{i=1}^d B(T)^i \, \mathbf{1}_d \left( \sin(\pi t) + \int_t^T \sin(\pi r)\, dr \right) \Big| \mathcal{F}_t \Bigg].
\]
Using the fact that \( \mathbb{E}^{\mathbb{Q}}[B(T) \mid \mathcal{F}_t] = B(t) + \int_t^T \xi(s)\, ds \), the final solution is:
\begin{equation}
Y(t) = \left( \sin(\pi t) + \int_t^T \sin(\pi s)\, ds \right)   \left( \frac{1}{d}\sum_{i=1}^d  B(t)^i + e^T - e^t \right).
\label{eq:solY}
\end{equation}
Moreover, the $i-$th component of process \( Z(t,s) \in \bR^d\) is given by:
\[
Z^i(t,s) = \mathbb{E}^{\mathbb{Q}} \Big[ [D_s \phi(t)]_i + \int_t^T \Phi(t,r)\, D_s Y(r) \, dr \,\Big|\, \mathcal{F}_s \Big].
\]
with 
\[
[D_s \phi(t)]_i = D_s \big[ \sin(\pi t) \cdot B(T)^i \big] = \frac{1}{d}\sin(\pi t),
\]
and \[
D_s Y(r) = \frac{1}{d}\left( \sin(\pi r) + \int_r^T \sin(\pi u)\, du \right) \mathbf{1}_{\{s \le r\}}.
\]
Plugging it into the representation of \( Z^i(t,s) \)
\begin{equation*}
    \begin{aligned}
Z^i(t,s) &= \mathbb{E}^{\mathbb{Q}} \left[ \sin(\pi t) + \int_t^T e^{-(r - t)}\, D_s Y(r)\, dr \,\middle|\, \mathcal{F}_s \right] \\
&= \frac{1}{d} \left( \sin(\pi t) + \int_{s}^T e^{-(r - t)} \left( \sin(\pi r) + \int_r^T \sin(\pi u)\, du \right) dr  \right ) , \qquad i= 0,\ldots,d
    \end{aligned}
\end{equation*}
so that \(Z(t,s) = (Z^1(t,s), \dots, Z^d(t,s)) \in \mathbb{R}^{1 \times d}\).
\subsubsection{Numerical Implementation}
\label{sec:numerical_implementation}

We now present the numerical results obtained by applying the neural network-based algorithm described in Section \ref{sec:algo} to the example in Section \ref{sec:example} with dimension $d=5$, which admits explicit expressions and thus provides a convenient benchmark to evaluate the accuracy of our numerical method. The network was trained on a GPU (AMD Instinct MI250X), and the typical training time for the presented example was approximately $9$ minutes.

Figure \ref{fig:Y_vs_true} compares the learned solution \(\widehat{Y}_i, i=0,\ldots,N\) with the analytical expression \(Y\) derived in Equation \eqref{eq:solY} evaluated at the discrete time-steps $t_i, i = 0, \ldots, N$. Figure \ref{fig:Z_plots} provides three complementary views of first component of the learned kernel \( Z \): the left subplot shows the surface plot of  \( \widehat{Z}^1_{i,j} \) compared to the reference surface \( Z^1(t,s) \) evaluated on the discrete grid points \( \{(t_i, t_j) : 0 \leq i \leq j < N\} \). The central subplot compares the sequences of \( (\widehat{Z}^1_{i,j})_{i=0}^j \) with the corresponding values \( (Z^1(t_i, t_j))_{i=0}^j \) for selected values of \( j \), illustrating how the learned kernel evolves along the first time axis for fixed values of $s$. The right subplot presents the analogous comparison of \( (\widehat{Z}^1_{i,j})_{j=i}^{N-1} \) with \( (Z^1(t_i, t_j))_{j=i}^{N-1} \) for selected values of \( i \), showing the evolution of the kernel across the other time axis. The remaining components of $Z$ display similar qualitative behavior, and their representation is therefore omitted; all components are, however, taken into account in the error plots.
 \\

Figure \ref{fig:loss_subplots} shows the evolution of the training loss across iterations at selected time steps. Due to the localized nature of the algorithm, the loss is computed separately at each iteration based on the current time step. We therefore report representative time steps only, as the loss at other time steps displays qualitatively similar behavior.\\

With the network parameters selected as described in Section \ref{sec:algo}, the resulting approximation errors, measured according to the metrics defined in Equations \eqref{eq:err} and \eqref{eq:rel_err}, are summarized in Table \ref{tab:err}.  To provide further insight into the temporal behavior of the approximation errors, Figure \ref{fig:error_visualization} depicts the evolution of the mean squared errors as functions of time. The left plot shows the mean squared error for the \(Y\)-component, defined as the expected squared difference between the true and estimated \(Y\) averaged over sample paths and plotted over time. The right plots display focus on the full two-dimensional error surface for the matrix valued process \(Z(t_i, t_j)\) over the time grid and all the spatial components, where the domain is restricted to \(j \geq i\) to respect the inherent causality and triangular structure of the time indices.

\begin{table}[h!]
\centering
\begin{tabular}{|c|c|c|}
\hline
 & $L^2$ Error & $L^2$ Relative Error \\
\hline
$Y$ & $2.74 \times 10^{-4}$ & $1.75 \times 10^{-4}$ \\
$Z$ & $3.08 \times 10^{-5}$ & $4.20 \times 10^{-4}$ \\
\hline
\end{tabular}
\caption{$L^2$ Error and Relative Error for $Y$ and $Z$}
\label{tab:err}
\end{table}

\begin{figure}[H]
\centering
\includegraphics[width=0.7\textwidth]{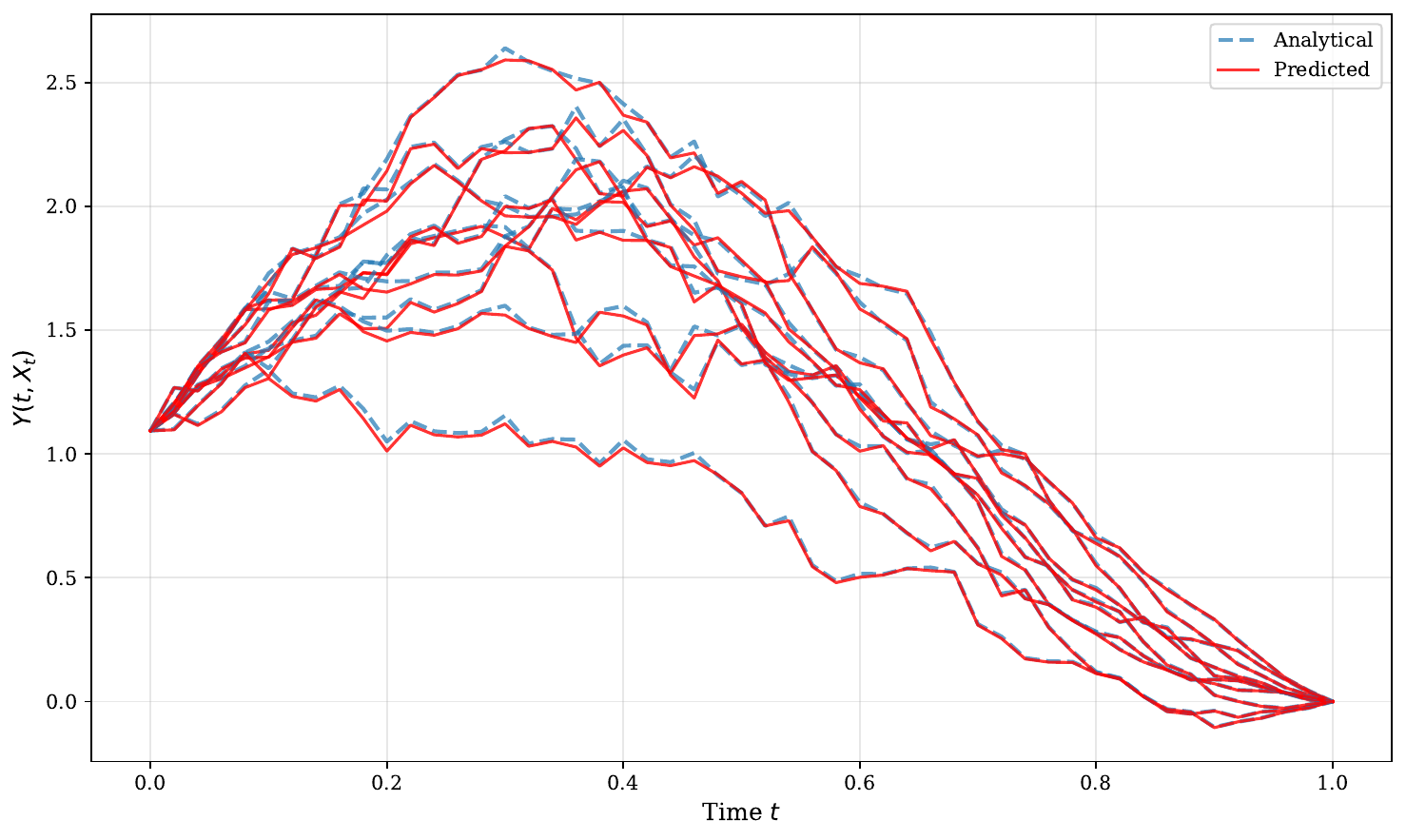}
\caption{Comparison of the learned \(Y(t)\) (red) with the analytical solution (dashed blue).}
\label{fig:Y_vs_true}
\end{figure}

\begin{figure}[htbp]
    \centering
    \begin{subfigure}[t]{0.32\textwidth}
        \includegraphics[width=0.9\textwidth]{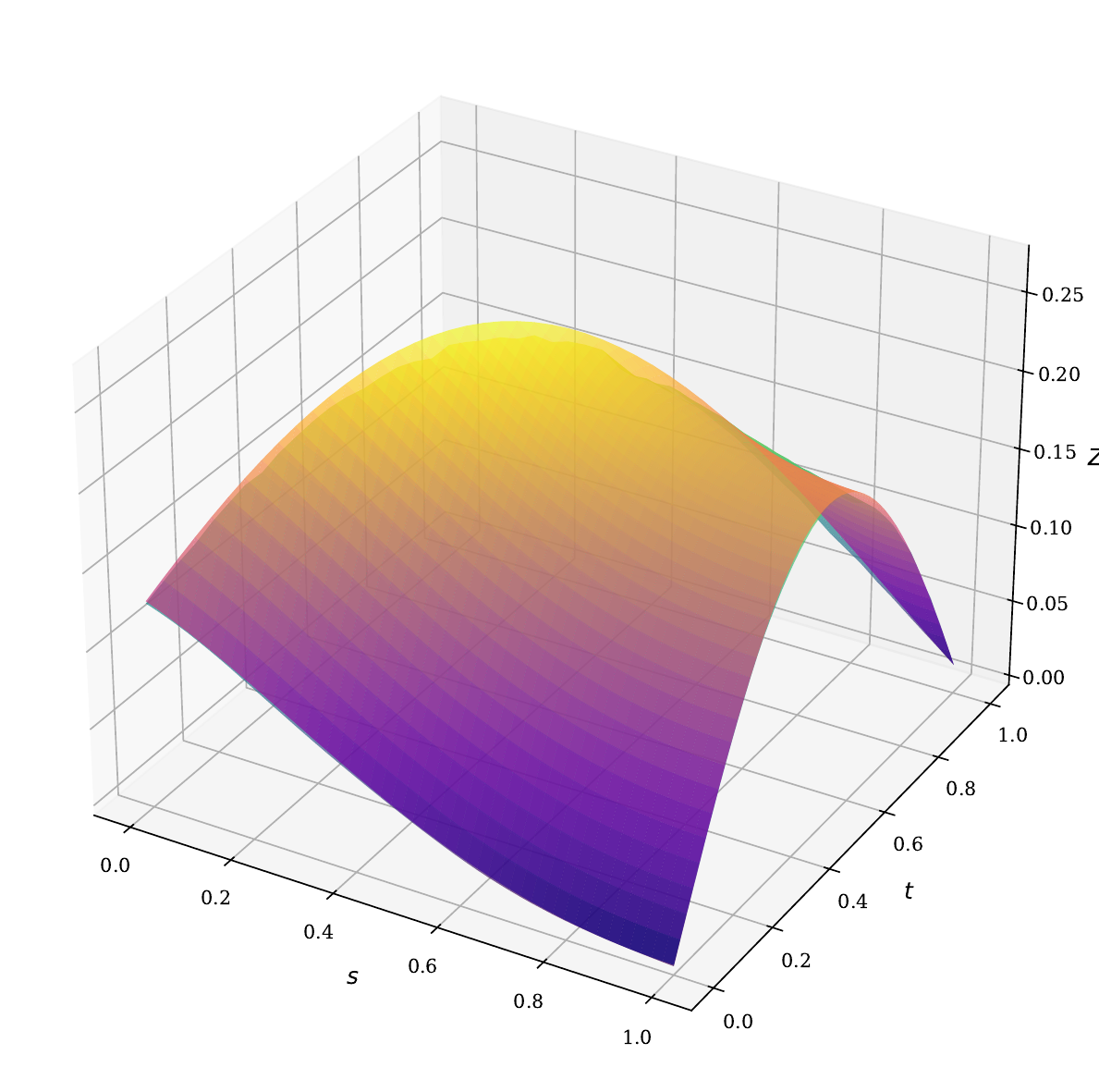}
        \caption{Surface plot of the learned kernel \( Z^1_{i,j} \), overlapped with the reference surface \( Z^1(t,s) \) on the evaluation grid.}
        \label{fig:kernel_surface}
    \end{subfigure}
    \hfill
    \begin{subfigure}[t]{0.32\textwidth}
        \includegraphics[width=\textwidth]{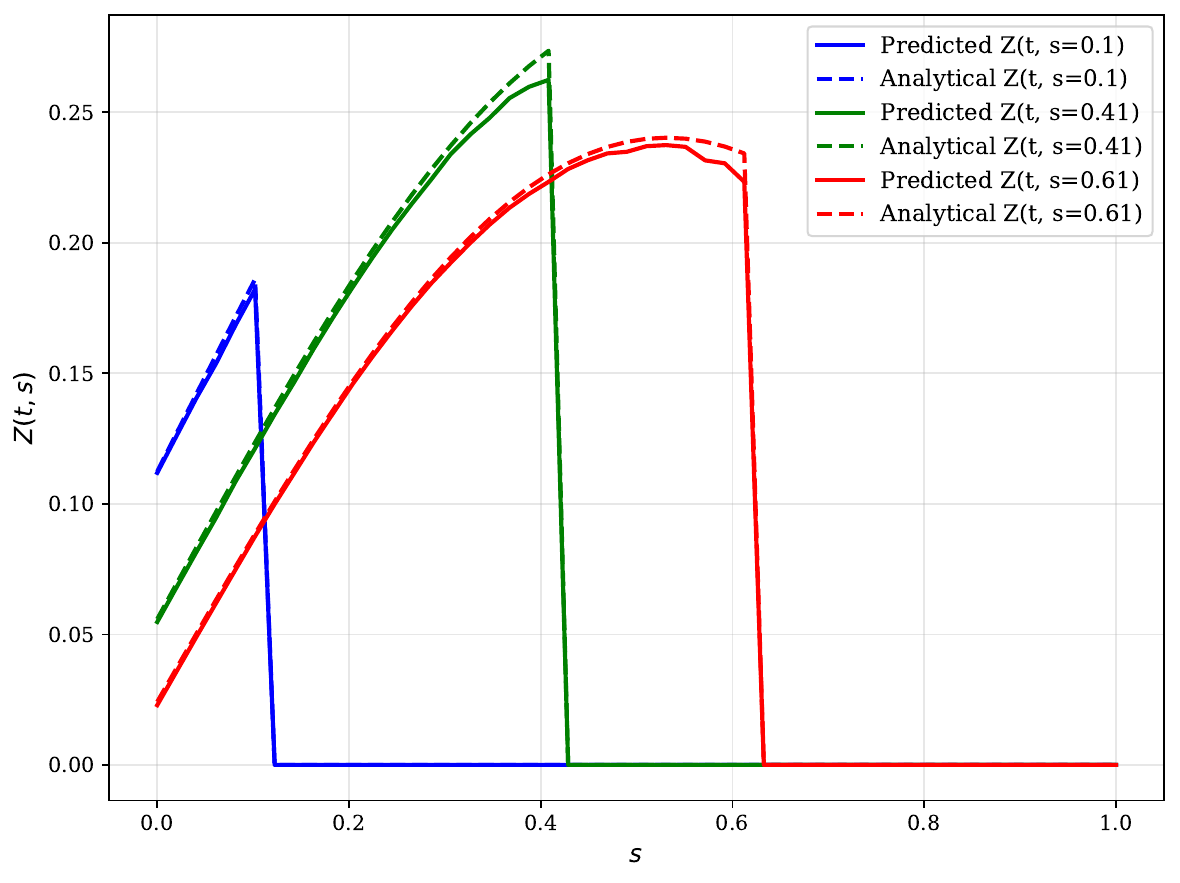}
        \caption{Comparison of $(\widehat{Z}^1_{i,j})_{i=0}^j$  with the corresponding values of $(Z^1(t_i,t_j))_{i=0}^j$ for selected values of $j$.}
        \label{fig:evolution_fixed_s}
    \end{subfigure}
    \hfill
    \begin{subfigure}[t]{0.32\textwidth}
        \includegraphics[width=\textwidth]{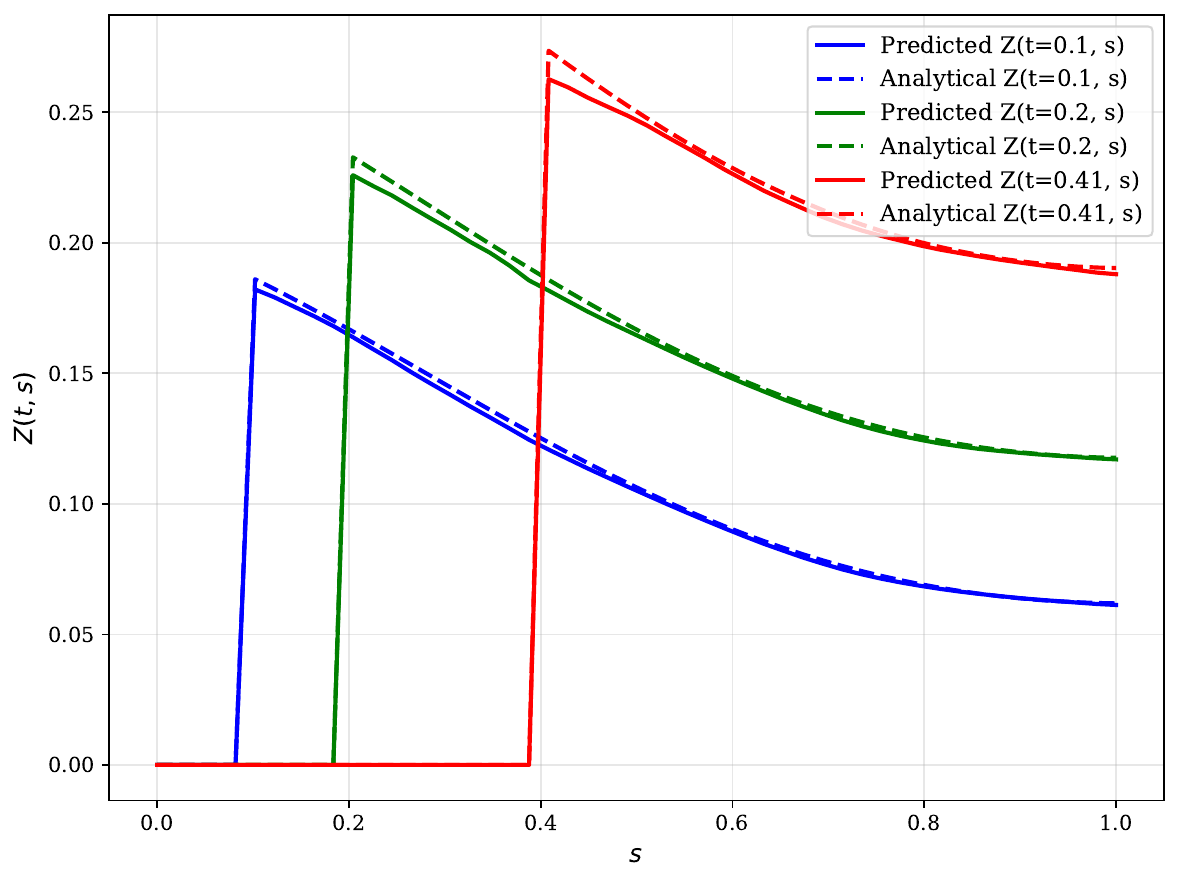}
        \caption{Comparison of $(\widehat{Z}^1_{i,j})_{j=i}^{N-1}$  with the corresponding values of $(Z^1(t_i,t_j))_{j=i}^{N-1}$ for selected values of $i$.}
        \label{fig:evolution_fixed_t}
    \end{subfigure}
    \caption{Visualization of the structure of the first component of the learned kernel \( Z_{i,j} \) and its continuous counterpart \( Z(t,s) \).}
    \label{fig:Z_plots}
\end{figure}

\begin{figure}[H]
    \centering
    \begin{adjustbox}{max width=\textwidth}
    \begin{subfigure}[t]{0.32\textwidth}
        \centering
      \includegraphics[width=\linewidth]{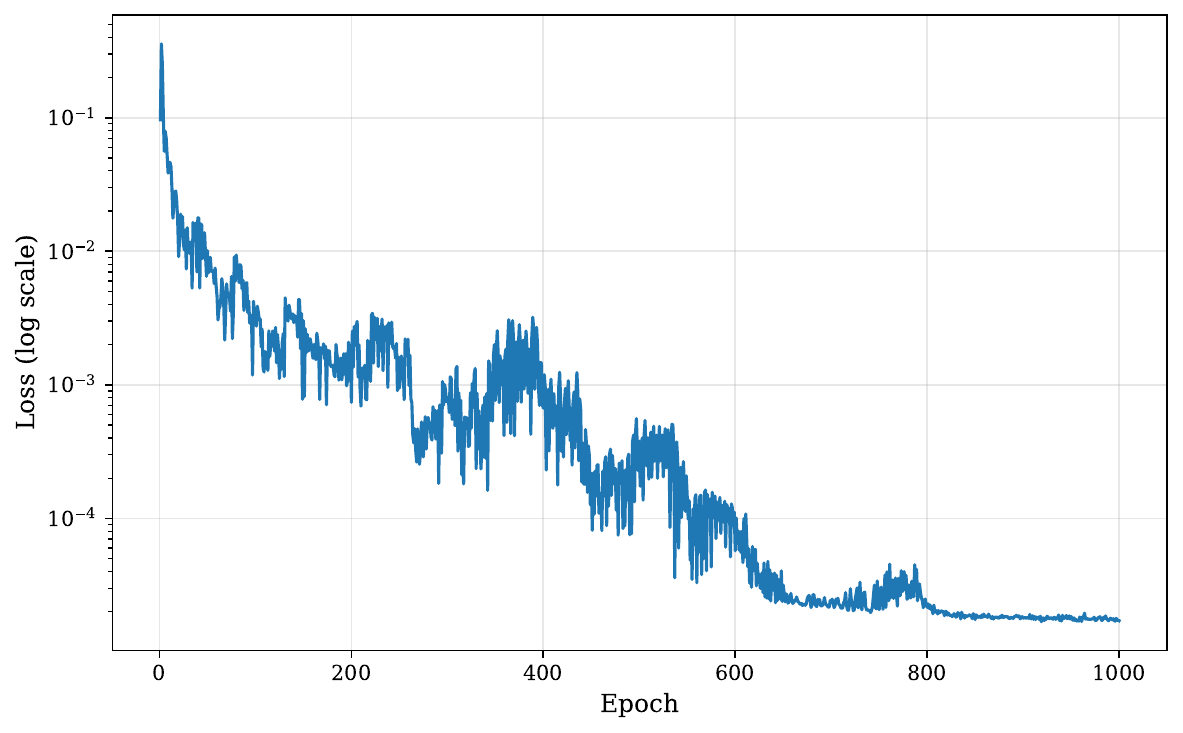}
        \caption{Loss at time step \( 50 \).}
        \label{fig:loss_tNminus1}
    \end{subfigure}
    \hspace{0.01\textwidth}
    \begin{subfigure}[t]{0.32\textwidth}
        \centering
        \includegraphics[width=\linewidth]{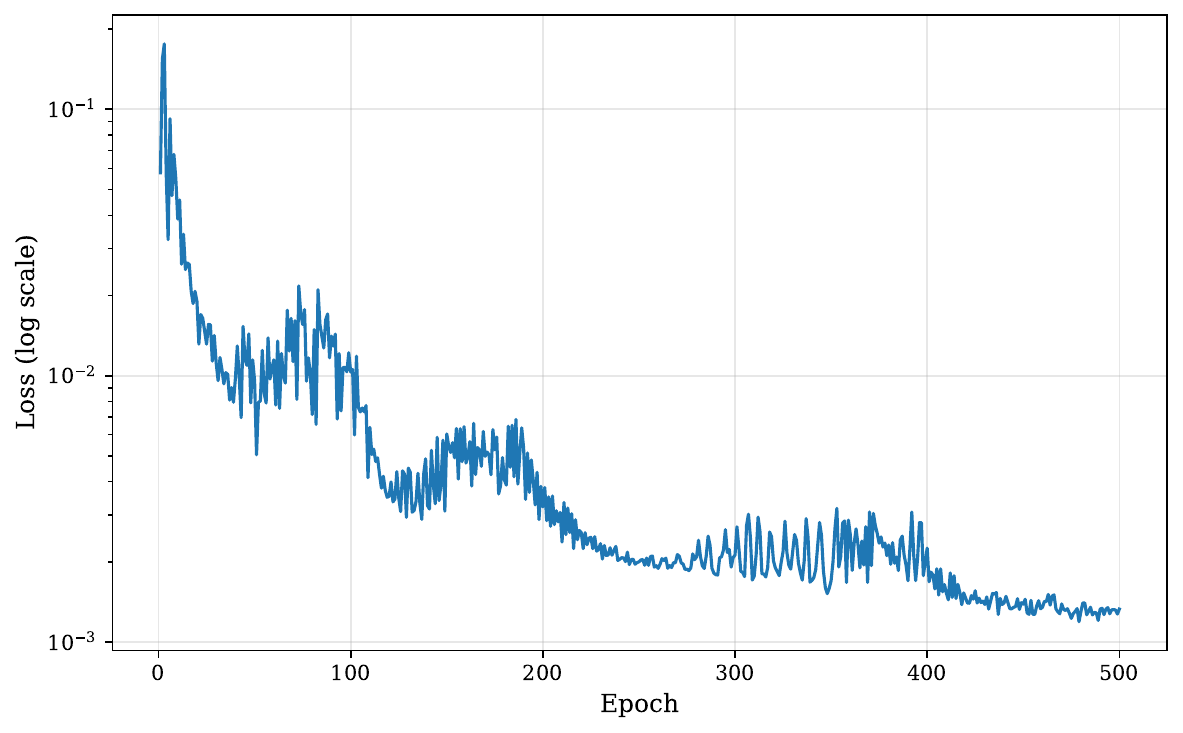}
        \caption{Loss at time step \( 25 \).}
        \label{fig:loss_tN-2}
    \end{subfigure}
    \hspace{0.01\textwidth}
    \begin{subfigure}[t]{0.32\textwidth}
        \centering
        \includegraphics[width=\linewidth]{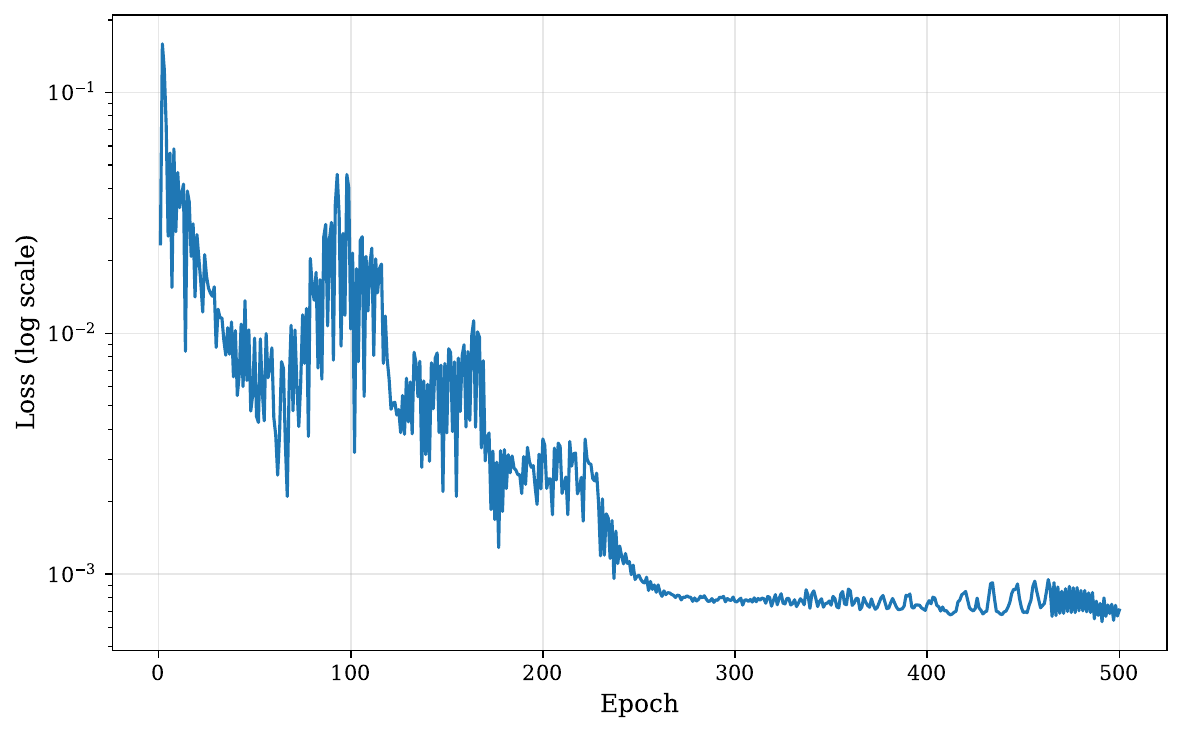}
        \caption{Loss at time step \( 0 \).}
        \label{fig:loss_t0}
    \end{subfigure}
     \end{adjustbox}
    \caption{Algorithm loss across iterations evaluated at selected time steps.}
    \label{fig:loss_subplots}

\end{figure}

\begin{figure}[H]
    \centering
    \begin{subfigure}[t]{0.32\textwidth}
        \centering
        \includegraphics[width=\linewidth]{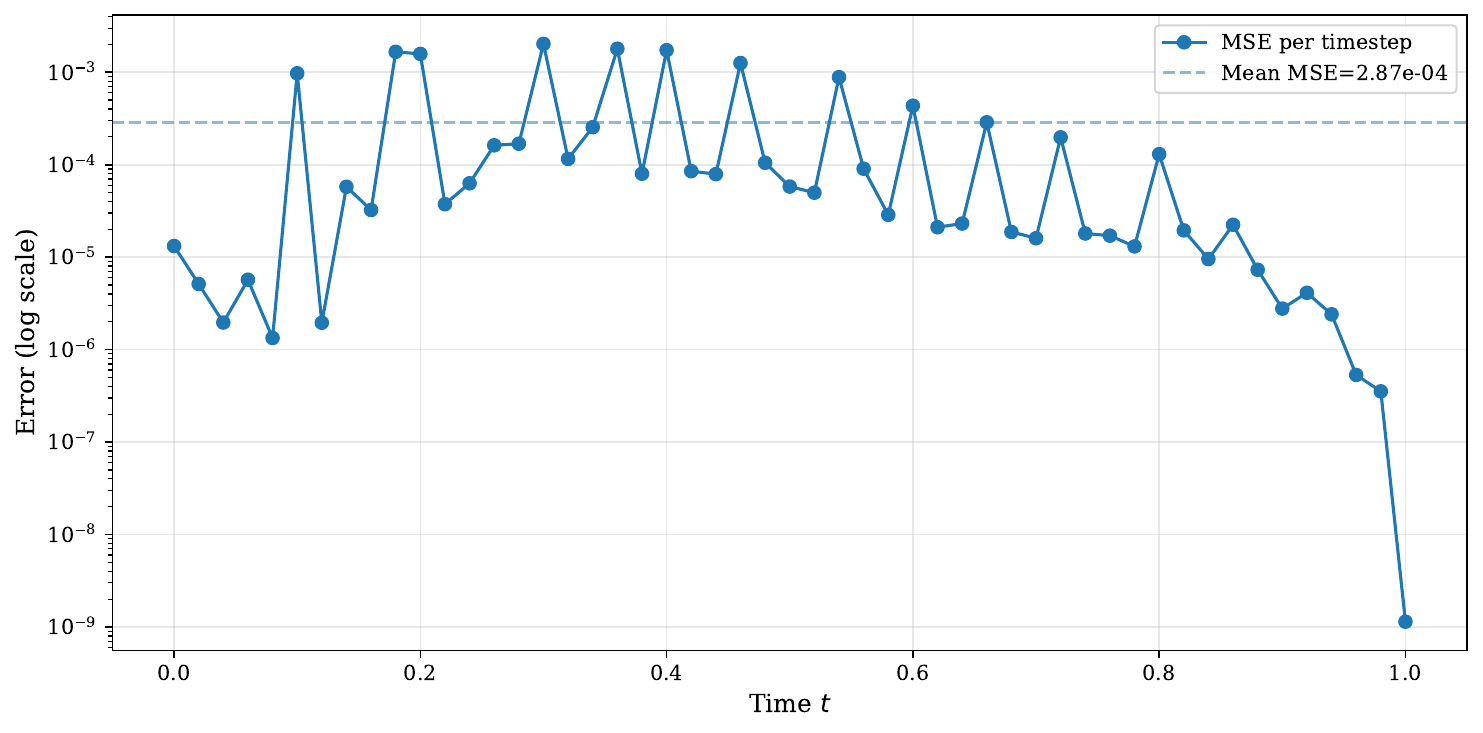}
        \caption{MSE of \( \widehat{Y} \) vs. \( Y \) per timestep.}
        \label{fig:mse_Y}
    \end{subfigure}
    \hfill
    \begin{subfigure}[t]{0.32\textwidth}
        \centering
        \includegraphics[width=0.9\linewidth]{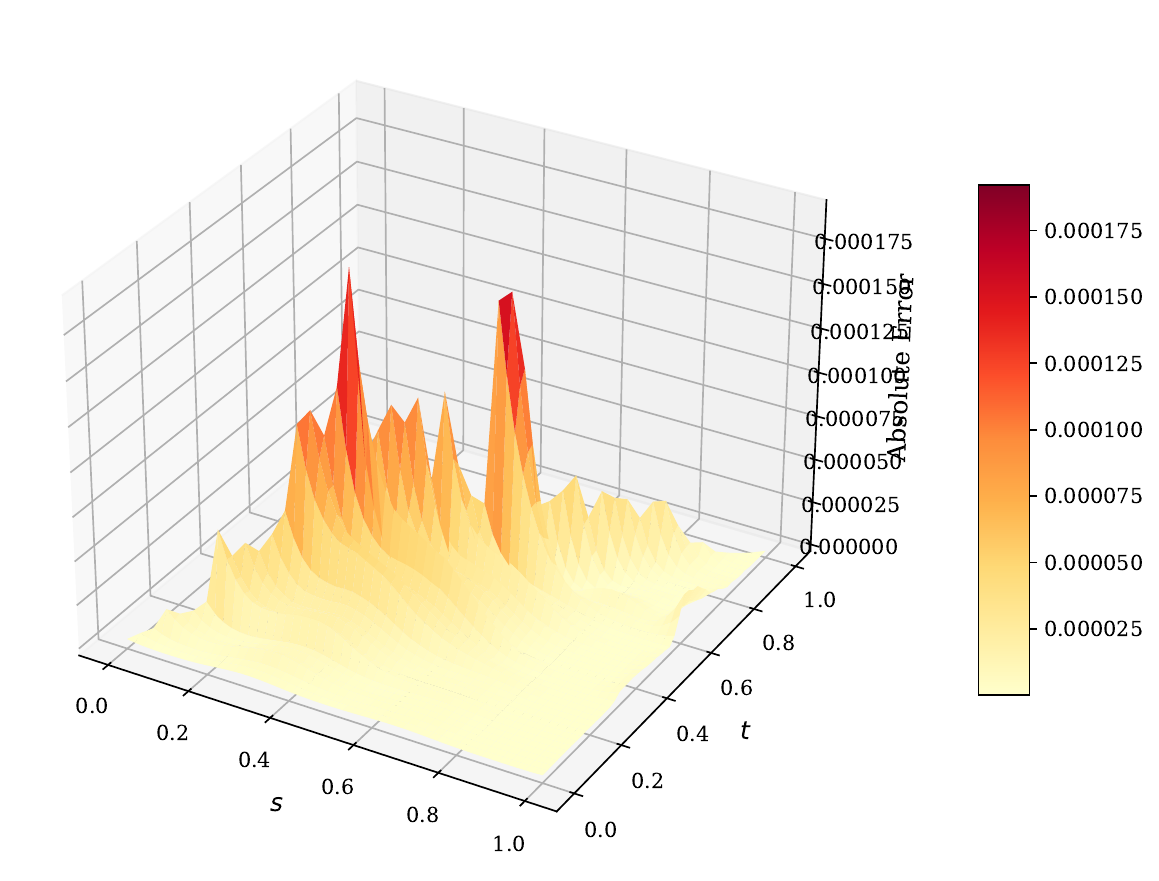}
        \caption{Surface plot of MSE between \( \widehat{Z} \) and \( Z \).}
        \label{fig:mse_Z_surface}
    \end{subfigure}
    \hfill
    \begin{subfigure}[t]{0.32\textwidth}
        \centering
        \includegraphics[width=0.9\linewidth]{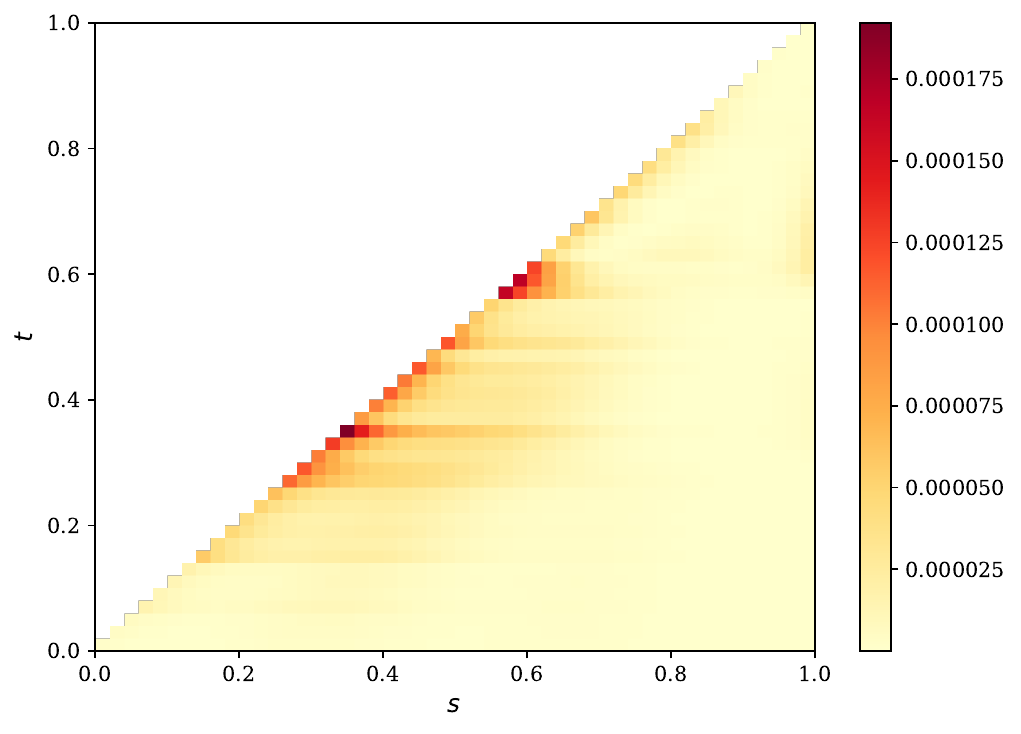}
        \caption{Heatmap of MSE between \( \widehat{Z} \) and \( Z \).}
        \label{fig:mse_Z_heatmap}
    \end{subfigure}
    \caption{Visualization of estimation errors.}
    \label{fig:error_visualization}
\end{figure}

\subsection{Recursive Valuation under Exponential Growth and Discounting}
\label{sec:example2}
In long term financial planning and asset management, agents often face the challenge of evaluating future income streams and terminal rewards under uncertainty. A natural approach is to model their preferences using recursive utility, where today's utility depends on future utility in a dynamically consistent way. When the underlying uncertainty is driven by market exposure, such preferences can be linked to the evolution of risky assets via forward-backward systems.

To account for such dynamics, we consider a BSVIE, where the forward component models exponential growth and the backward component captures discounted future rewards and their sensitivity to risk.

Let the forward state variable represent the evolution of $d$ risky assets modeled as geometric Brownian motions:
\begin{equation*} 
X(t) = \exp\Big( (\mu - \tfrac{1}{2} \operatorname{diag}(\sigma \sigma^\top)) t + \sigma B(t) \Big) \in \mathbb{R}^d,
\end{equation*}
where $\mu \in \mathbb{R}^d$ is the drift vector, $\sigma \in \mathbb{R}^{d \times d}$ is the diagonal volatility matrix, and $(B(t))_{t \ge 0} \in \mathbb{R}^d$ is a standard $d$-dimensional Brownian motion.

We now consider the following BSVIE for the scalar continuation utility $Y(t)$:
\begin{equation*}
Y(t) = e^{-\lambda t} \frac{1}{d} \sum_{i=1}^d X^i(t) + \int_t^T \lambda_0 \frac{1}{d} \sum_{i=1}^d X^i(s)\, ds - \int_t^T Z(t,s) dB(s),
\end{equation*}
where $\lambda \ge 0$ is a discount rate, $\lambda_0 \in \mathbb{R}$ governs the weight of intermediate income flows, and $Z(t,s) \in \mathbb{R}^{1 \times d}$ encodes the sensitivity of $Y(t)$ to each component of $B(s)$.

\begin{remark}
    Notice that in this example, the driver of the BSVIE is linear in the state and independent of the solution, and therefore trivially satisfies the Lipschitz condition in these variables. The terminal term is square-integrable, and both the driver and terminal term satisfy the H\"older continuity in time. Moreover, the forward process exhibits linear growth, and all coefficients are measurable. Therefore, this example satisfies all the assumptions required for the existence, uniqueness, and convergence of the numerical scheme presented  in Assumptions  \eqref{ass:scheme_bsig} and  \eqref{ass:scheme_gf}.
\end{remark}

This BSVIE reflects the recursive valuation of future utility from a vector of risky assets: the agent receives a discounted terminal reward given by the mean of all assets, accumulates intermediate benefits proportional to the mean of the state vector $X(s)$, and dynamically hedges future uncertainty through the vector-valued process $Z(t,s)$. The two-parameter structure of $Z(t,s)$ captures the marginal sensitivity of the utility at time $t$ with respect to randomness revealed in each asset at future times $s \ge t$, a feature not captured by standard BSDEs.

Such BSVIEs extend the Duffie-Epstein recursive utility framework by incorporating memory and horizon-dependent discounting. They are particularly relevant in long-term investment and consumption models where agents exhibit forward-looking behavior with intertemporal hedging motives.\\
Using the Markov property of the multidimensional state \(X(t) \in \mathbb{R}^d\), we compute the explicit representation of \(Y(t)\) as:
\begin{equation*} 
Y(t) = \mathbb{E} \left[ \frac{1}{d} \sum_{i=1}^d e^{-\lambda t} X^i(t) + \int_t^T \lambda_0 \frac{1}{d} \sum_{i=1}^d X^i(s) \, ds \,\Big|\, \mathcal{F}_t \right].
\end{equation*}

Since \(\mathbb{E}[X^i(s) \mid \mathcal{F}_t] = X^i(t) e^{\mu_i (s-t)}\) for each \(i = 1, \dots, d\), we obtain the closed-form expression:
\begin{equation} \label{eq:Yexplicit}
\begin{aligned}
Y(t) &= \frac{1}{d} \sum_{i=1}^d X^i(t) \left( e^{-\lambda t} e^{\mu_i (T-t)} + \lambda_0 \frac{e^{\mu_i (T-t)} - 1}{\mu_i} \right)
\end{aligned}
\end{equation}
To determine the stochastic integrand \( Z(t,s) \), we apply the Malliavin representation:
\[
Z(t,s) = \mathbb{E} \left[ D_s Y(t) \,\middle|\, \mathcal{F}_s \right], \quad t \le s \le T.
\]
We compute the Malliavin derivatives for each component \(i\) are:
\[
D_s X^i(t) = (\sigma X(T))_i \mathbf{1}_{\{s \le T\}}, \quad
D_s X_u^i = (\sigma X_u)_i \mathbf{1}_{\{s \le u\}}, \quad u \in [t,T].
\]
which yield:
\begin{equation*}
\begin{aligned}
    Z^i(t,s) =&  \frac{\sigma_i}{d}\mathbb{E} \left[   e^{-\lambda t} X^i(t) + \int_s^T \lambda_0 X_u^i \, du  \,\Big|\, \mathcal{F}_s \right]\\
=& \frac{\sigma_i}{d} \, X^i(s) \left( e^{-\lambda t} e^{\mu_i (T-s)} + \lambda_0 \frac{e^{\mu_i (T-s)} - 1}{\mu_i} \right),
\end{aligned}
\end{equation*}
so that \(Z(t,s) = (Z^1(t,s), \dots, Z^d(t,s)) \in \mathbb{R}^{1 \times d}\).

This expression captures the marginal impact of noise at time \( s \) on the utility \( Y(t) \), consistent with the anticipative nature of BSVIEs.

\subsubsection{Numerical Implementation} 

We now present the numerical results obtained by applying the neural network-based algorithm described in Section \ref{sec:numerical_implementation} to the example introduced in Section \eqref{sec:example2}. For the numerical simulations, we consider a time horizon \( T = 1 \), initial condition \( X_0 = 1.0 \), $\lambda = \lambda_0 = \frac{1}{2}$.
The drift coefficient is specified by
\[
\mu_i = \mu_{\text{base}} \left( 1 + 0.3  \frac{i-1}{d-1} \right), \quad i = 1, \dots, d,
\]
and the diagonal entries of the volatility matrix are given by
\[
\sigma_i = \sigma_{\text{base}} \left( 1 + 0.2  \frac{i-1}{d-1} \right), \quad i = 1, \dots, d.
\]
This defines a deterministic linear variation of drift and volatility across dimensions. In what follows, we fix $d=5$ and therefore consider
\(\mu = (0.07, 0.085, 0.1, 0.115, 0.13)\) and
\(\sigma = \text{diag}(0.4, 0.45, 0.5, 0.55, 0.6)\).

To approximate the solution of the BSVIE, we employ the neural network-based method trained using the same architecture and hyperparameters described in Section \ref{sec:algo}, where we detailed the general learning setup for BSVIEs. The training follows Algorithm \ref{algo:BSVIE}.
The forward geometric Brownian motion is generated using the Euler-Maruyama scheme. The backward processes \( Y(t) \), \( Z(t,s) \), are learned simultaneously through the neural representation, trained over multiple trajectories to ensure statistical robustness. The network was trained on a GPU (AMD Instinct MI250X), and the typical training time for the presented example was approximately $8$ minutes and $48$ seconds.\\

Below, we report the key outcomes of the simulation. Figure \ref{fig:Y2} compares the  learned solution \(\widehat{Y}_i, i = 0, \ldots, N\) with the analytical expression of \(Y\) derived in Equation \eqref{eq:Yexplicit}, evaluated at discrete time steps \(t_i\). Figure \ref{fig:surf2} shows a comparison between the first component of the learned kernel \(\widehat{Z}(t_i, t_j)\) and the reference solution \(Z(t_i, t_j)\), both evaluated over the domain \(0 \leq i \leq j < N\). Figure \ref{fig:fixedk} further compares the first component of the sequences \((\widehat{Z}_{i,j})_{i=0}^j\) with their reference counterparts \((Z(t_i, t_j))_{i=0}^j\),  illustrating how the learned kernel varies along the first time axis for fixed values of $s=t_j$.  Similarly, \ref{fig:fixedn} shows a comparison between the first component of \( (\widehat{Z}_{i,j})_{j=i}^{N-1}\) and \( (Z(t_i,t_j))_{j=i}^{N-1}\) for selected values of $i$, showing slices of the both kernels for fixed values of $t = t_i$.  \\

Figure \ref{fig:loss_subplots2} illustrates the evolution of the training loss across epochs for selected time steps. Table \ref{tab:err2} reports the approximation errors, evaluated according to the metrics defined in Equations \eqref{eq:err} and \eqref{eq:rel_err}. To complement these results, Figure \ref{fig:error_visualization2} displays the temporal evolution of the mean squared errors for both the \( Y \) and \( Z \) components. Notice that the slightly higher MSE near $t=1$ arises from several factors. First, compared to Example~1, the terminal payoff is more complex and therefore harder to learn (e.g., compare Figure \ref{fig:mseY2} with the corresponding error pattern in Figure \ref{fig:mse_Y}). Moreover, due to the backward nature of the scheme, the final time steps are inherently the most challenging to learn, as there is no warm-start initialization at the terminal network. Nevertheless, the MSE remains on the order of $10^{-4}$ and does not accumulate going backward in time, confirming the stability of the training procedure.

\begin{table}[h!]
\centering
\begin{tabular}{|c|c|c|}
\hline
 & $L^2$ Error & $L^2$ Relative Error \\
\hline
$Y$ & $2.31 \times 10^{-5}$ & $2.21 \times 10^{-5}$ \\
$Z$ & $6.25 \times 10^{-5}$ & $3.12 \times 10^{-4}$ \\
\hline
\end{tabular}
\caption{$L^2$ Error and Relative Error for $Y$ and $Z$}
\label{tab:err2}
\end{table}

\begin{figure}[H]
\centering
\includegraphics[width=0.7\textwidth]{ 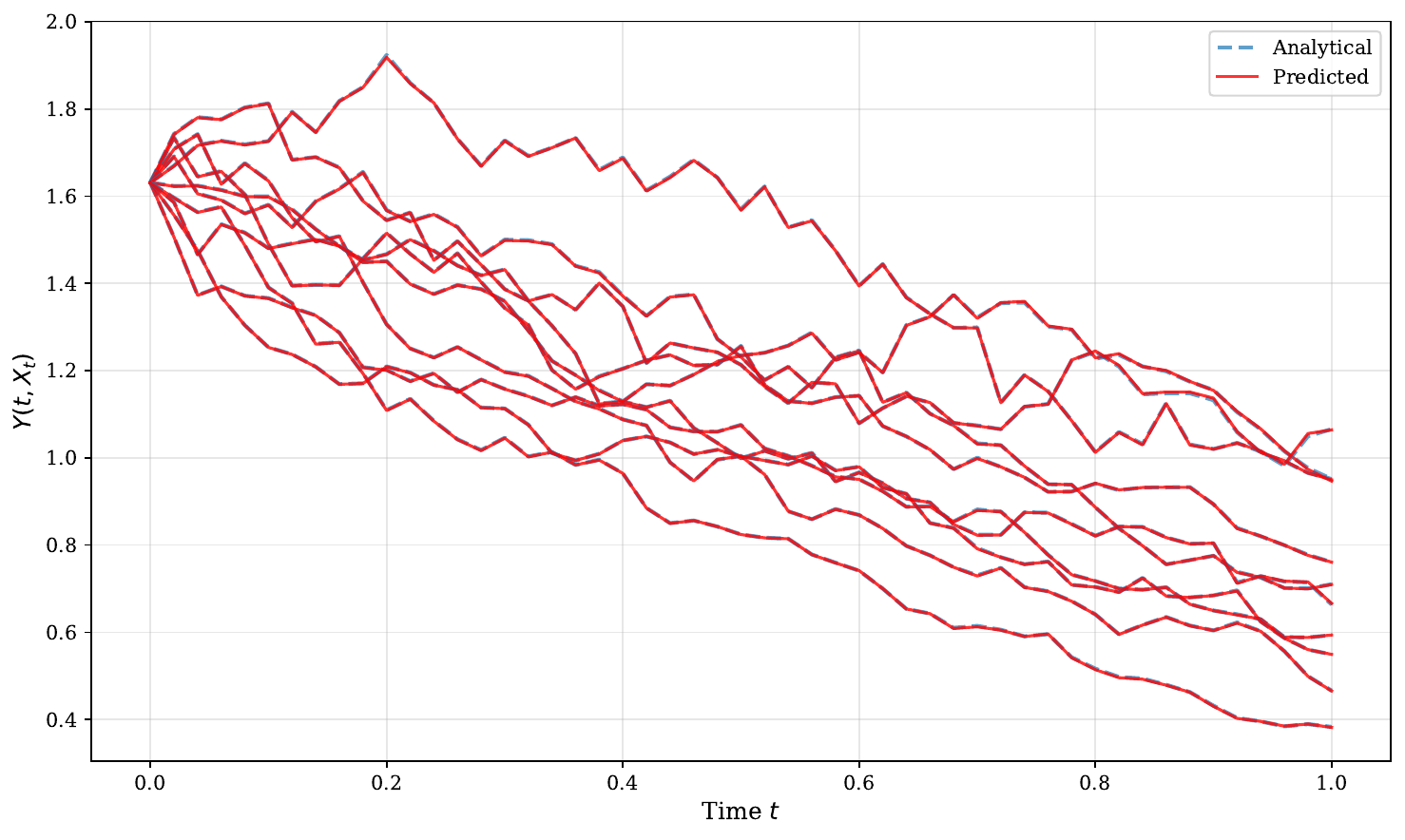}
\caption{Comparison of the learned \(Y(t)\) (red) with the analytical solution (dashed blue).}
\label{fig:Y2}
\end{figure}

\begin{figure}[H]
    \centering
    \begin{subfigure}[t]{0.32\textwidth}
        \includegraphics[width=\textwidth]{ 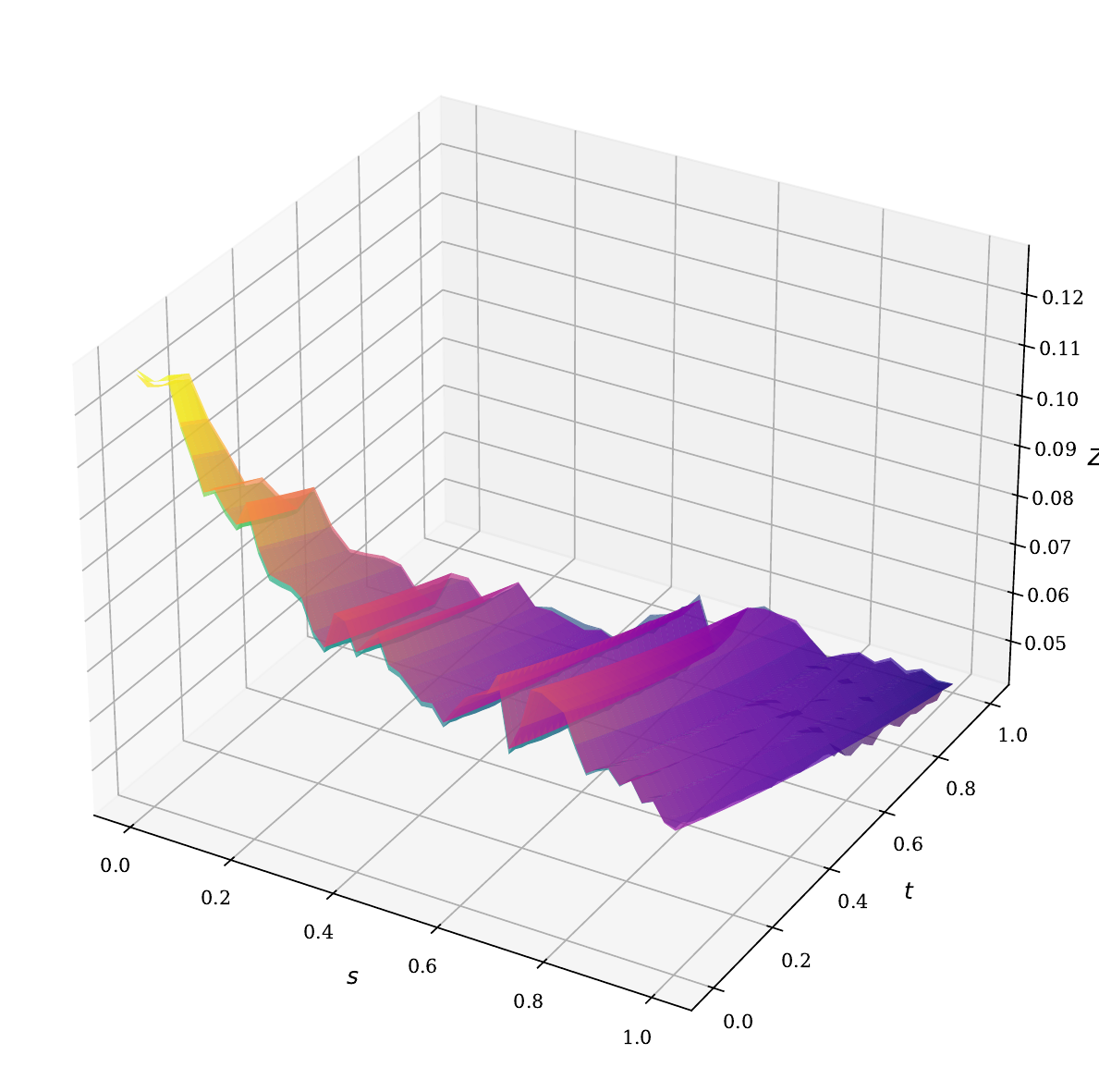}
        \caption{Surface plot of the learned kernel \( Z^1_{i,j} \), overlapped with the reference surface \( Z^1(t,s) \) on the evaluation grid.}
        \label{fig:surf2}
    \end{subfigure}
    \hfill
    \begin{subfigure}[t]{0.32\textwidth}
        \includegraphics[width=\textwidth]{ 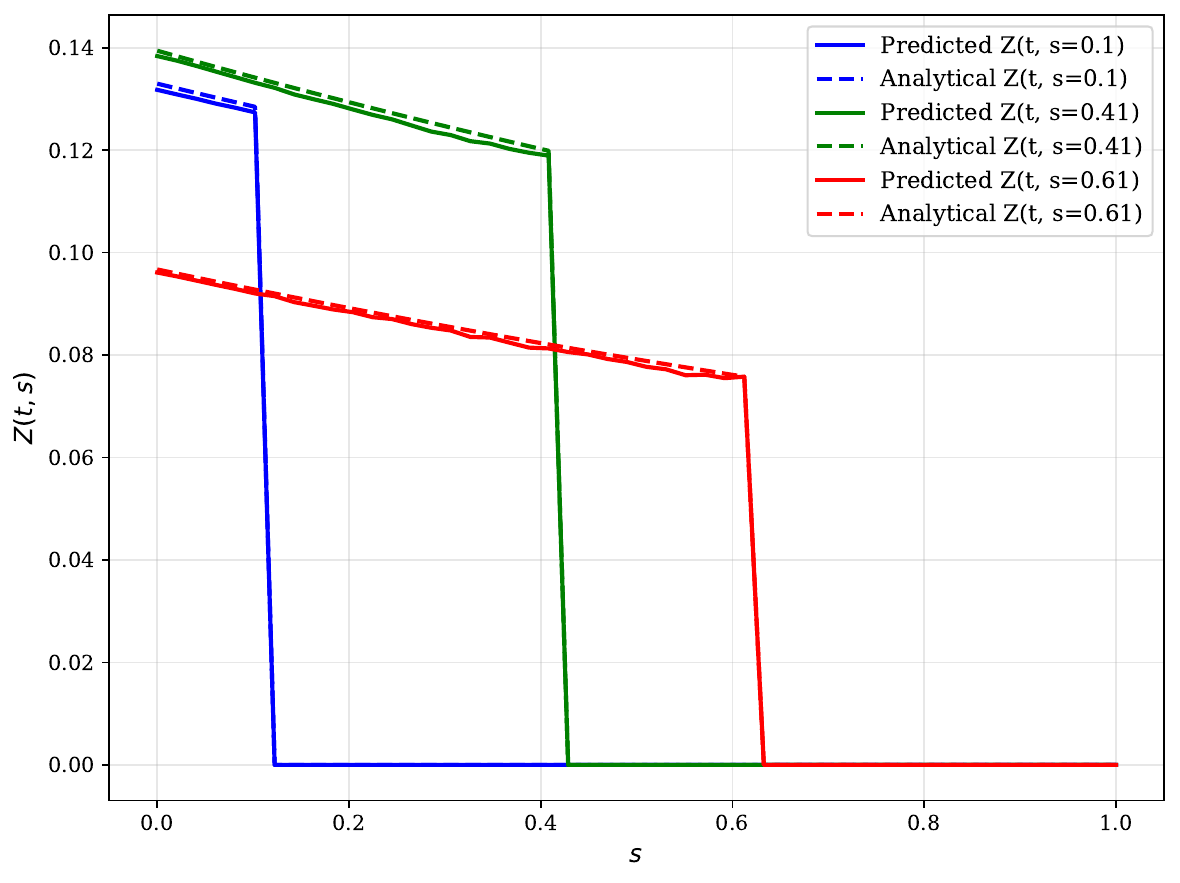}
        \caption{Comparison of $(\widehat{Z}^1_{i,j})_{i=0}^j$  with the corresponding values of $(Z^1(t_i,t_j))_{i=0}^j$ for selected values of $j$.}
        \label{fig:fixedk}
    \end{subfigure}
    \hfill
    \begin{subfigure}[t]{0.32\textwidth}
        \includegraphics[width=\textwidth]{ 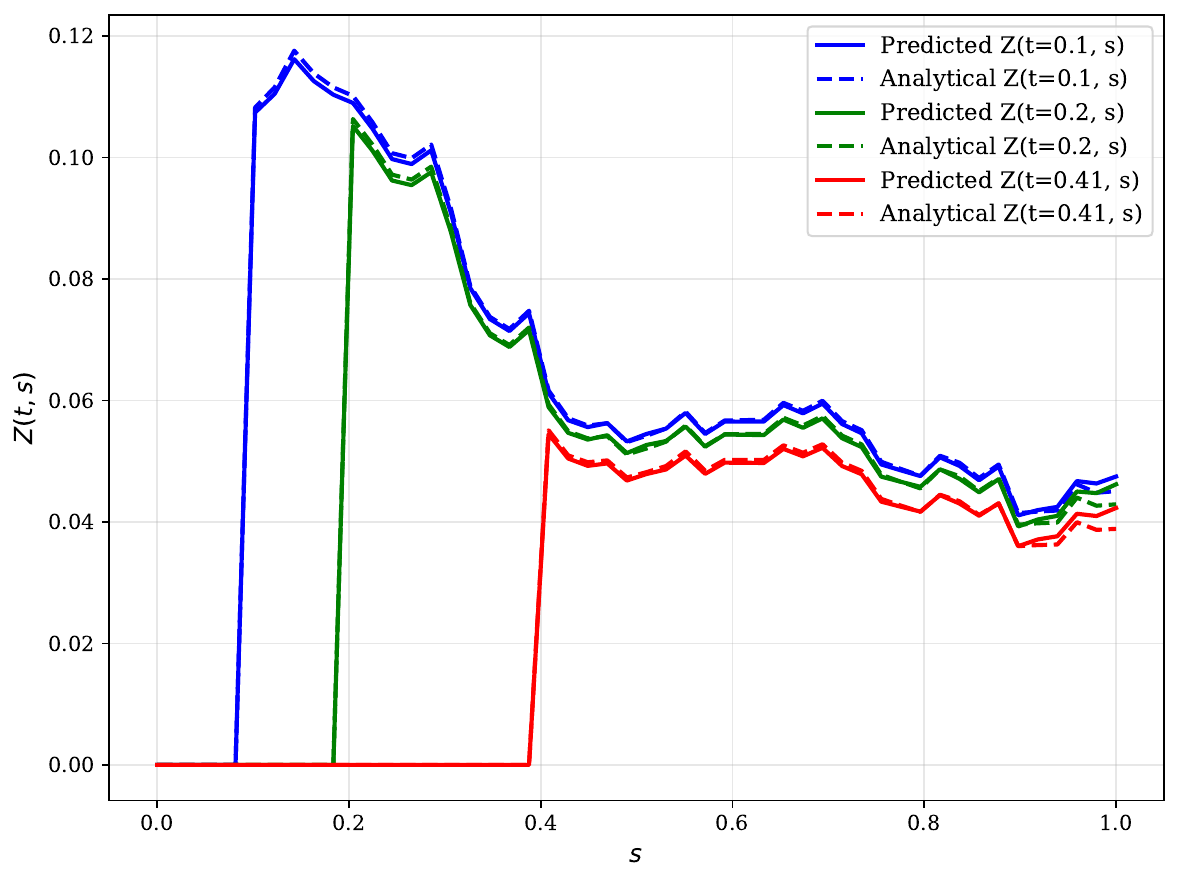}
        \caption{Comparison of $(\widehat{Z}^1_{i,j})_{j=i}^{N-1}$  with the corresponding values of $(Z(^1t_i,t_j))_{j=i}^{N-1}$ for selected values of $i$.}
        \label{fig:fixedn}
    \end{subfigure}
    \caption{Visualization of the structure of the first component of the learned kernel \( Z_{i,j} \) and its continuous counterpart \( Z(t,s) \).}
    \label{fig:Z_plots2}
\end{figure}

\begin{figure}[H]
    \centering
    \begin{adjustbox}{max width=\textwidth}
    \begin{subfigure}[t]{0.32\textwidth}
        \centering
      \includegraphics[width=\linewidth]{ 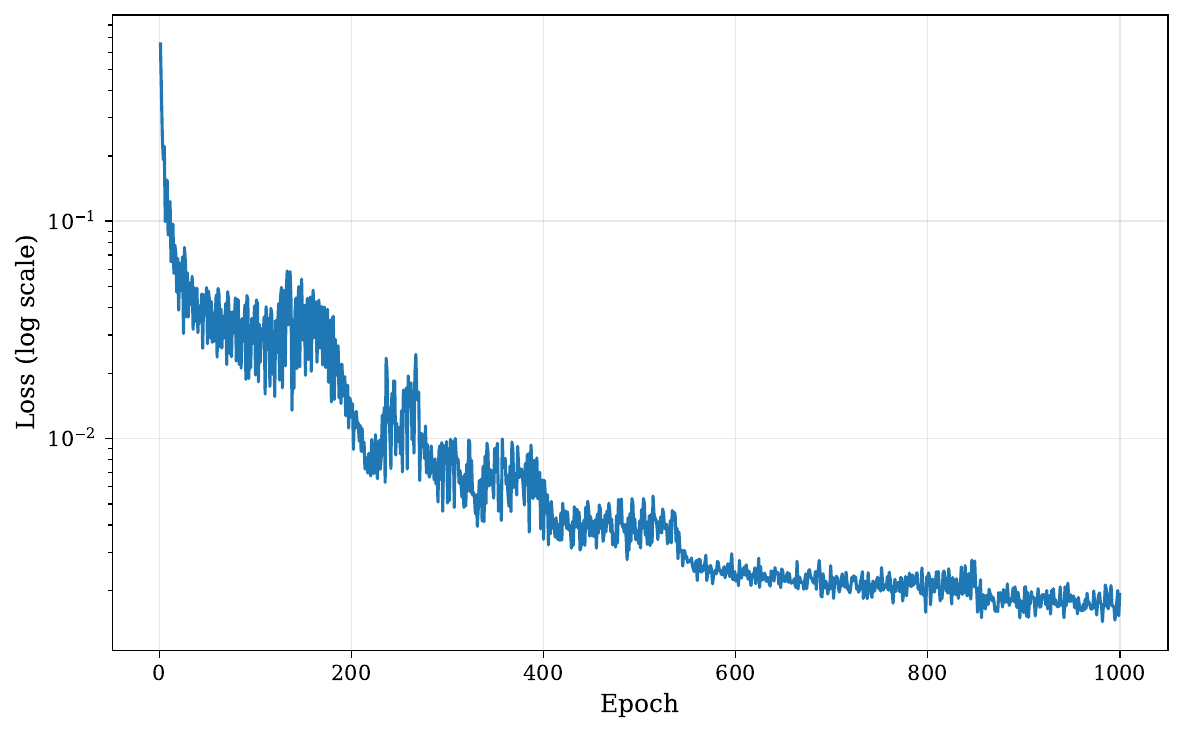}
        \caption{Loss at time step \( 50 \).}
    \end{subfigure}
    \hspace{0.01\textwidth}
    \begin{subfigure}[t]{0.32\textwidth}
        \centering
        \includegraphics[width=\linewidth]{ 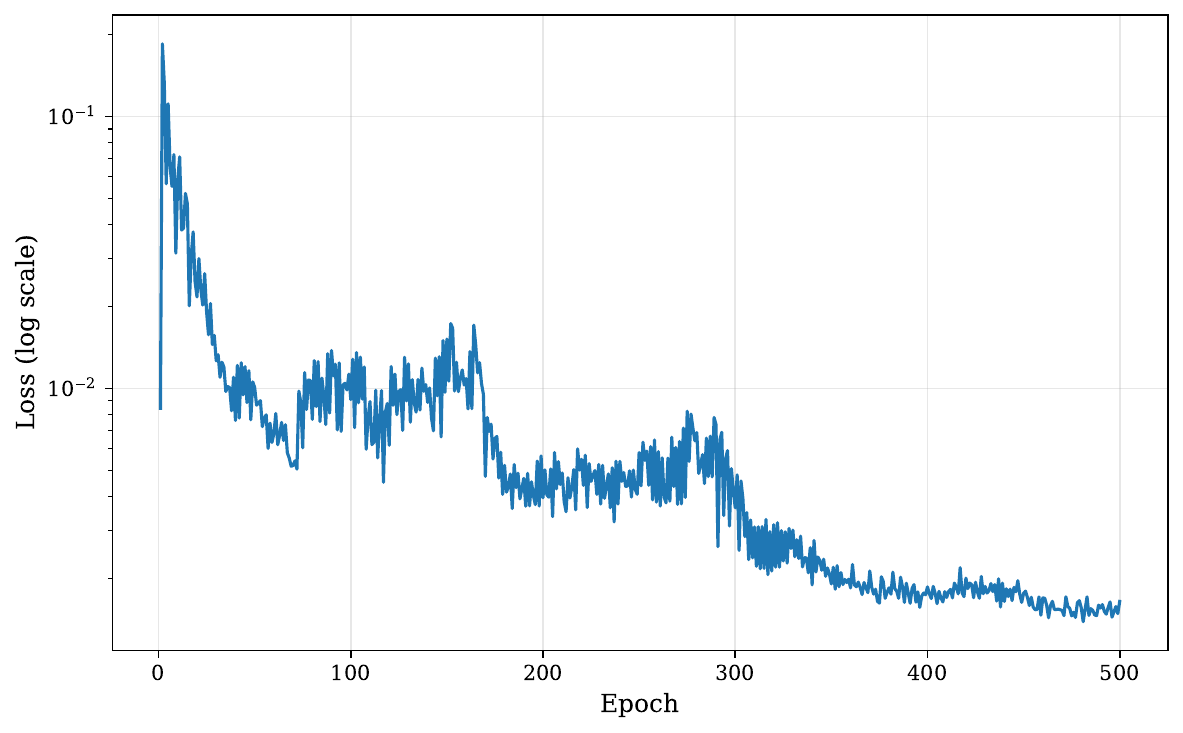}
        \caption{Loss at time step \( 25 \).}
    \end{subfigure}
    \hspace{0.01\textwidth}
    \begin{subfigure}[t]{0.32\textwidth}
        \centering
        \includegraphics[width=\linewidth]{ 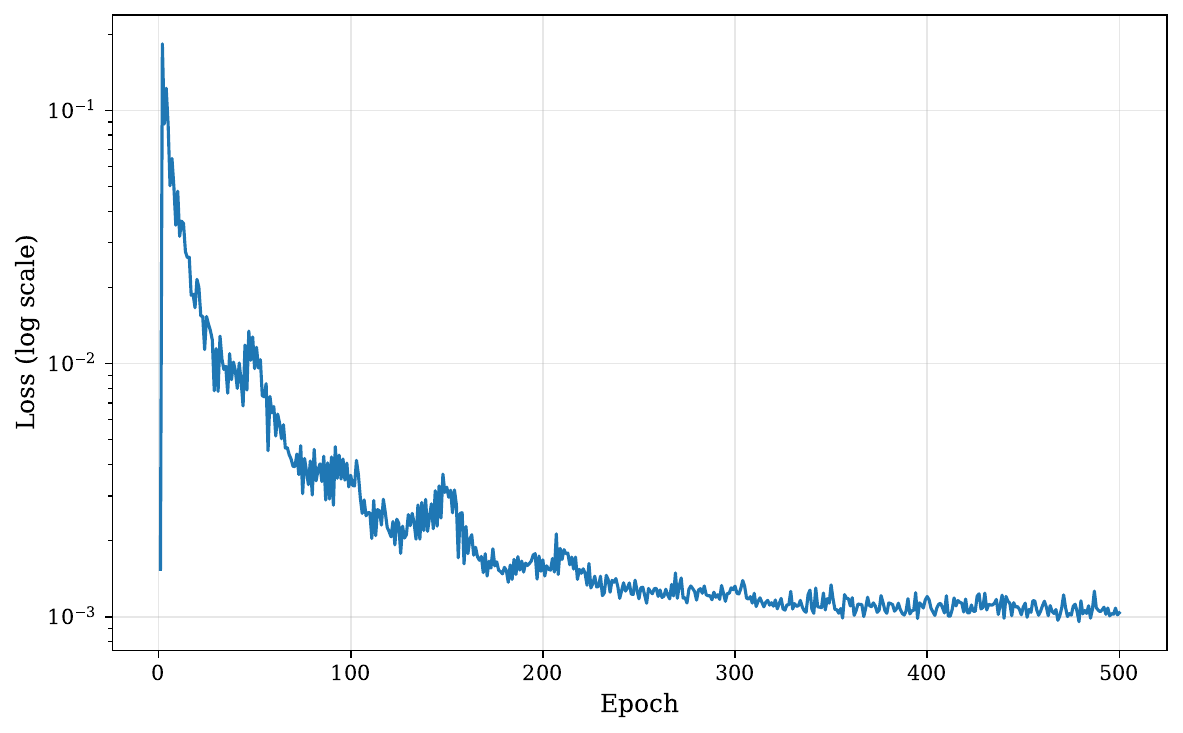}
        \caption{Loss at time step \( 0 \).}
    \end{subfigure}
     \end{adjustbox}
    \caption{Algorithm loss across iterations evaluated at selected time steps.}
    \label{fig:loss_subplots2}

\end{figure}

\begin{figure}[H]
    \centering
    \begin{subfigure}[t]{0.32\textwidth}
        \centering
        \includegraphics[width=\linewidth]{ 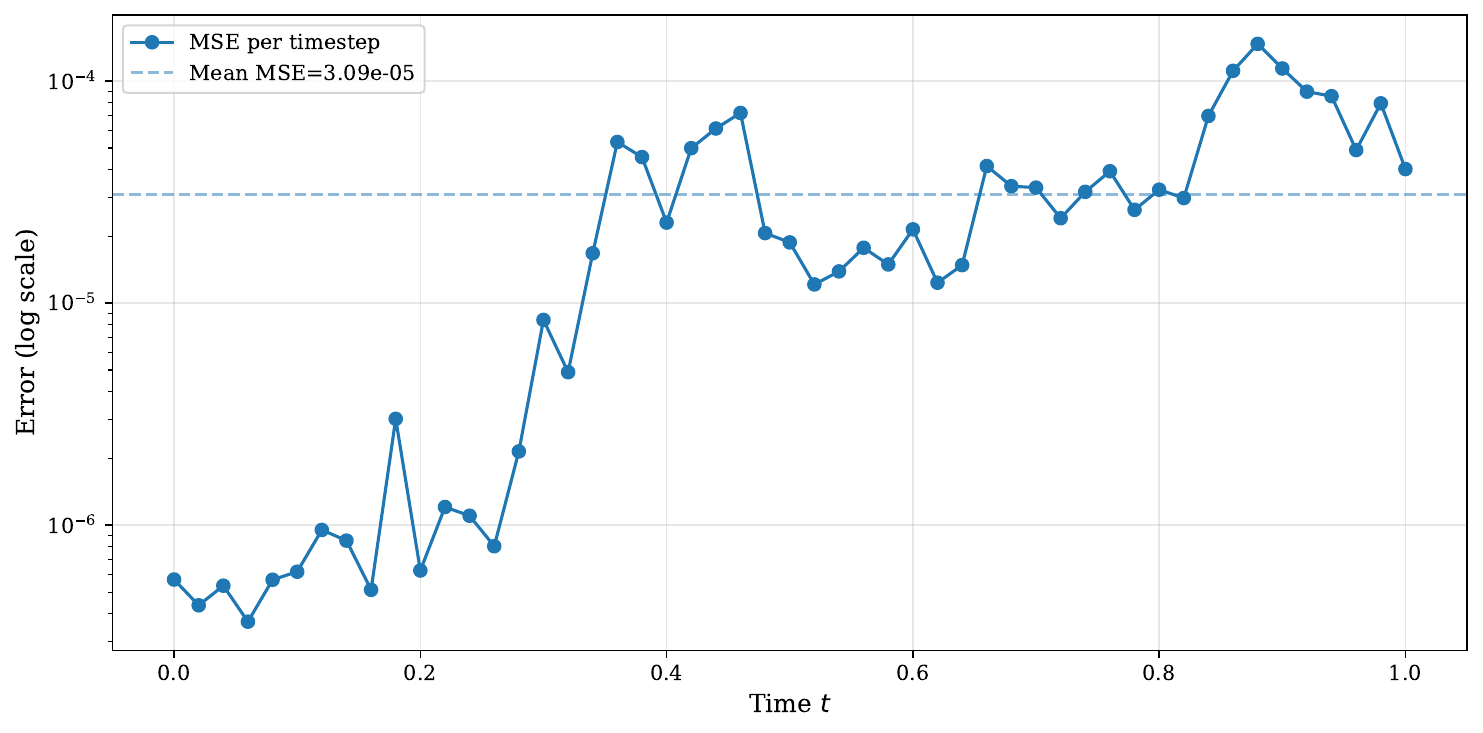}
        \caption{MSE of \( \widehat{Y} \) vs. \( Y \) over time.}
        \label{fig:mseY2}
    \end{subfigure}
    \hfill
    \begin{subfigure}[t]{0.32\textwidth}
        \centering
        \includegraphics[width=0.9\linewidth]{ 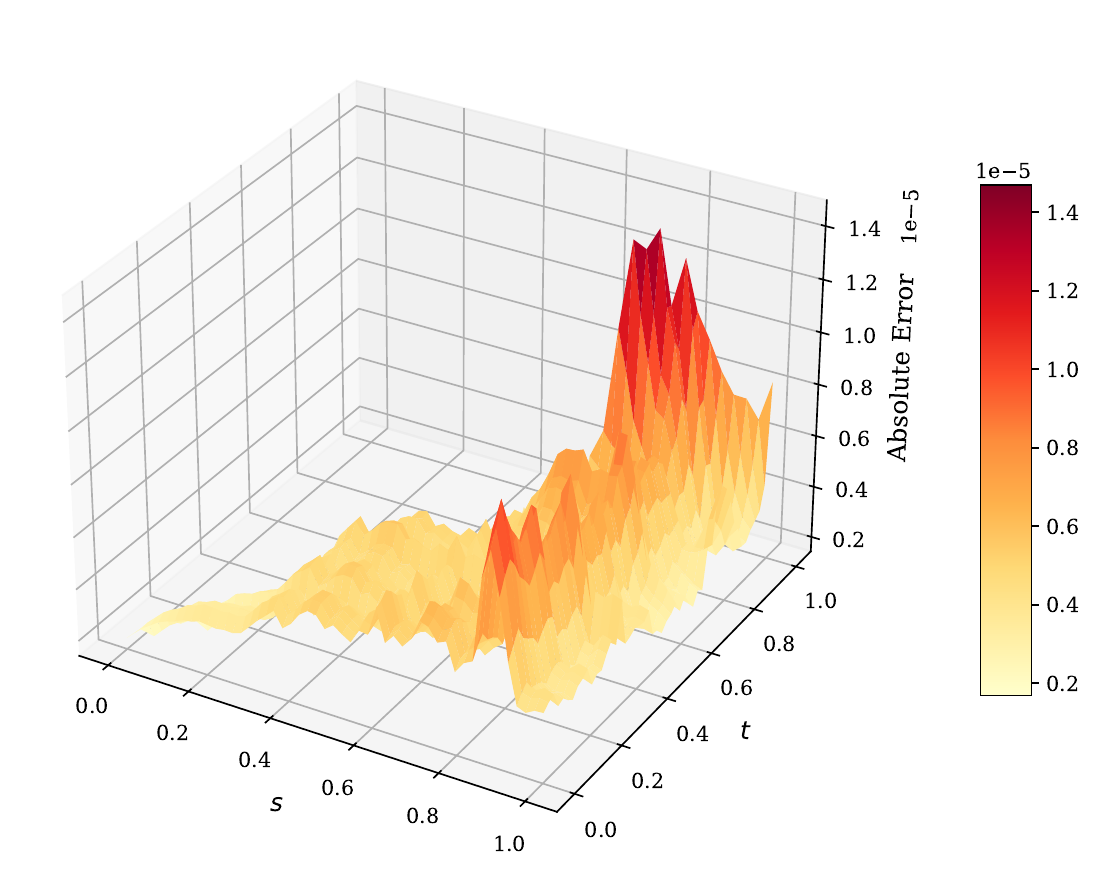}
        \caption{Surface plot of MSE between \( \widehat{Z} \) and \( Z \).}
    \end{subfigure}
    \hfill
    \begin{subfigure}[t]{0.32\textwidth}
        \centering
        \includegraphics[width=0.9\linewidth]{ 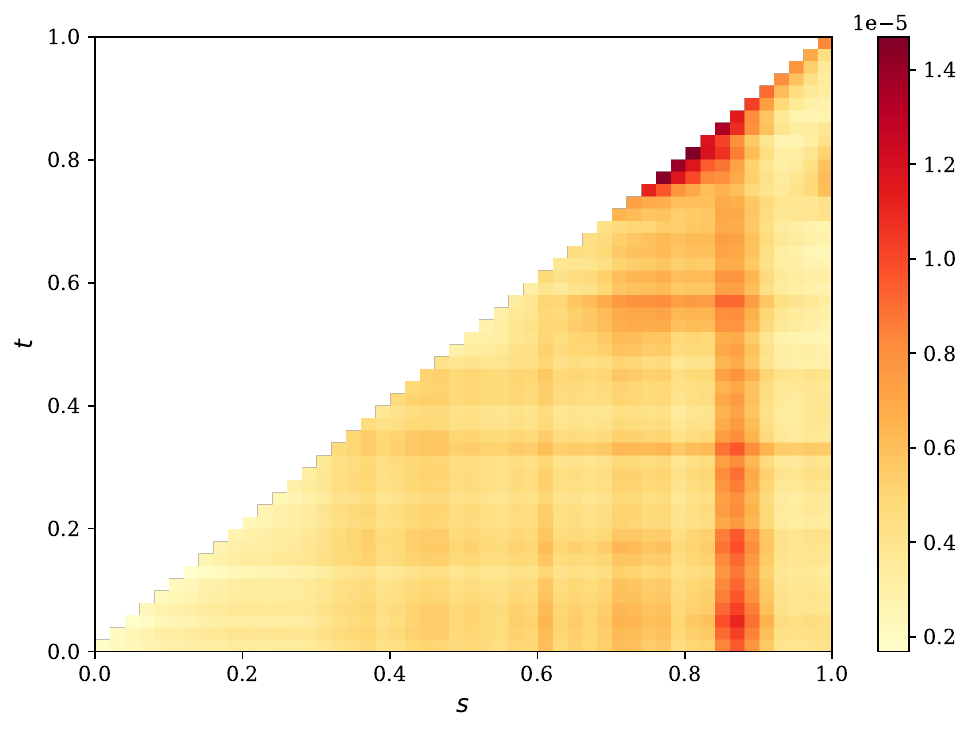}
        \caption{Heatmap of MSE between \( \widehat{Z} \) and \( Z \).}
    \end{subfigure}
    \caption{Visualization of estimation errors.}
    \label{fig:error_visualization2}
\end{figure}

\subsection{Recursive Utility with Nonlinear Wealth Effects}\label{ex_nonlinear}

In this example, we extend the recursive utility framework to capture nonlinear wealth effects and state-dependent risk attitudes in a multi-asset economy. Let the forward state variable represent the evolution of $d$ risky assets modeled as arithmetic Brownian motions:
\begin{equation*} 
X(t) = x_0 + \mu t + \sigma B(t) \in \mathbb{R}^d,
\end{equation*}
where $\mu \in \mathbb{R}^d$ is the drift vector, $\sigma \in \mathbb{R}^{d \times d}$ is the diagonal volatility matrix, and $(B(t))_{t \ge 0} \in \mathbb{R}^d$ is a standard $d$-dimensional Brownian motion.

We consider the following BSVIE:
\begin{equation}
\begin{aligned}
Y(t) &= t \sin(k \sum_{i=1}^d X^i(T)) 
+ \int_t^T \left( 
    \frac{t}{2} \sum_{i=1}^d \sin( X^i(s)) \| \sigma \|^2 
    - \mu^\top \sigma^{-1} Z(t,s)
\right) ds 
- \int_t^T Z(t,s) \, dB(s)
\end{aligned}
\end{equation}
whose solution is, by direct calculation \cite{andersson2025deep}, 
\begin{equation}
Y(t) = t \sin \left( \sum_{i=1}^d  X^i(t)\right), 
\qquad 
Z(t,s) = t \cos \left(\sum_{i=1}^d  X^i(s) \right) \, \sigma \mathbf{1}_d
\label{eq:sol_nonlinear}
\end{equation}
This BSVIE models an agent with cyclical preferences over total portfolio wealth: the terminal condition captures bounded sensitivity to aggregate wealth, where the agent's valuation oscillates with wealth. The driver introduces asset-specific nonlinear rewards and adjusts for exposure to market risk. 

\begin{remark}
    This example satisfies Assumptions \eqref{ass:scheme_bsig} and \eqref{ass:scheme_gf} as the forward process has constant coefficients, the terminal term is square-integrable and Lipschitz in the state variable, and the driver is smooth, Lipschitz in the state 
variable, independent of the remaining arguments, and H\"older-continuous in time. With $\sigma$ invertible, all required regularity conditions are fulfilled, and the 
theoretical results on well-posedness and numerical convergence apply to this setting.
\end{remark}

\subsubsection{Numerical Implementation}
We now present the numerical results obtained by applying to this example the neural network-based algorithm described in Section \ref{sec:numerical_implementation}. We work in dimension $d=5$, consider a time horizon \( T = 1 \), drift coefficient \( \mu = (0.07,\, 0.085,\, 0.1,\, 0.115,\, 0.13) \), volatility \( \sigma = \text{diag}(0.24,\, 0.27,\, 0.3,\, 0.33,\, 0.36) \), initial condition \( X_0 = 1.0 \).
The neural network's architecture and hyperparameters are the same as described in Section \ref{sec:algo}. The training follows Algorithm \ref{algo:BSVIE}.
The forward geometric Brownian motion is generated using the Euler-Maruyama scheme. The backward processes \( Y(t) \), \( Z(t,s) \), are learned simultaneously through the neural representation, trained over multiple trajectories to ensure statistical robustness. The network was trained on a GPU (AMD Instinct MI250X), and the typical training time for the presented example was approximately $8$ minutes and $34$ seconds.\\
Below, we report the key outcomes of the simulation. Figure \ref{fig:Y2_1b} compares the  learned solution \(\widehat{Y}_i, i = 0, \ldots, N\) with the analytical expression of \(Y\) derived in Equation \eqref{eq:sol_nonlinear}, evaluated at discrete time steps \(t_i\). Figure \ref{fig:surf2_1b} shows a comparison between the first component of the learned kernel \(\widehat{Z}(t_i, t_j)\) and the reference solution \(Z(t_i, t_j)\), both evaluated over the domain \(0 \leq i \leq j < N\). Figure \ref{fig:fixedk_1b} further compares the first component of the sequences \((\widehat{Z}_{i,j})_{i=0}^j\) with their reference counterparts \((Z(t_i, t_j))_{i=0}^j\),  illustrating how the learned kernel varies along the first time axis for fixed values of $s=t_j$.  Similarly, \ref{fig:fixedn_1b} shows a comparison between the first component of \( (\widehat{Z}_{i,j})_{j=i}^{N-1}\) and \( (Z(t_i,t_j))_{j=i}^{N-1}\) for selected values of $i$, showing slices of the both kernels for fixed values of $t = t_i$.  \\

Figure \ref{fig:loss_subplots2_1b} illustrates the evolution of the training loss across epochs for selected time steps. Table \ref{tab:err2_1b} reports the approximation errors, evaluated according to the metrics defined in Equations \eqref{eq:err} and \eqref{eq:rel_err}. To complement these results, Figure \ref{fig:error_visualization2_1b} displays the temporal evolution of the mean squared errors for both the \( Y \) and \( Z \) components.
\begin{table}[h!]
\centering
\begin{tabular}{|c|c|c|}
\hline
 & $L^2$ Error & $L^2$ Relative Error \\
\hline
$Y$ & $6.06 \times 10^{-5}$ & $3.22 \times 10^{-4}$ \\
$Z$ & $1.47 \times 10^{-5}$ & $2.16 \times 10^{-3}$ \\
\hline
\end{tabular}
\caption{$L^2$ Error and Relative Error for $Y$ and $Z$}
\label{tab:err2_1b}
\end{table}

\begin{figure}[H]
\centering
\includegraphics[width=0.7\textwidth]{ 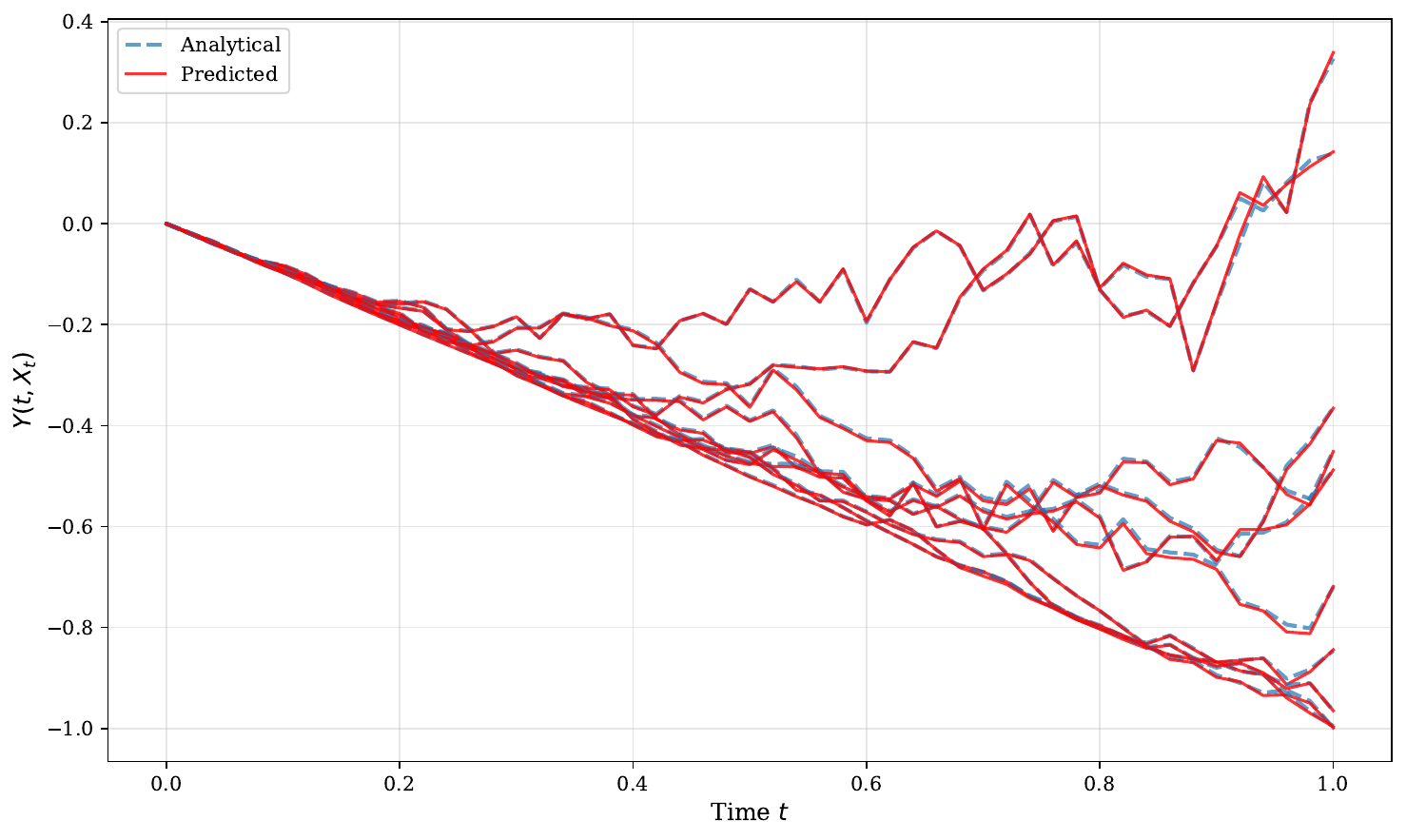}
\caption{Comparison of the learned \(Y(t)\) (red) with the analytical solution (dashed blue).}
\label{fig:Y2_1b}
\end{figure}

\begin{figure}[H]
    \centering
    \begin{subfigure}[t]{0.32\textwidth}
        \includegraphics[width=0.9\textwidth]{ 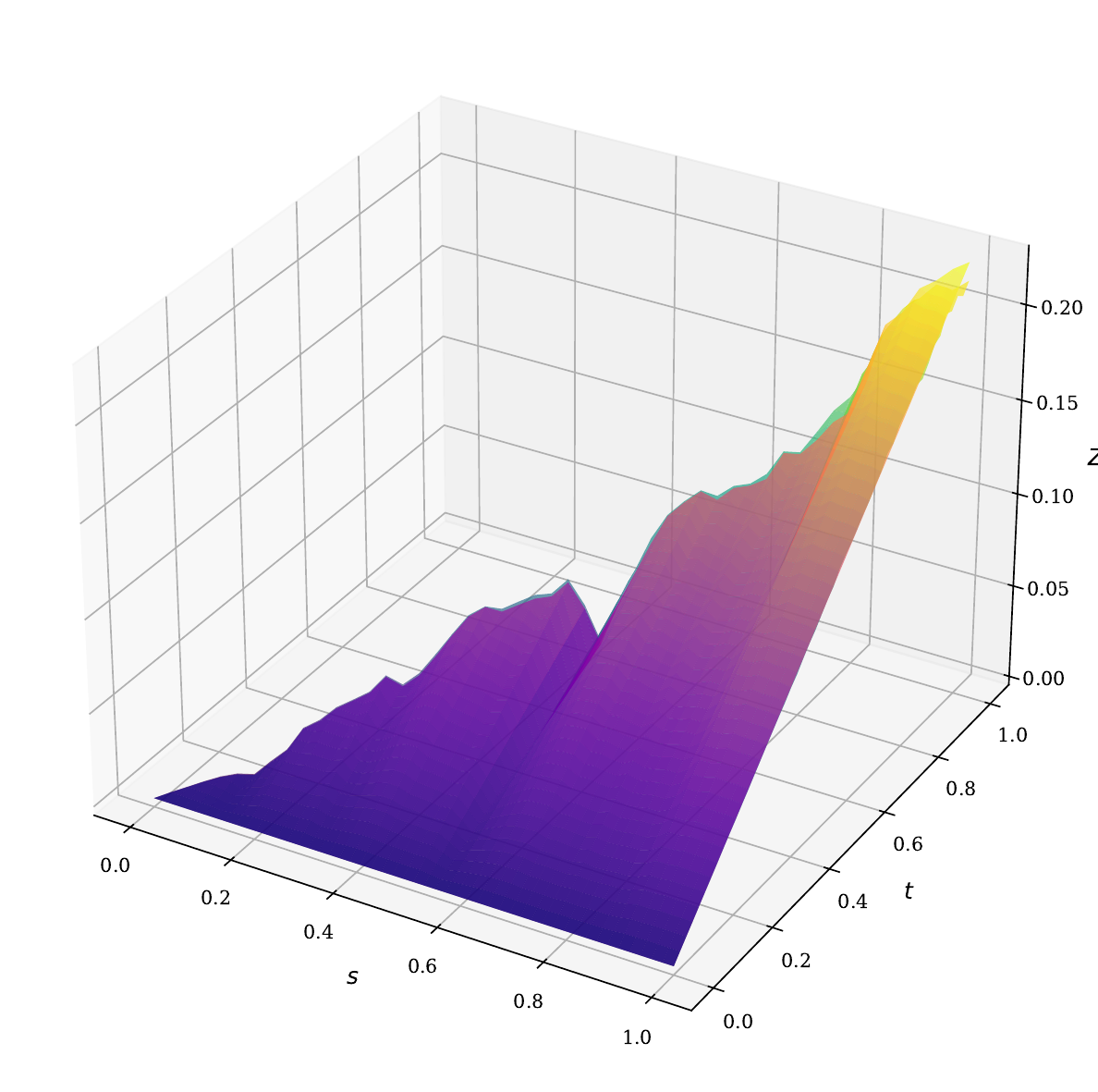}
        \caption{Surface plot of the learned kernel \( Z^1_{i,j} \), overlapped with the reference surface \( Z^1(t,s) \) on the evaluation grid.}
        \label{fig:surf2_1b}
    \end{subfigure}
    \hfill
    \begin{subfigure}[t]{0.32\textwidth}
        \includegraphics[width=\textwidth]{ 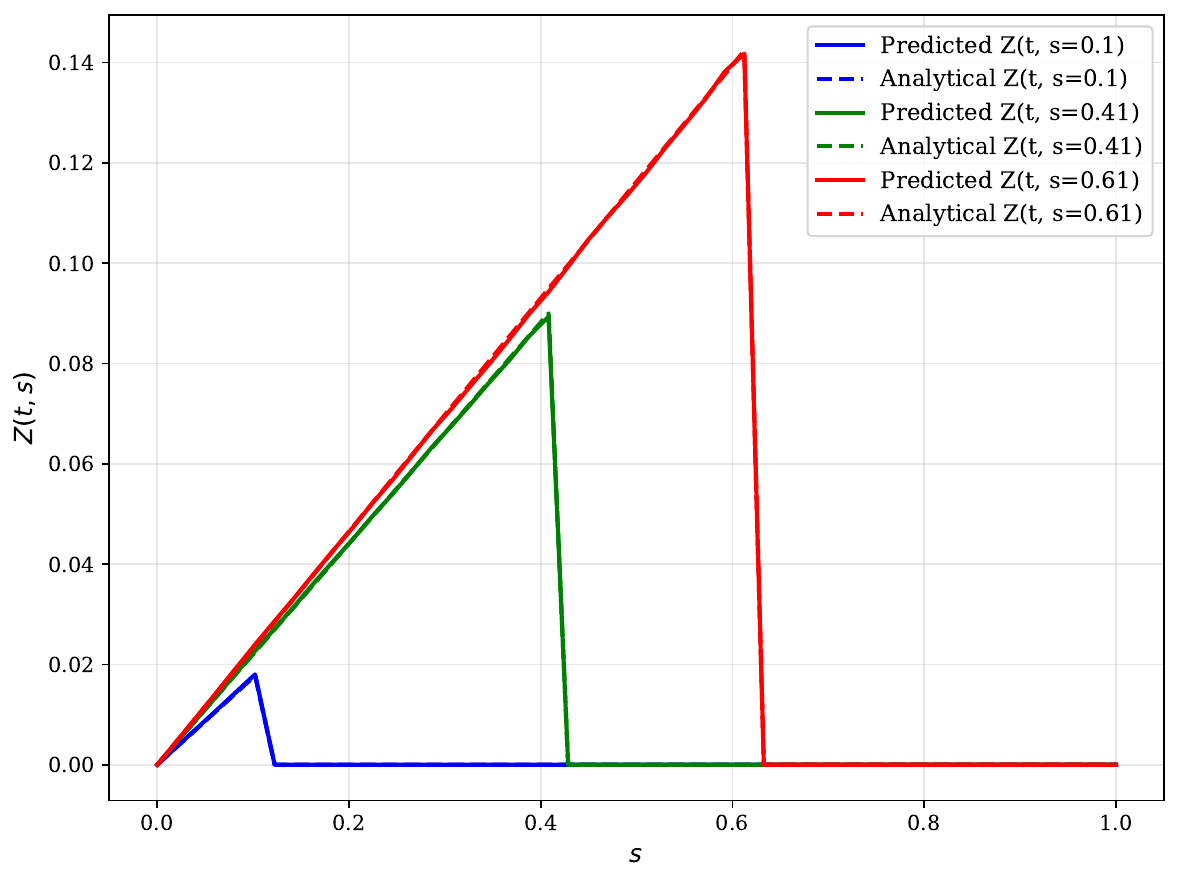}
        \caption{Comparison of $(\widehat{Z}^1_{i,j})_{i=0}^j$  with the corresponding values of $(Z^1(t_i,t_j))_{i=0}^j$ for selected values of $j$.}
        \label{fig:fixedk_1b}
    \end{subfigure}
    \hfill
    \begin{subfigure}[t]{0.32\textwidth}
        \includegraphics[width=\textwidth]{ 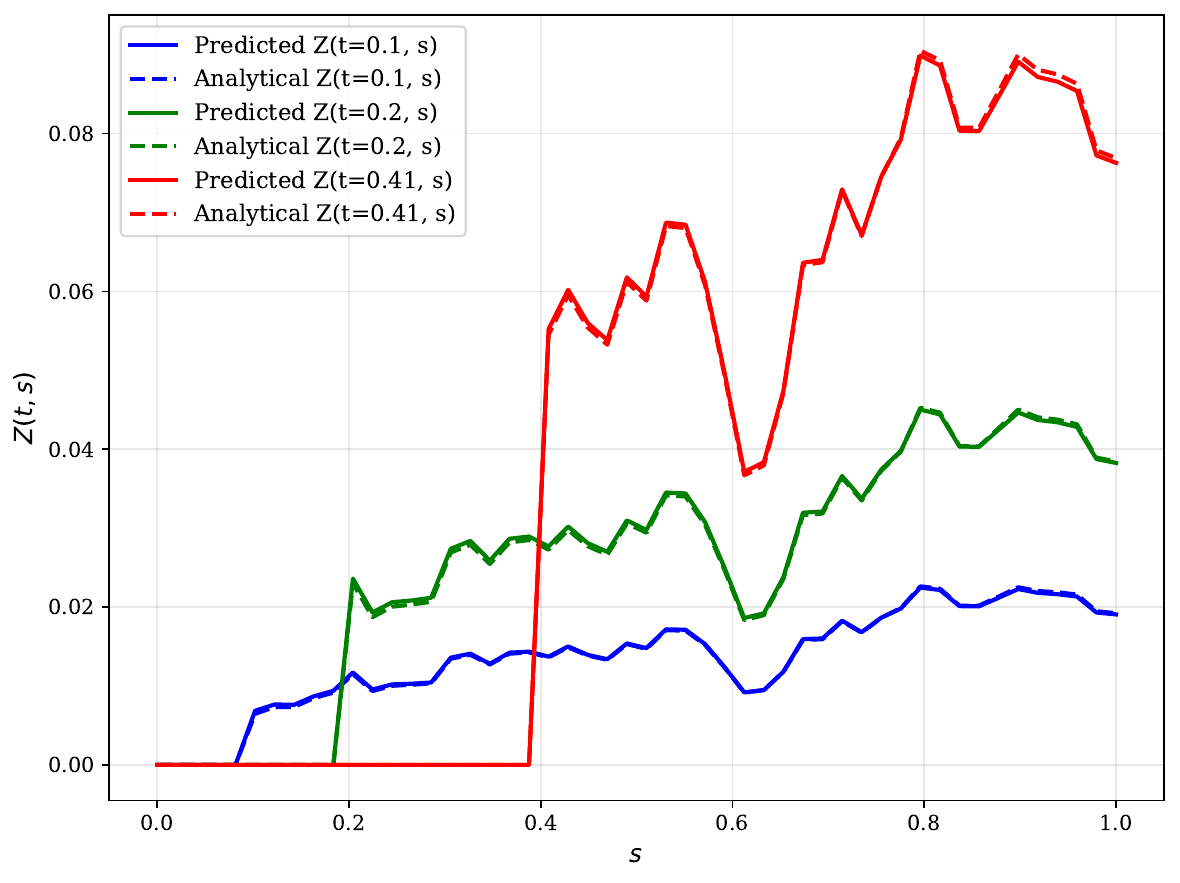}
        \caption{Comparison of $(\widehat{Z}^1_{i,j})_{j=i}^{N-1}$  with the corresponding values of $(Z(t_i,t_j))_{j=i}^{N-1}$ for selected values of $i$.}
        \label{fig:fixedn_1b}
    \end{subfigure}
    \caption{Visualization of the structure of the first component of the learned kernel \( Z_{i,j} \) and its continuous counterpart \( Z(t,s) \).}
    \label{fig:Z_plots2_1b}
\end{figure}

\begin{figure}[H]
    \centering
    \begin{adjustbox}{max width=\textwidth}
    \begin{subfigure}[t]{0.32\textwidth}
        \centering
      \includegraphics[width=\linewidth]{ 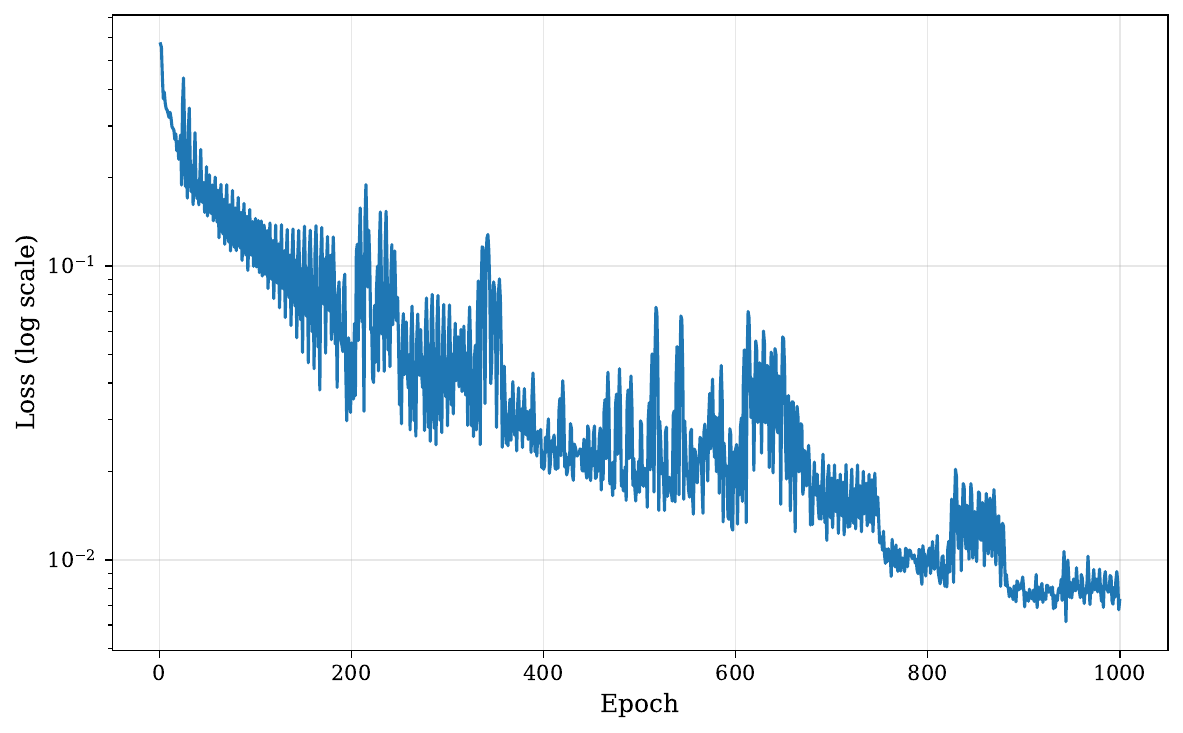}
        \caption{Loss at time step \( 50 \).}
    \end{subfigure}
    \hspace{0.01\textwidth}
    \begin{subfigure}[t]{0.32\textwidth}
        \centering
        \includegraphics[width=\linewidth]{ 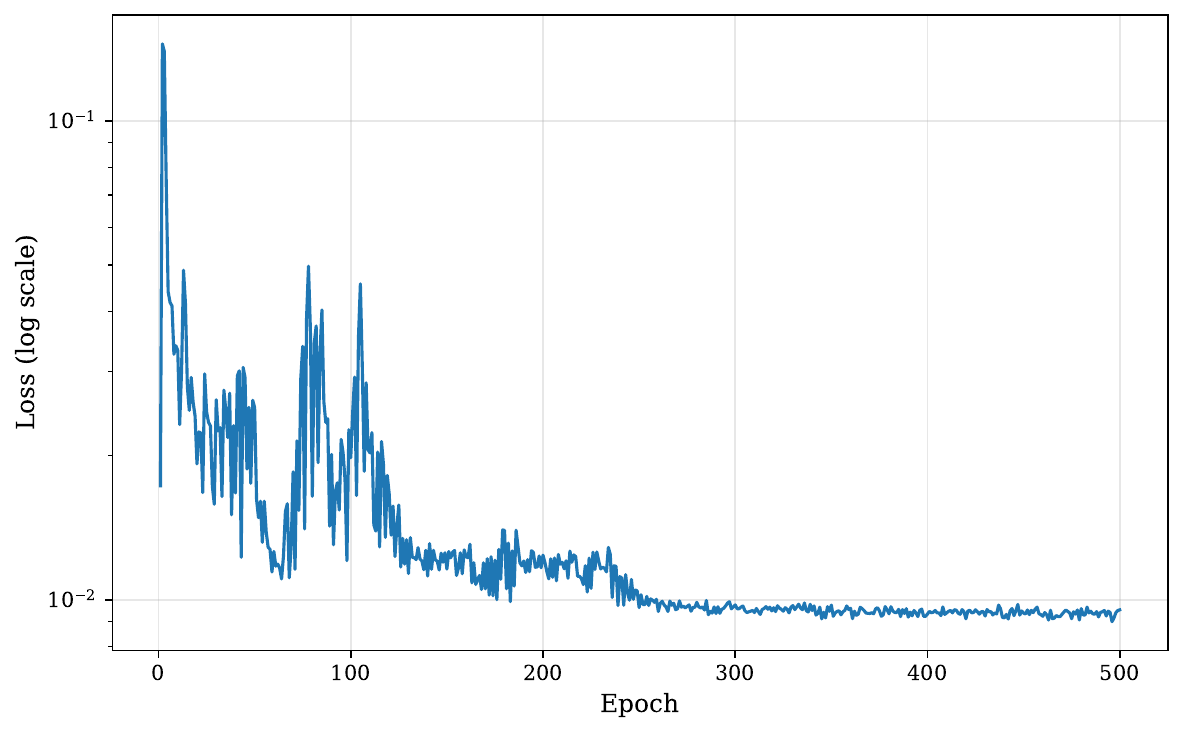}
        \caption{Loss at time step \( 25 \).}
    \end{subfigure}
    \hspace{0.01\textwidth}
    \begin{subfigure}[t]{0.32\textwidth}
        \centering
        \includegraphics[width=\linewidth]{ 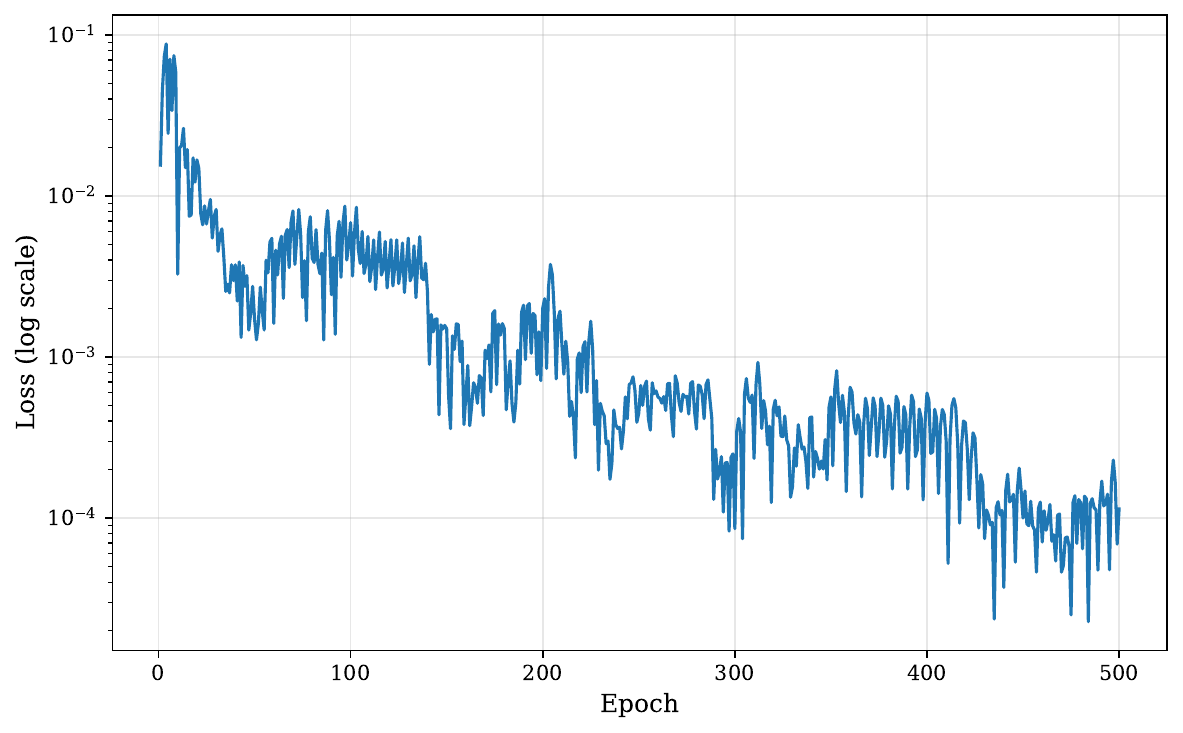}
        \caption{Loss at time step \( 0 \).}
    \end{subfigure}
     \end{adjustbox}
    \caption{Algorithm loss across iterations evaluated at selected time steps.}
    \label{fig:loss_subplots2_1b}

\end{figure}

\begin{figure}[H]
    \centering
    \begin{subfigure}[t]{0.32\textwidth}
        \centering
        \includegraphics[width=\linewidth]{ 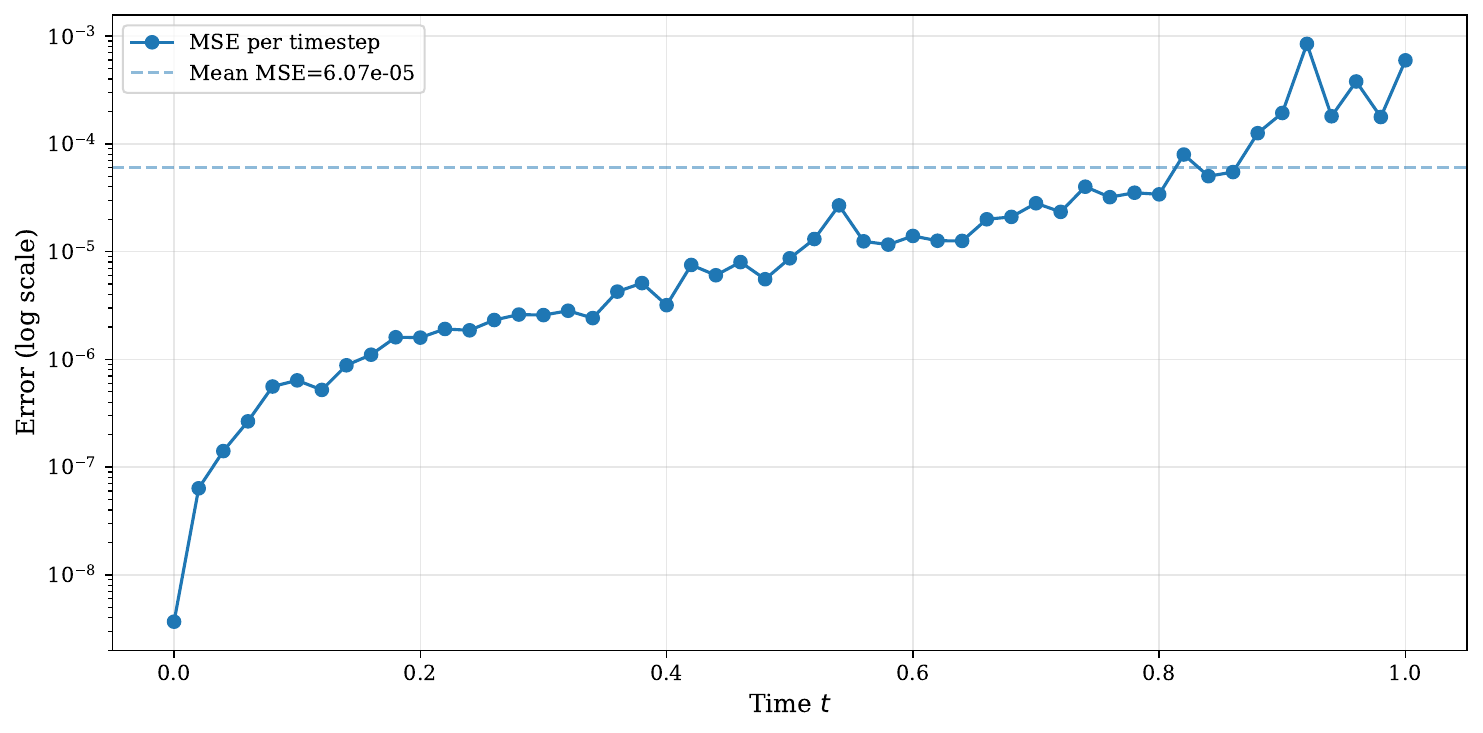}
        \caption{MSE of \( \widehat{Y} \) vs. \( Y \) over time.}
    \end{subfigure}
    \hfill
    \begin{subfigure}[t]{0.32\textwidth}
        \centering
        \includegraphics[width=0.9\linewidth]{ 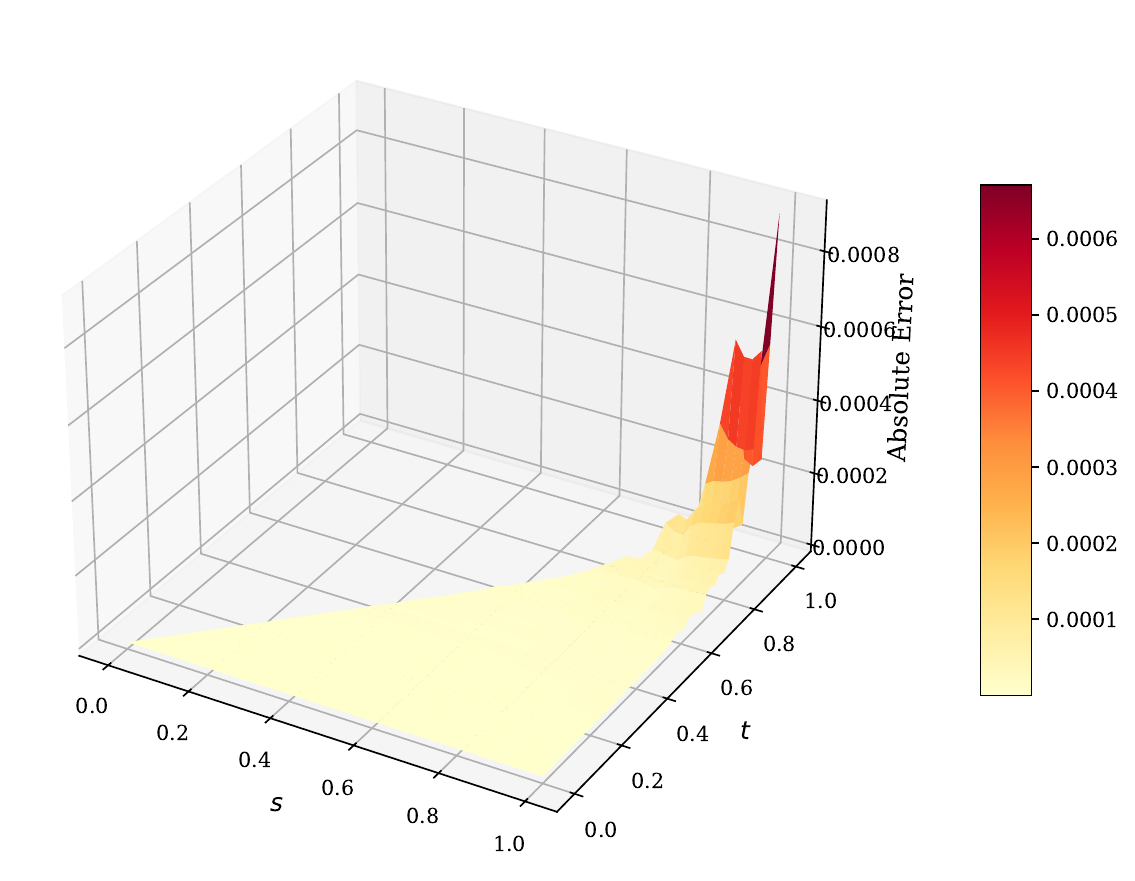}
        \caption{Surface plot of MSE between \( \widehat{Z} \) and \( Z \).}
    \end{subfigure}
    \hfill
    \begin{subfigure}[t]{0.32\textwidth}
        \centering
        \includegraphics[width=0.9\linewidth]{ 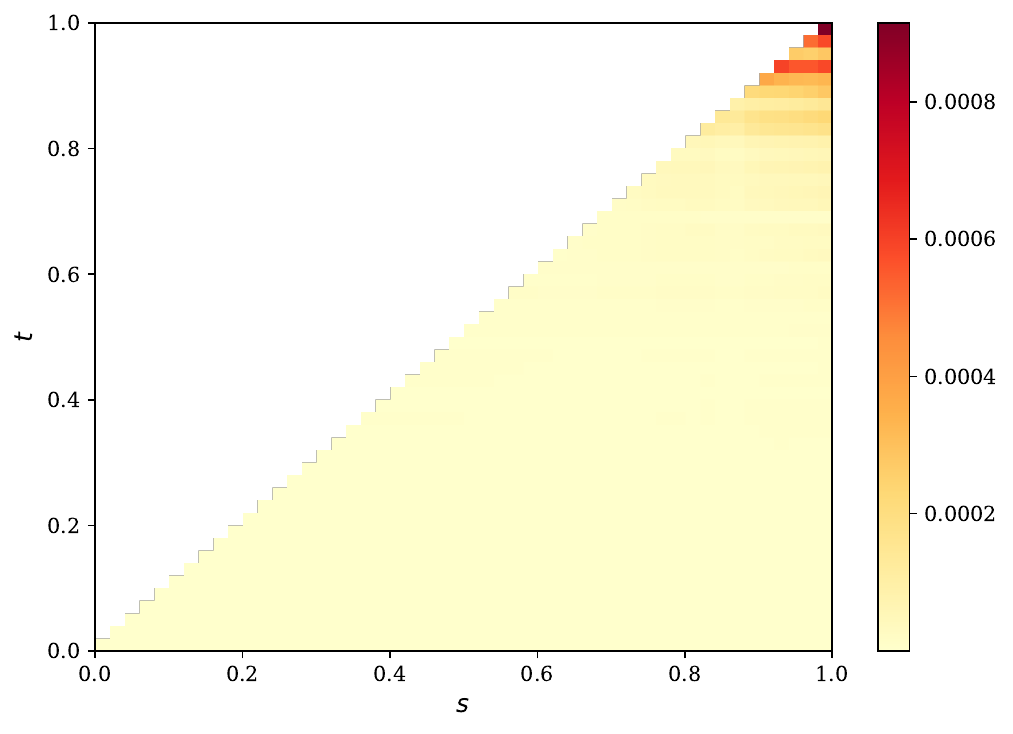}
        \caption{Heatmap of MSE between \( \widehat{Z} \) and \( Z \).}
    \end{subfigure}
    \caption{Visualization of estimation errors.}
    \label{fig:error_visualization2_1b}
\end{figure}

\subsubsection{Sensitivity Analysis}
The architectural choices presented in Section \ref{sec:scheme}, and used consistently across all numerical examples, were determined through systematic hyperparameter tuning. We explored variations in network width (10-200 neurons), depth (2-4 hidden layers), activation functions (ReLU, sigmoid, tanh), optimizers (Adam, AdamW), and learning rate schedulers (exponential decay, cosine annealing, reduce-on-plateau). We finally selected the minimal architecture that achieved satisfactory accuracy while maintaining computational efficiency. \\
A comprehensive sensitivity analysis over all architectural parameters would be prohibitively expensive, given the large number of parameters and their possible combinations. 
Instead, we perform a targeted evaluation on the nonlinear example from Section \ref{ex_nonlinear}, and we present a focused sensitivity analysis to highlight the robustness of our architectural choices.\\

We systematically varied the sizes of the hidden layers in the networks, taking 
$h_Y \in \{20,40,60,80\}$ while maintaining $h_Z = 2h_Y$, and the batch size 
$M \in \{2^{11}, 2^{12}, 2^{13}\}$. 
Figure \ref{fig:sensitivity_analysis} presents the resulting performance in terms of accuracy 
(measured by the MSE in $Y$ and $Z$) and computational time 
across these configurations.\\
This reveals that the errors can vary by roughly two orders of magnitude and that the network width for larger batch sizes has a non-monotonic effect on accuracy: for batch size $2^{13}$ the error in $Y$ initially decreases 
as the width increases from $h_Y=20$ to $h_Y=40$, 
but then rises at $h_Y=60$, suggesting potential overfitting 
or optimization challenges in wider networks. We note that such behavior may also stem from interactions with other hyperparameters (e.g., learning rate, weight decay) that were held fixed in this analysis, highlighting the complex interplay between architectural and training choices.\\
We notice that training time increases moderately with network width and more substantially with batch size, 
illustrating the computational trade-offs involved.\\
Regarding batch size selection, while $M=2^{13}$ achieved a slightly lower MSE for $Y$, the difference is negligible as both batch sizes yield errors of the same order of magnitude. Moreover, $M=2^{12}$ provides approximately $10\%$ reduction in training time. These considerations motivated our choice of $M=2^{12}$ as the optimal trade-off for all subsequent experiments.

\begin{figure}[H]
    \centering
    \includegraphics[width=0.9\linewidth]{ 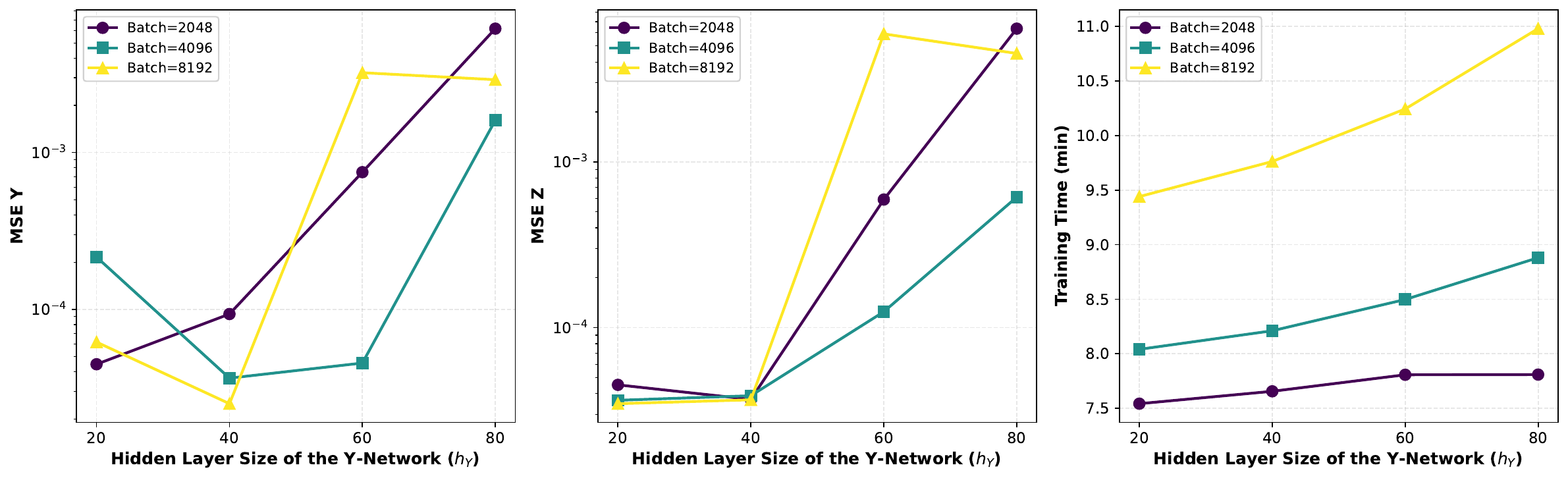}
    \caption{Performance sensitivity to network width $h_Y$, width ratio $h_Z = 2h_Y$, 
    and batch size $M$. 
    We represent MSE in $Y$ and $Z$ as well as training time for each configuration.}
    \label{fig:sensitivity_analysis}
\end{figure}

\section{Numerical Solution of RBSVIEs}
\label{sec:reflected}
We now extend the deep learning-based scheme introduced in Section \ref{sec:numerics} to address RBSVIEs. Such equations arise when the agent's utility is constrained to remain above a floor process, as motivated by regret aversion models. The setting we consider follows the formulation in \cite{agram2021class}, which proves existence and uniqueness for a class of continuous RBSVIEs driven by Brownian motion.

\vspace{1em}
\textbf{Problem Setting.}
We consider the following RBSVIE:
\begin{equation*}
\begin{cases}
Y(t) = g(t, X(T)) + \int_t^T f(t, s, X(s), Y(s), Z(t,s)) \, ds + \int_t^T K(t,ds) - \int_t^T Z(t,s) \, dB(s), \\\\
Y(t) \ge L(t).
\end{cases}
\end{equation*}

Here, the reflection process \(K(t,\cdot)\) is non-decreasing in \(s\), continuous in \(t\), and increases only when the constraint \(Y(t) = L(t)\) is binding (Skorokhod flatness condition that is: for all \( 0 \leq \alpha < \beta \leq T \),
    \[
        K(t, \alpha) = K(t, \beta) \quad \text{whenever } Y(u) > L(u) \text{ for all } u \in [\alpha, \beta], \quad \mathbb{P}\text{-a.s.}\] 
The process \(L(t)\) denotes a continuous lower barrier (the floor), and the triple \((Y, Z, K)\) constitutes the solution.

\vspace{1em}
\textbf{Numerical Scheme.}
We follow the structure of the deep BSDE solver, incorporating a projection step to enforce the reflection condition. This mirrors the RDBDP approach of \cite{hure2020deep} for variational inequalities. Specifically, for each \(i \in \{0, \ldots, N\}\), we construct neural networks \(\mathcal{Y}^i(\cdot; \xi)\) and \(\mathcal{Z}^{i}(\cdot, \cdot; \eta)\) to approximate \(Y_i\) and \(Z_i(t_i, t_j)\). The projection step enforces the constraint \(Y_i \ge L(t_i)\) through
\[
\widehat{Y}_i^k := \max\big\{ \mathcal{Y}^i(t_i, X_i^k; \xi),\, L(t_i) \big\}.
\]

\vspace{0.5em}
The full algorithm reads as follows:

\begin{algorithm}[H]
\caption{DeepRBSVIE: RBSVIE Solver}
\label{algo:RBSVIE}
\begin{algorithmic}[1]
\For{$i = N$ to $0$}
    \For{each epoch}
        \For{$k = 1$ to $M$}
            \State Simulate trajectory \((X_i^k)_{i=0}^N\) via Euler-Maruyama scheme.
            \State Predict \(\widetilde{Y}_i^j := \mathcal{Y}^i(t_i, X_i^j; \xi_i)\)
            \For{$j = i+1$ to $N$}
                \State Predict \(\widehat{Z}_i^j(t_i, t_j) := \mathcal{Z}^{i}(t_i, t_j, X_i^j, X_j^j; \eta_i)\)
            \EndFor
            \State Compute loss:
            \small
         \[
           \begin{aligned}
            \ell_i^k = \Bigg| \widetilde{Y}_i^k - \Bigg( & g(t_i, X_i^k, X_N^k) + \sum_{j=i}^{N-1} f(t_i, t_j, X_i^k, X_j^k, \widehat{Y}_j^k, \widehat{Z}_i^k(t_i, t_j)) \Delta t \\
            & - \sum_{j=i+1}^{N-1} \widehat{Z}_i^k(t_i, t_j) \Delta B_j^k \Bigg) \Bigg|^2
          \end{aligned}
        \]
        \EndFor
        \State Minimize loss: \(\ell_i = \frac{1}{M} \sum_{k=1}^M \ell_i^k\) using gradient descent.
        \State Update: \(\widehat{Y}_i^k := \max\{\widetilde{Y}_i^k, L(t_i)\}\)
    \EndFor
\EndFor
\State \textbf{Return} \(\{\widehat{Y}_i^k, \widehat{Z}_i^k(t_i, \cdot)\}_{i}\)
\end{algorithmic}
\end{algorithm}

\medskip
The projection step \(\widehat{Y}_i^k = \max\{ \widetilde{Y}_i^k, L(t_i) \}\) mimics the Skorokhod reflection in the discrete time setting. The convergence of such methods relies on approximation properties of the networks, stability of the discrete scheme, and control of the training error.

\subsection{Application: Recursive Utility with Regret Aversion Floors}

We consider an agent evaluating a terminal payoff $\phi$ using time-inconsistent preferences and being subject to \emph{regret aversion}. Specifically, the agent exhibits regret aversion, fearing that their utility process falls below a dynamically meaningful benchmark or "floor."  This leads to a RBSVIE, capturing both the memory effects of time inconsistent discounting and the behavioral floor constraint.

In behavioral economics, regret aversion is typically modeled through a constraint on the utility process. The agent's perceived utility shoud not fall below a floor $L(t)$, which might represent historical reference point, a guaranteed satisfaction threshold etc. To ensure this condition, we impose a reflection mechanism that minimally adjusts the process when the constraint becomes binding. \\

Let $\phi$ be a fixed $\mathcal{F}_T$-measurable terminal payoff. The agent evaluates this reward via a recursive process $Y(t)$ governed by the following RBSVIE:
\begin{equation*}
\begin{aligned}
Y(t) &= f(T - t)\, \phi %+ \int_t^T f(t,s)\, ds 
+ \int_t^T K(t, ds) - \int_t^T Z(t,s) \, dB(s), \quad t \in [0,T], \\
Y(t) &\geq L(t). \\
%\int_t^T (&Y(t) - L(t))\, dK(t,s) = 0.
\end{aligned}
\end{equation*}
Here the term $f(T - t)\, \phi$ represents a time inconsistent terminal reward, where $f \colon \bR^+ \to \bR^+$ is a time inconsistent discounting kernel, e.g.  $f(x) = \frac{1}{1 + \delta x}$ representing hyperbolic discounting. %The integral $\int_t^T f(t,s)\, ds$ captures a nonlocal deterministic contribution reflecting how future time intervals are weighted from the perspective of the present time $t$, even in the absence of intermediate payoffs. 
The lower barrier $L(t)$ serves as the regret-aversion floor, representing the agent's minimum acceptable utility level, which may be a fixed threshold, a moving average of past valuations, or a proportion of expected reward. To enforce the constraint $Y(t) \geq L(t)$, the model includes a two parameter reflection field $K(t,s)$, which is non-decreasing in $s$ for each fixed $t$ and acts only when the constraint is binding. The stochastic integral $\int_t^T Z(t,s)\, dB(s)$ accounts for volatility in the backward component.

The process $Y(t)$ represents the agent's valuation of the delayed payoff $\phi$, computed recursively in a time-inconsistent manner due to the Volterra structure. To prevent violations of the floor $L(t)$, a reflection mechanism is introduced via $K(t,s)$, which ensures that the process $Y(t)$ is pushed upward only when necessary. The flatness condition ensures that $K(t,s)$ increases only on sets where $Y(t) = L(t)$, ensuring minimal intervention.

\begin{example}[Option Payoff with Behavioral Floor] \label{ex}
Consider an underlying forward process following $d-$dimensional geometric Brownian motion 
\[
dX(t) = \mu X(t) \, dt + \sigma X(t) \, dB(t), \quad X(0) = x_0
\]
with given expected return \( \mu \in \bR^d \) and diagonal volatility matrix \( \sigma \in \bR^{d \times d} \).
Suppose $f(x) = \frac{1}{1 + x}$, $\phi = \max \left ( \frac{1}{d}\sum_{i=1}^d X^i(T) - K, 0\right)$, and $L(t) = 0.05$. In this case, $\phi$ represents the payoff of a European call option on the average of the $d$ underlying assets (i.e., an arithmetic basket call), and the agent evaluates it using a hyperbolically discounted, path-dependent utility process. The lower bound $L(t)$ represents a regret-aversion floor, insisting that the perceived utility $Y(t)$ remain at least $5\%$ of a nominal utility level. The reflection field $K(t,s)$ acts to maintain this constraint over the time interval $[t,T]$ while preserving the structure of the valuation process.
\end{example}

\subsection{Numerical Implementation}

For the numerical simulation of Example \ref{ex}, we consider a time horizon \( T = 1 \) and $d=5$. We set $\mu = (0.07, 0.085, 0.1, 0.115, 0.13), \sigma= \text{diag}(0.16,0.18,0.2,0.22,0.24)$ and initial condition \( X(0) = 1.0 \). The strike price of the option is fixed at \( K = 1.0 \), and the regret-aversion floor is set as a constant \( L(t) = 0.05 \). 

To approximate the solution of the RBSVIE, we employ a neural network-based method trained using the same architecture and hyperparameters described in Section \ref{sec:numerical_implementation}, where we detailed the general learning setup for BSVIEs. The training follows Algorithm \ref{algo:RBSVIE}, which handles the presence of the reflection term.

The simulation is carried out on a uniform time grid with \( N = 50 \) time steps. The forward geometric Brownian motion is generated using the Euler-Maruyama scheme. The backward processes \( Y(t) \), \( Z(t,s) \), and the action of the penalization field approximating the reflection \( K(t,s) \) are learned simultaneously through the neural representation, trained over multiple trajectories.

Below, we report the key outcomes of the simulation. In Figure~\ref{fig:Yrefl}, we show three sample paths of the approximated recursive utility under regret aversion. The effect of the floor constraint is evident, as the process is pushed upward whenever it approaches the prescribed threshold. In Figure~\ref{fig:Zrefl}, we present the corresponding first component of the 
$Z$ surfaces for the same three samples across the discretized 
$Z(t,s)$ domain. Finally, in Figure \ref{fig:loss_subplots_refl} a plot of the loss function across training epochs for different time steps is presented, demonstrating convergence and stability of the learning scheme.

\begin{figure}[htbp]
    \centering
\includegraphics[width=0.6\textwidth]{ 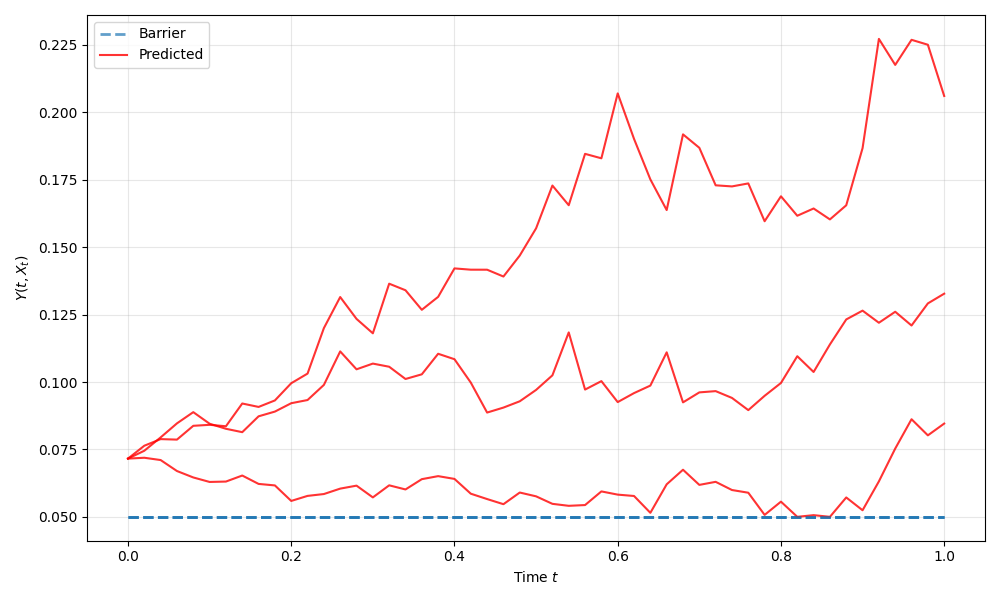}
    \caption{Three sample plots of the learned \(\widehat{Y}\) (continuous line) and reflecting barrier (dashed line).}
    \label{fig:Yrefl}
\end{figure}
\begin{figure}[htbp]
    \centering
    \includegraphics[width=0.8\textwidth]{ 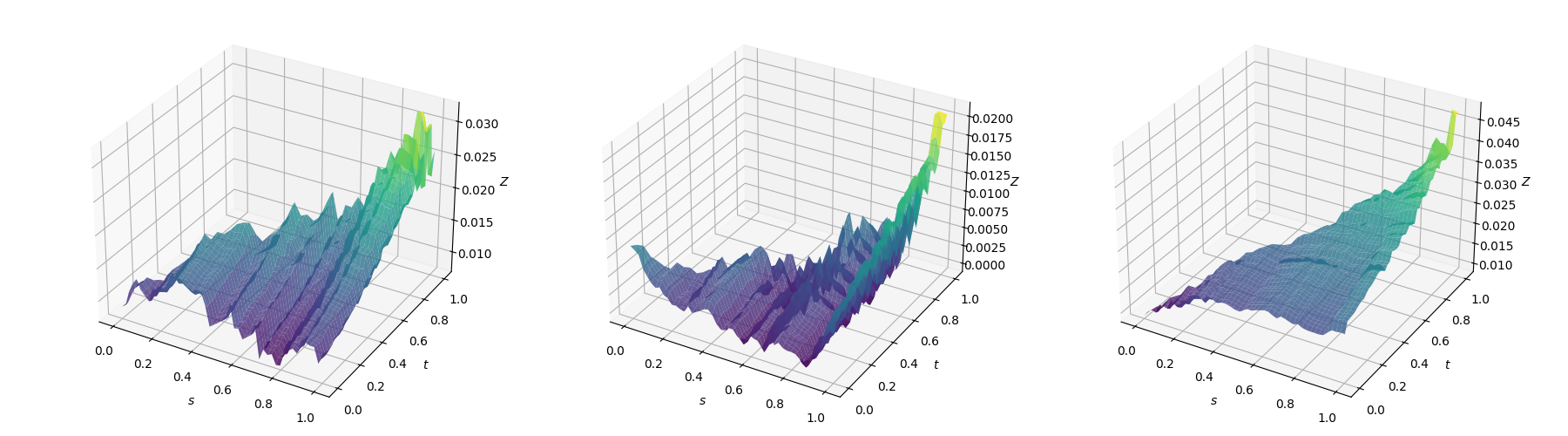}
    \caption{Three sample plots of the first component of the learned kernel \(Z_{n,k}\).}
    \label{fig:Zrefl}
\end{figure}

\begin{figure}[H]
    \centering
    \begin{adjustbox}{max width=\textwidth}
    \begin{subfigure}[t]{0.32\textwidth}
        \centering
      \includegraphics[width=\linewidth]{ 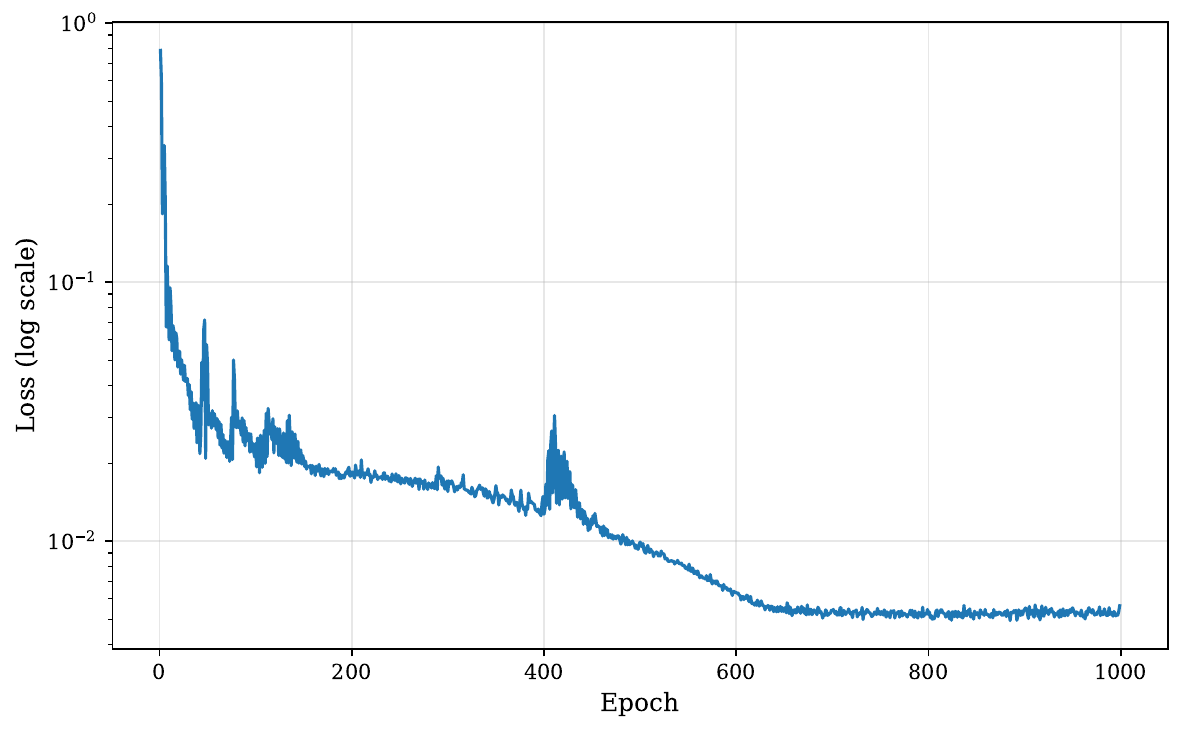}
        \caption{Loss at time step \( 50 \).}
    \end{subfigure}
    \hspace{0.01\textwidth}
    \begin{subfigure}[t]{0.32\textwidth}
        \centering
        \includegraphics[width=\linewidth]{ 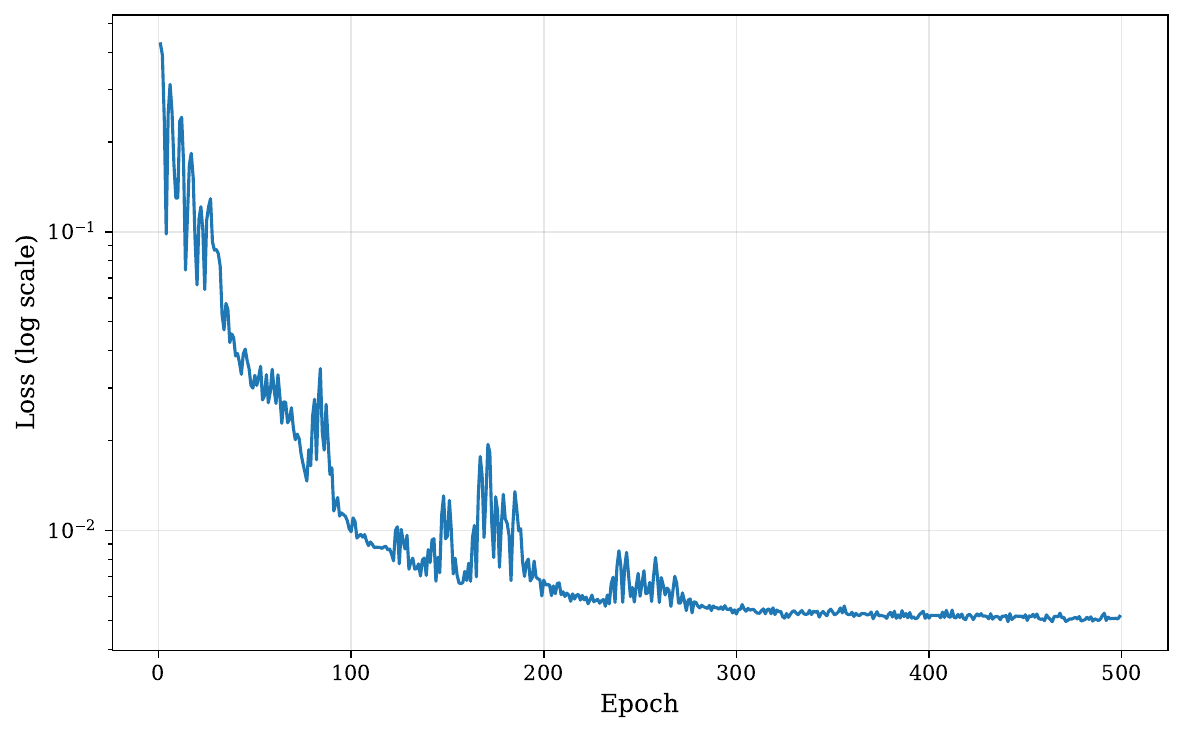}
        \caption{Loss at time step \( 25 \).}
    \end{subfigure}
    \hspace{0.01\textwidth}
    \begin{subfigure}[t]{0.32\textwidth}
        \centering
        \includegraphics[width=\linewidth]{ 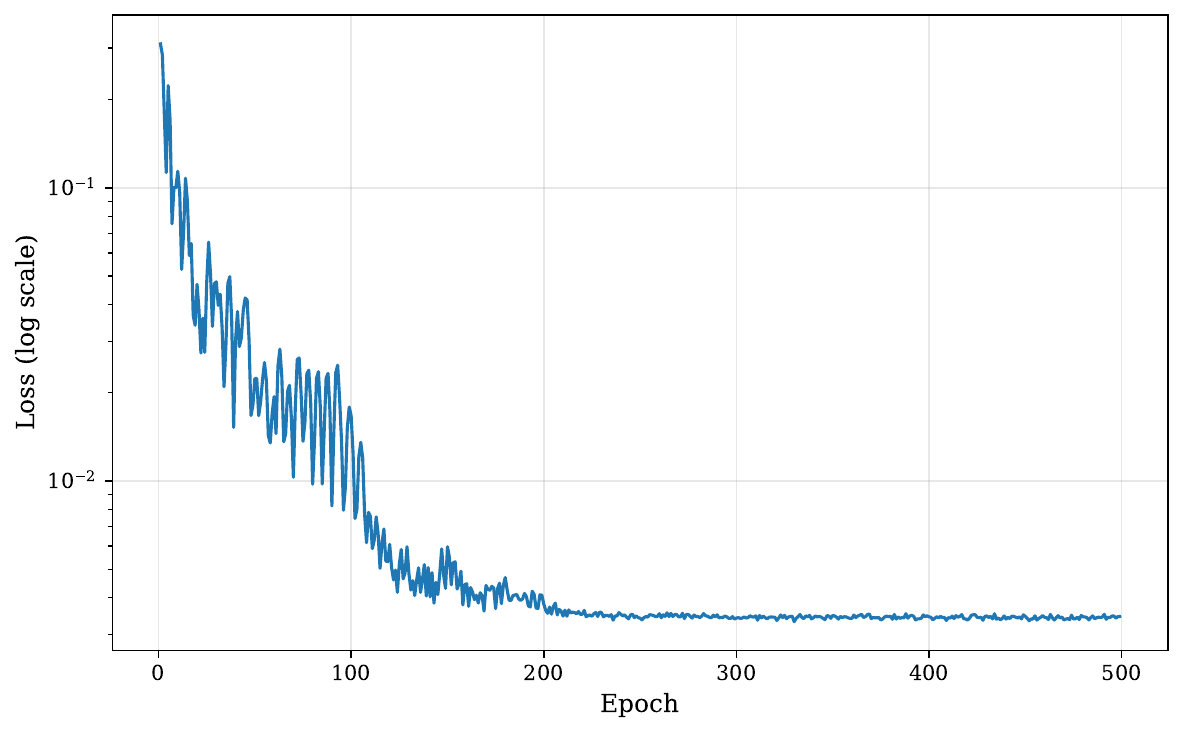}
        \caption{Loss at time step \( 0 \).}
    \end{subfigure}
     \end{adjustbox}
    \caption{Algorithm loss across iterations evaluated at selected time steps.}
    \label{fig:loss_subplots_refl}

\end{figure}

\paragraph{Acknowledgments.}
All simulations were performed on the Dardel HPC system. The computations were enabled by resources provided by the National Academic Infrastructure for Supercomputing in Sweden (NAISS).

\bibliographystyle{amsplain}
\bibliography{references}

\end{document}